\documentclass[12pt]{amsart}

\pdfoutput=1

\usepackage[text={420pt,660pt},centering]{geometry}

\usepackage{color}
\usepackage{esint,amssymb}
\usepackage{graphicx}
\usepackage{MnSymbol}
\usepackage{mathtools}
\usepackage[colorlinks=true, pdfstartview=FitV, linkcolor=blue, citecolor=blue, urlcolor=blue,pagebackref=false]{hyperref}
\usepackage{microtype}

\usepackage{bm}
\usepackage{scalerel} 
\usepackage{dsfont}

\definecolor{darkgreen}{rgb}{0,0.5,0}
\definecolor{darkred}{rgb}{0.9,0.1,0.1}

\newtheorem{proposition}{Proposition}
\newtheorem{theorem}[proposition]{Theorem}
\newtheorem{lemma}[proposition]{Lemma}
\newtheorem{corollary}[proposition]{Corollary}

\theoremstyle{definition}
\newtheorem{remark}[proposition]{Remark}

\newcommand{\cref}[1]{Corollary~\ref{c.#1}}

\numberwithin{equation}{section}
\numberwithin{proposition}{section}

\newcommand{\A}{\mathcal{A}}
\newcommand{\Ahom}{\overline{\A}}

\newcommand{\Z}{\mathbb{Z}}
\newcommand{\N}{\mathbb{N}}
\newcommand{\R}{\mathbb{R}}
\renewcommand{\S}{\mathbb{S}}
\newcommand{\E}{\mathbb{E}}
\renewcommand{\P}{\mathbb{P}}
\newcommand{\F}{\mathcal{F}}
\newcommand{\Zd}{\mathbb{Z}^d}
\newcommand{\Rd}{{\mathbb{R}^d}}

\newcommand{\ep}{\varepsilon}
\newcommand{\eps}{\varepsilon}

\newcommand{\J}{\mathcal{J}}
\renewcommand{\a}{\mathbf{a}}
\newcommand{\h}{\mathbf{h}}
\newcommand{\g}{\mathbf{g}}
\newcommand{\f}{\mathbf{f}}

\newcommand{\T}{\mathsf{\!t}}
\renewcommand{\subset}{\subseteq}

\renewcommand{\a}{\mathbf{a}}

\newcommand{\s}{\mathbf{s}}
\newcommand{\m}{\mathbf{m}}
\newcommand{\ahom}{{\overbracket[1pt][-1pt]{\a}}}  
\newcommand{\Abf}{{\mathbf{A}}}
\newcommand{\Abfh}{{\overbracket[1pt][-1pt]{{\mathbf{A}}}}}  

\renewcommand{\subset}{\subseteq}

\newcommand{\cu}{{\scaleobj{1.2	}{\square}}}
\newcommand{\cud}{{\scaleobj{1.2}{\boxdot}}}

\newcommand{\Lso}{L^2_{\mathrm{sol},0}}

\renewcommand{\fint}{\strokedint}

\newcommand{\Ll}{\left}
\newcommand{\Rr}{\right}

\DeclareMathOperator{\dist}{dist}

\DeclareMathOperator{\var}{var}

\newcommand{\X}{\mathcal{X}}

\newcommand{\ov}{\overline}
\renewcommand{\bar}{\overline}
\newcommand{\un}{\underline}
\renewcommand{\tilde}{\widetilde}
\newcommand{\td}{\widetilde}

\newcommand{\indc}{\mathds{1}}
\newcommand{\1}{\mathds{1}}

\newcommand{\mcl}{\mathcal}

\newcommand{\be}{\beta}

\renewcommand{\O}{\mathcal{O}}
\renewcommand{\S}{\mathcal{S}}

\renewcommand{\hat}{\widehat}

\newcommand{\C}{\mathcal{C}}

\newcommand{\sym}{\mathbf{s}}
\renewcommand{\skew}{\mathbf{m}}

\newcommand{\pa}{\mathrm{par}}
\newcommand{\per}{\#}

\begin{document}

\title[Stochastic homogenization of parabolic equations]{Quantitative stochastic homogenization and regularity theory of parabolic equations}

\begin{abstract}
We develop a quantitative theory of stochastic homogenization for linear, uniformly parabolic equations with coefficients depending on space and time. Inspired by recent works in the elliptic setting, our analysis is focused on certain subadditive quantities derived from a variational interpretation of parabolic equations. These subadditive quantities are intimately connected to spatial averages of the fluxes and gradients of solutions. We implement a renormalization-type scheme to obtain an algebraic rate for their convergence, which is essentially a quantification of the weak convergence of the gradients and fluxes of solutions to their homogenized limits. As a consequence, we obtain estimates of the homogenization error for the Cauchy-Dirichlet problem which are optimal in stochastic integrability. We also develop a higher regularity theory for solutions of the heterogeneous equation, including a uniform $C^{0,1}$-type estimate and a Liouville theorem of every finite order. 
\end{abstract}

\author[S. Armstrong]{Scott Armstrong}
\address[S. Armstrong]{Courant Institute of Mathematical Sciences, New York University, 251 Mercer St., New York, NY 10012}
\email{scotta@cims.nyu.edu}

\author[A. Bordas]{Alexandre Bordas}
\address[A. Bordas]{Ecole normale sup\'erieure de Lyon, 46 all\'ee d'Italie, 69007 Lyon, France}
 \email{alexandre.bordas@ens-lyon.fr}

\author[J.-C. Mourrat]{Jean-Christophe Mourrat}
\address[J.-C. Mourrat]{Ecole normale sup\'erieure de Lyon, CNRS, 46 all\'ee d'Italie, 69007 Lyon, France}
\email{jean-christophe.mourrat@ens-lyon.fr}

\keywords{stochastic homogenization, parabolic equation}
\subjclass[2010]{35B27, 35B45, 60K37, 60F05}
\date{\today}

\maketitle
\setcounter{tocdepth}{1}
\tableofcontents

\section{Introduction}

\subsection{Motivation and informal summary of results}
In this paper, we develop a quantitative theory of stochastic homogenization for linear, uniformly parabolic equations with coefficients depending on both the space and time variables. We consider equations of the form
\begin{equation}  
\label{e.lin.parab}
\partial_t u^\ep  - \nabla \cdot \left(   \a\left(\frac t{\ep^2},\frac x\ep \right) \nabla u^\ep\right) = 0 
\quad 
\mbox{in} \ I \times U,
\end{equation}
where $I \subset \R$ is an open interval, $U$ is a bounded Lipschitz domain of $\Rd$, and $(t,x) \mapsto \a(t,x)$ is a stationary random field taking values in the set of real $d$-by-$d$ matrices satisfying, for a fixed constant~$\Lambda\in [1,\infty)$,
\begin{equation} 
\label{e.unifellip-1}
\forall \xi\in\Rd, \quad 
\xi \cdot \a(t,x) \xi \geq \Lambda^{-1}  \left|\xi\right|^2
\quad \mbox{and} \quad 
\left|\a(t,x) \xi \right| \le \Lambda \left|\xi\right|.
\end{equation}
Here the symbol~$\nabla$ denotes the gradient in the space variables only, that is, $\nabla w = \left( \partial_{x_1} w , \ldots, \partial_{x_d} w \right)$. We let $\P$ be the law of the random field $\a(t,x)$, which we assume to be invariant under translations by elements of $\Z \times \Zd$ and to have a finite range of dependence. (See below in the following subsection for the precise assumptions.)

\smallskip

We are interested in the behavior of the solutions~$u^\ep(t,x)$ for $0<\ep\ll 1$. It is well-known that, under very general qualitative assumptions on the coefficients (stationarity and ergodicity), the equation~\eqref{e.lin.parab} homogenizes to an effective limiting equation of the form
\begin{equation} 
\label{e.lin.parab.homog}
\partial_t u - \nabla \cdot \left( \ahom \nabla u \right) = 0 \quad \mbox{in} \ I\times U,
\end{equation}
where $\ahom$ is a deterministic $d$-by-$d$ matrix. This principle can be formulated in various ways, but it means for example that the solutions~$u^\ep$ of~\eqref{e.lin.parab}, subject to appropriate initial-boundary conditions, converge as $\ep \to 0$, $\P$--almost surely and in some appropriate function space, to solutions of the homogenized equation~\eqref{e.lin.parab.homog}.\footnote{We remark that we are unaware of a reference which proves this specific result in the parabolic setting. Nevertheless, we maintain that it is essentially well-known, since the classical qualitative proof given in the elliptic case~(see for instance~\cite{PV1,BLP,JKO}) can be straightforwardly generalized to the parabolic setting.} Such a result is usually proved by soft arguments, using an abstract version of the ergodic theorem, which unfortunately does not give \emph{quantitative} information concerning the convergence. 

\smallskip

There has been a lot of recent interest in quantitative stochastic homogenization for elliptic equations, particularly in the case of linear, uniformly elliptic equations. This essentially began with the work of Gloria and Otto~\cite{GO1,GO2}, who proved the first quantitative results which are optimal in the scaling of the parameter~$\ep$ (see also~\cite{GNO}). Their work motivated a great number of subsequent works, and we refer to the recently completed monograph~\cite{AKMBook} for more background, references and historical information. 

\smallskip

In this paper, motivated by the desire to obtain quantitative homogenization results---in particular, explicit estimates of the homogenization error---we develop an analytic approach for parabolic equations with random coefficients based on the ideas recently introduced in~\cite{AS,AM,AKM1,AKM2}, which are perhaps best presented in~\cite{AKMBook}. Those papers developed a rather complete quantitative theory of elliptic homogenization starting from the observation that certain energy quantities---which are very natural from a variational perspective---are also rather convenient for studying the homogenization process. This is because: (i) they efficiently encode information about the weak convergence of the fluxes, gradients, and energy densities of solutions; and (ii) they are amenable to renormalization arguments in the sense that we can obtain rates of convergence for the quantities by iterating the length scale. This variational approach allows one to circumvent the need for nonlinear concentration inequalities, because it reveals a ``linear'' structure of the randomness: while the solutions are very nonlinear functions of the coefficients, the energy quantities turn out to be essentially linear. This observation greatly simplifies the theory and allows one to derive estimates which are optimal both in the scaling of~$\ep$ as well as in stochastic integrability. A related approach inspired by~\cite{AS,AM,AKM1,AKM2} has also recently been developed in~\cite{GO5,GO6}.

\smallskip

The two main results of this paper are (i) a quantitative estimate on the homogenization error for Cauchy-Dirichlet problems (Theorem~\ref{t.CDP}) and (ii) a complete large-scale regularity theory (Theorem~\ref{t.regularity}). It has already been observed in the elliptic case (see~\cite{AKMBook}) that results of this type are the first step towards optimal quantitative estimates and scaling limits for first-order correctors as well as optimal error estimate for boundary-value problems. At the same time, the results in this paper are the first quantitative  stochastic homogenization results, to our knowledge, for parabolic equations with coefficients with space-time dependence.  

\smallskip

The starting point for adapting the techniques of~\cite{AKMBook} to the parabolic case is a variational characterization of divergence-form parabolic equations that was first discovered by Brezis and Ekeland~\cite{BE1,BE2}. We give a self-contained presentation of this characterization in Appendix~\ref{s.app}, where we also give a convex analytic proof of the well-posedness of general Cauchy-Dirichlet problems inspired by~\cite{ghoussoub-tzou}. Based on this variational principle, we introduce subadditive quantities for the homogenization problem in Section~\ref{s.varstructure} and adapt the methods of~\cite{AKMBook}, using an iteration of scales and a renormation-type argument, to obtain an algebraic rate of convergence in Section~\ref{s.iteration}. Compared to the elliptic case, the main sources of additional difficulty in the iteration argument have to do with the need to control certain weak Sobolev norms of the time derivatives of the solutions. We accomplish this with the help of some functional inequalities we prove in Section~\ref{s.funineqs}. In Section~\ref{s.CDP}, we show that the convergence of the subadditive quantities gives us approximate first-order correctors with good quantitative bounds, which allows us to prove Theorem~\ref{t.CDP}. In the last section, we obtain the regularity result of Theorem~\ref{t.regularity}. In the rest of this introduction, we state the assumptions, notation and main results.

\subsection{Assumptions}
We fix a spatial dimension $d\geq 2$ and a parameter $\Lambda \in [1,\infty)$. We let $\Omega$ denote the set of all possible coefficient fields $\a(t,x)$, which are assumed to be measurable maps from $\R \times \Rd$ into the set~$\Omega_0$ of matrices~$\a$ satisfying 
\begin{equation} 
\label{e.unifellip0}
\forall \xi\in\Rd, \quad
\xi \cdot \a \xi \ge \Lambda^{-1} |\xi|^2 
\quad \mbox{and} \quad 
 \left|\a \xi \right| \le \Lambda |\xi|.
\end{equation}
That is, we define 
\begin{equation*} \label{}
\Omega_0:= \left\{ \a \in \R^{d\times d} \,:\, \a \ \mbox{satisfies~\eqref{e.unifellip0}} \right\}
\end{equation*}
and then set
\begin{equation}  
\label{e.def.Omega}
\Omega := \left\{ 
\a: \R \times \Rd \to \Omega_0 \
\mbox{is Lebesgue-measurable} \right\}.
\end{equation}
For every Borel subset~$V \subseteq \R\times \Rd$, we define~$\F_V$ to be the $\sigma$--algebra representing the information obtaining by observing the coefficient field in~$V$. Formally, 
\begin{align} 
\label{e.def.FV}
\F_V 
:= 
&  \ 
\mbox{the $\sigma$--algebra generated by the random elements}
\\ & \ \notag
\a \mapsto \int_{V} \varphi(t,x) \a(t,x) \,dt \,dx, \quad \varphi\in C^\infty_c(\R \times\Rd).
\end{align}
The largest of the~$\sigma$-algebras in this family is~$\F := \F_{\R\times\Rd}$. 
We assume that~$\P$ is a given probability measure on the measurable space~$(\Omega,\F)$ which satisfies the following two assumptions:
\begin{enumerate}

\item[(P1)] $\P$ is \emph{stationary} with respect to $\Z\times \Zd$--translations: for every $z\in \Z \times\Zd$ and event $A\in \F$,
\begin{equation*} 
\P \left[ A \right] = \P \left[ T_z A \right].
\end{equation*}

\smallskip

\item[(P2)] $\P$ has a \emph{unit range of dependence}: for every pair of Borel subsets $U, V\subseteq \R \times \Rd$, 
\begin{equation*}
\dist(U,V) \geq 1
\implies
\mbox{$\F_U$\ and \ $\F_V$\ are \ $\P$--independent.}
\end{equation*}
\end{enumerate}
Here~``$\dist$'' is defined with respect to the usual Euclidean distance on~$\R \times\Rd$. We denote by~$\E\left[X \right]$ the expectation of an~$\F$-measurable random variable~$X$ with respect to~$\P$. While we assume that the coefficient field has a finite range of dependence for simplicity, we point out that this hypothesis can be weakened using arguments similar to those exposed in \cite{AM}.

\subsection{Notation}

We unfortunately must introduce quite a bit of notation, particularly since we are working with parabolic equations which require us to define various function spaces. We collect the notation needed in this subsection, which the reader is encouraged to skim and consult as a reference. 

\subsubsection*{General notation}

We denote the set of natural numbers by $\N := \{0,1,2,\ldots\}$. We use the symbols $\wedge$ and $\vee$ to denote minimum and maximum, respectively, for example $r\wedge s = \min\{r,s\}$ for $r,s\in\R$. For every $r\in\R$, we also denote~$r_+:= r\vee 0$ and $r_- := r\wedge 0$. For any $m\in\N$ and measurable subset $E\subseteq \R^m$, the Lebesgue measure of $E$ is denoted by $|E|$, unless $E$ is a finite set, in which case $|E|$ is the cardinality of $E$. This is often used for $m\in\{1,d,1+d\}$. A slash through the integral denotes normalization by the Lebesgue measure: $\fint_E := \frac{1}{|E|} \int_E$. The mean of a function $f\in L^1(E)$ is also denoted by $\left( f \right)_E:= \fint_E f$.

\smallskip

A \emph{parabolic cylinder} is any set of the form $I\times U$ where $I = (I_-,I_+)\subseteq \R$ is a bounded open interval and $U\subseteq \Rd$ is a bounded Lipschitz domain. We denote the parabolic boundary of $I\times U$ by
\begin{equation*}  
\partial_{\sqcup}(I\times U) := \Ll(I \times \partial U\Rr) \cup \Ll( \{I_-\} \times U \Rr) .
\end{equation*}
We denote the Euclidean ball of~$\Rd$ of radius $r\in (0,\infty]$ centered at~$x\in\Rd$ by $B_r(x)$, and put $B_r:=B_r(0)$. 
Throughout, we work with the triadic cubes defined for every $n \in (0,\infty)$ by
\begin{equation*}  
I_n := \Ll( -\frac {3^{2n}}{2}, \frac{3^{2n}}{2} \Rr), \qquad \cu_n := \Ll( - \frac{3^n}{2}, \frac{3^n}{2} \Rr)^d , \qquad \cud_n := I_n \times \cu_n.
\end{equation*}
Note that the parabolic cylinder $\cud_n$ is evidently not a cube per se since its sides have a scaling which match the parabolic scaling. However, we note that for each $m,n\in\N$ with $m<n$, we can write $\cud_n$ as the disjoint union (up to a set of Lebesgue measure zero) of exactly $3^{(2+d)(m-n)}$ cubes of the form $z+ \cud_m$ with $z\in 3^{2m} \Z \times 3^m\Z$. 

\smallskip

We also use the following notation for parabolic cylinders: for each $r\in (0,\infty]$ and $(t,x)\in \R\times\Rd$, we define
\begin{equation} 
\label{e.alternativQr}
\tilde I_r:= (-r^2,0], \ \quad \ 
Q_r(t,x):= (t,x) + \tilde I_r \times B_r, \quad \mbox{and} \quad Q_r := Q_r(0,0).
\end{equation}

\subsubsection*{Function spaces}
For every bounded Lipschitz domain $U \subseteq \Rd$ with $|U| < \infty$ and $p\in [1,\infty)$, we denote the normalized $L^p(U)$ norm of a function $f\in L^p(U)$ by 
\begin{equation} 
\label{e.def.nL}
\| f \|_{\underline{L}^p(U)} := \left( \fint_U \left| f\right|^p\right)^{\frac1p} = \left| U \right|^{-\frac1p}\| f \|_{{L}^p(U)}.
\end{equation}
For $p=\infty$, we denote $\| f \|_{\underline{L}^\infty(U)} := \| f \|_{L^\infty(U)}$. 
We use similar notation to denote normalized (scale-invariant) Sobolev norms: for every $p\in [1,\infty)$ and $f\in W^{1,p}(U)$, 
\begin{equation*} \label{}
\left\| f \right\|_{\underline{W}^{1,p}(U)} 
:= 
|U|^{-\frac1d} \left\| f \right\|_{\underline{L}^{p}(U)} 
+ \left\| \nabla f \right\|_{\underline{L}^{p}(U)} 
\end{equation*}
In the case~$p=2$ we use the notation~$\left\| f \right\|_{\underline{H}^1(U)}:= \left\| f \right\|_{\underline{W}^{1,2}(U)}$. As usual, $H^1_0(U)$ and $W^{1,p}_0(U)$ respectively denote the closure in $H^1(U)$ and $W^{1,p}(U)$, respectively, of the compactly supported smooth functions in~$U$. 
The dual spaces to $W^{1,p}(U)$ and $W^{1,p}_0(U)$ are denoted by $\hat W^{-1,p'}(U)$ and $W^{-1,p'}(U)$, respectively, where $p':=\frac{p}{p-1}$ is the H\"older conjugate exponent of~$p$. The normalized, scale-invariant dual norms are respectively defined by
\begin{equation*} \label{}
\left\| v \right\|_{\underline{\hat{W}}^{-1,p'}(U)}
:=
\sup\left\{
\fint_U uv 
\,:\,
u \in W^{1,p}(U), \ \left\| u \right\|_{\underline{W}^{1,p}(U)} \leq 1
\right\}
\end{equation*}
and
\begin{equation*} \label{}
\left\| v \right\|_{\underline{{W}}^{-1,p'}(U)}
:=
\sup\left\{
\fint_U uv 
\,:\,
u \in W^{1,p}_0(U), \ \left\| u \right\|_{\underline{W}^{1,p}(U)} \leq 1
\right\}.
\end{equation*}
Here we are abusing notation by denoting the natural pairing $\left\langle u,w\right\rangle$ between the two dual spaces (up to a constant) by the normalized integral. This is done to emphasize the normalization that we wish to enforce, which extends the action of an element $w\in C^\infty_c(U)$ on $W^{1,p}(U)$ by $u \mapsto \fint_{U} uw$. For $p=2$, we also write $\|\cdot\|_{\un {\hat H}^{-1}(U)}:= \left\| v \right\|_{\underline{\hat{W}}^{-1,2}(U)}$ and $\|\cdot\|_{\un H^{-1}(U)}:=\left\| v \right\|_{\underline{{W}}^{-1,2}(U)}$.

\smallskip

We next introduce function spaces designed for parabolic equations. For each~$n\in\N$, bounded Lipschitz domain~$U\subseteq\R^n$, Banach space~$X$ and~$p\in[1,\infty)$, we denote by~$L^p(U;X)$ the space of Lebesgue-measurable mappings $u:U \to X$ such that 
\begin{equation*} \label{}
\left\| u \right\|_{\underline{L}^p(U;X)} 
:=
\left( \fint_{U} \left\| u(x) \right\|_{X}^p \,dx
\right)^{\frac1p} <\infty. 
\end{equation*}
For every interval $I = (I_-,I_+) \subset \R$ and bounded Lipschitz domain $U \subset \Rd$, we define the function space
\begin{equation}  
\label{e.def.X}
H^1_\pa(I\times U) := \Ll\{ u \in L^2(I;H^1(U))\,:\,\partial_t u \in L^2(I;H^{-1}(U)) \Rr\} ,
\end{equation}
which is the closure of bounded smooth functions on $I\times U$ with respect to the norm
\begin{equation}
\label{e.def.Xnorm}
\|u\|_{\un {H}^1_\pa(I \times U)} := \|u\|_{\un L^2(I;\un H^1(U))} + \|\partial_t u\|_{\un L^2(I;\un H^{-1}(U))}.
\end{equation}
We denote by $H^1_{\pa,\sqcup}(I\times U)$ the closure in $H^1_\pa(I\times U)$ of the set of smooth functions with compact support in $(I_-,I_+] \times U$. In other words, a function in $H^1_{\pa,\sqcup}(I\times U)$ has zero trace on the lateral boundary $I \times \partial U$ and the initial time $\{ I_-\}\times U$ but does not necessarily vanish at the final time. 

\smallskip

We let $H^1_{\pa,0}(I\times U)$ denote the completion of the set of smooth functions with compact support in $I \times U$ with respect to the norm
\begin{equation}  
\label{e.def.X0}
\|v\|_{\un H^1_{\pa,0}(I\times U)} := \|v\|_{\un L^2(I;\un H^1(U))} + \|\partial_t v\|_{\un L^2(I;\hat {\un H}^{-1}(U))}.
\end{equation}
Note that compared with \eqref{e.def.Xnorm}, here we require the time derivative $\partial_t v(t,\cdot)$ to be an element of $\hat H^{-1}(U)$ instead of $H^{-1}(U)$. In particular, for $v \in H^1_{\pa,0}(I\times U)$, the spatial average of $\partial_t v$ over $I\times U$ is well-defined, since constant functions belong to $L^2(I;H^1(U))$ (while they do not belong to $L^2(I;H^1_0(U))$). Moreover, the boundary condition imposes that for every $v \in H^1_{\pa,0}(I\times U)$, 
\begin{equation}
\label{e.dtv.mean.zero}
\int_{I\times U} \partial_t v = 0.
\end{equation}
This identity is indeed clear if $v$ is smooth and compactly supported in $I\times U$, and we can then obtain the general case by density. 

\smallskip

In certain situations, it is useful to work with variations of $H^1_\pa(I\times U)$ in which the exponent of integrability is $p\in (1,\infty)$ rather than~$2$. So we also define the function spaces
\begin{equation}  
\label{e.def.W1p.par}
W^{1,p}_\pa(I\times U) := 
\Ll\{ u \in L^p\left(I;W^{1,p}(U)\right)\,:\,\partial_t u \in L^p\left(I;W^{-1,p}(U)\right) \Rr\} ,
\end{equation}
which is the closure of bounded smooth functions on $I\times U$ with respect to the norm
\begin{equation}
\label{e.def.W1p.par.norm}
\left\|u\right\|_{\underline{W}^{1,p}_\pa\left(I \times U\right)} 
:= \left\|u\right\|_{\underline{L}^p(I;\underline{W}^{1,p}(U))} + \|\partial_t u\|_{\un L^p\left(I;\underline{W}^{-1,p}(U)\right)}.
\end{equation}
Similarly to $H^1_{\pa,\sqcup}(I\times U)$, we denote by $W^{1,p}_{\pa,\sqcup}(I\times U)$ the closure in $W^{1,p}_\pa(I\times U)$ of the set of smooth functions with compact support in $(I_-,I_+] \times U$. Finally, for every parabolic cylinder $V$, we denote by $W^{1,p}_{\pa,\,\mathrm{loc}}(V)$, $H^1_{\pa,\,\mathrm{loc}}(V)$, and so forth, the functions on~$V$ which are, respectively, elements of $W^{1,p}_\pa(W)$ and $H^1_{\pa}(W)$, etc, for every  subcylinder $W\subseteq V$ with $\overline{W} \subseteq V$. 

\smallskip

We next turn to the definitions of the negative parabolic Sobolev spaces. 
We denote by $\hat{H}^{-1}_{\pa}(V)$ and $H^{-1}_{\pa}(V)$ the dual spaces to $H^1_{\pa}(V)$ and $H^1_{\pa,\sqcup}(V)$, respectively, with (normalized, scale-invariant) dual norms given by
\begin{equation}
\label{e.def.H-1par}
\left\{ 
\begin{aligned} 
&
\| f \|_{\underline{\hat{H}}^{-1}_{\pa}(V)}
:= 
\sup \left\{ \fint_{V} f w  \,:\, w\in H^1_{\pa}(V), 
\ \left\| w \right\|_{\underline{H}^{1}_\pa(V)} \leq 1
\right\}
\\ &
\| f \|_{\underline{{H}}^{-1}_{\pa}(V)}
:= 
\sup \left\{ \fint_{V} f w  \,:\, w\in H^1_{\pa,\sqcup}(V), 
\ \left\| w \right\|_{\underline{H}^{1}_\pa(V)} \leq 1
\right\}.
\end{aligned}
\right.
\end{equation}
As explained above, the notation $\fint_V fw$ should be interpreted as the canonical pairing between $f\in \hat{H}^{-1}_{\pa}(V)$ or $f\in H^{-1}_{\pa}(V)$, respectively, and $w\in H^1_{\pa}(V)$ or~$w\in H^1_{\pa,\sqcup}(V)$, which extends the action of bounded smooth functions on $H^1_{\pa}(V)$ or ~$H^1_{\pa,\sqcup}(V)$. We similarly define the space~$W^{-1,p}_{\pa}(V)$ to be the dual space of the Banach space $W^{1,p'}_{\pa,\sqcup}(V)$, where $p':=\frac{p}{p-1}$, and endow it with the (normalized, scale-invariant) norm
\begin{equation} 
\label{e.negativeW1pa}
\left\| f \right\|_{\underline{W}^{-1,p}_{\pa}(V)} 
:= 
\sup \left\{ \fint_{V} f w  \,:\, w\in W^{1,p'}_{\pa,\sqcup}(V), 
\ \left\| w \right\|_{\underline{W}^{1,p'}_\pa(V)} \leq 1
\right\}.
\end{equation}
Recall that negative Sobolev norms arise naturally when one wishes to quantify \emph{weak} convergence in $L^p$ or positive Sobolev spaces (see~\cite[Section 1.4]{AKMBook}). This is indeed their purpose in this paper.

\subsubsection*{The $\O_s$ notation}
Since the random variables we encounter in this paper are very often the sum of a deterministic quantity and a ``small'' random part, it is useful to work with the notation introduced in~\cite{AKM2} for expressing the sizes of random variables (essentially, an alternative notation for certain Orlicz norms). It is intended to remind us of ``big-$O$'' notation and is convenient because it compresses some of our computations and makes our inequalities easier to understand at a glance. 

\smallskip

If $X$ is a random variable and $s,k\in (0,\infty)$, then we write 
\begin{equation*} \label{}
X \leq \O_s(k)
\end{equation*}
as a shorthand for the statement that 
\begin{equation}
\label{e.chebyforward}
\E \left[ \exp \left( \left( \frac{X_+}{k} \right)^s \right) \right] \leq 2. 
\end{equation}
Roughly, this means that ``$X$ is of order $k$ with stretched exponential tails with exponent $s$.'' More precisely, we can use Chebyshev's inequality to see that
\begin{equation} 
\label{e.Cheby}
X \leq \O_s(k) \implies \forall \lambda>0, \ \P \left[ X > \lambda k \right] \leq 2\exp\left( -\lambda^s \right). 
\end{equation}
The converse of this statement is almost true: for every $k\geq 0$, 
\begin{equation}
\label{e.chebyconverse}
\forall \lambda \ge 0, \quad \P\Ll[ X \ge \lambda k \Rr] \le \exp \Ll( - \lambda^s \Rr)  \implies  X \le \O_s \Ll(2^\frac 1 s \, \theta\Rr).
\end{equation}
This can be obtained by integration. We also use the notation
\begin{equation*} \label{}
X = \O_s(k) 
\iff 
X \leq \O_s(k) \ \mbox{and} \ -X \leq \O_s(k). 
\end{equation*}
Similarly, we write $X\leq Y + \O_s(k)$ to mean that $X-Y \leq \O_s(k)$ and $X = Y+\O_s(k)$ to mean that $X-Y = \O_s(k)$. If $s\in [1,\infty)$, then Jensen's inequality gives us a triangle inequality for $\O_s(\cdot)$ in the following sense: for any measure space $(E,\mathcal{S},\mu)$, measurable function $K : E \to (0,\infty)$ and jointly measurable family $\left\{ X(z) \right\}_{z \in E}$ of nonnegative random variables, we have 
\begin{equation} 
\label{e.Osums}
\forall z\in E, \ X(z) \le \O_s(K(z))
\implies
\int_{E} X\,d\mu \leq \O_s\left( \int_E K \,d\mu \right). 
\end{equation}
If $s\in (0,1]$, then the statement is true after adding a prefactor constant $C_s>1$ to the right side. 
For a proof of \eqref{e.Cheby}-\eqref{e.Osums}, see~\cite[Appendix A]{AKMBook}.

\subsection{Statement of the main results}

We present two main results. The first provides an algebraic convergence rate for the homogenization limit of the Cauchy-Dirichlet initial-value problem in a parabolic cylinder~$I\times U$, where~$U\subseteq\Rd$ is a bounded Lipschitz domain. This is a parabolic counterpart of a theorem proved in the elliptic setting in~\cite{AS} (see also~\cite[Theorem 2.16]{AKMBook}). 

\begin{theorem}
\label{t.CDP}
Fix $s\in (0,2+d)$, a bounded Lipschitz domain $U \subseteq B_1$, an interval $I:= (I_-,0) \subseteq \left(-\tfrac14,0\right)$ and an exponent $\delta > 0$. Put $V:= I\times U$. There exist an exponent $\beta(\delta, V,d, \Lambda)>0$, a constant $C(s,V,\delta,d,\Lambda)<\infty$ and a random variable $\X$ satisfying
\begin{equation*} \label{}
\X = \O_1(C)
\end{equation*} 
such that the following convergence result holds: for each $\ep \in (0,\frac12]$ and initial-boundary condition~$f\in W^{1,2+\delta}_{\pa} \left(V\right)$, denoting
\begin{equation*} \label{}
\a^\ep(t,x):= \a\left(\frac t{\ep^2}, \frac x\ep  \right)
\end{equation*}
and taking~$u^\ep, u \in f + H^1_{\pa,\sqcup}(V)$ to be the solutions of the Cauchy-Dirichlet problems
\begin{equation} 
\label{e.CDPs}
\left\{
\begin{aligned}
& \partial_t u^\ep - \nabla \cdot \left( \a^\ep \nabla u^\ep  \right) = 0 & \mbox{in} & \ V, 
\\ 
& u^\ep = f & \mbox{on} & \ \partial_\sqcup V,
\end{aligned}
\right.
\ \ \mbox{and} \ \ 
\left\{
\begin{aligned}
& \partial_t u - \nabla \cdot \left( \ahom \nabla u  \right) = 0 & \mbox{in} & \ V, 
\\ 
& u = f & \mbox{on} & \ \partial_\sqcup V,
\end{aligned}\right.
\end{equation}
we have the estimate
\begin{multline}
\label{e.CPDerrorestimate}
\left\|  \nabla u^\ep - \nabla u \right\|_{\hat{H}^{-1}_\pa(V)}
+ \left\| \a^\ep \nabla u^\ep - \ahom \nabla u \right\|_{\hat{H}^{-1}_\pa(V)}
+ \left\| u^\ep - u \right\|_{{L}^2(V)}
\\
\leq C  \left\| \nabla f \right\|_{W^{1,2+\delta}_\pa(V)} \left( \ep^{\beta(2+d-s)} + \X \ep^{s} \right). 
\end{multline}
\end{theorem}

As well as estimating the homogenization error, notice that the estimate~\eqref{e.CPDerrorestimate} quantifies the weak convergence in~$L^2(V)$ of the gradients and fluxes of~$u^\ep$ to those of~$u$.
The random part of the error, namely~$\X\ep^s$ for an~$s$ arbitrarily close to~$2+d$, is very small compared to the deterministic part, $\ep^{\beta(2+d-s)}$. It is also important for applications to observe that~$\X$ is independent of the initial-boundary condition~$f$. 

\smallskip

On the right side of \eqref{e.CPDerrorestimate}, we have split the error into a possibly rather large deterministic part (large, since we do not control the smallness of $\beta > 0$) plus a random error. While the typical size of the error is estimated suboptimally, since $\beta>0$ is small, the \emph{tail behavior} of this random part is sharply estimated. In particular, we see that the probability for the term $(\eps^{\beta(2+d-s)} + \X \eps^s)$ to be $O(1)$ is smaller than $\exp(-\eps^{-s}/C)$, for arbitary $s < 2+d$. This estimate is sharp, in the sense that it would be false for any $s > 2+d$. We refer to \cite[Remark~2.5 and Section~3.5]{AKMBook} for similar considerations in the elliptic setting.

\smallskip

The second theorem we present here is a large-scale regularity result, a parabolic counterpart to~\cite[Theorem 3.6]{AKMBook}. In particular, we seek to classify all ancient solutions of the parabolic equation which exhibit at most polynomial growth at infinity and backwards in time. This requires us to introduce some additional notation. 

\smallskip

We denote polynomials in the variables~$t,x_1,\ldots,x_d$ by $\mathcal{P}(\R\times \Rd)$. The \emph{parabolic degree}~$\deg_p(w)$ of an element~$w \in \mathcal{P}(\R\times \Rd)$ is the degree of the polynomial~$(t,x) \mapsto w(t^2,x)$. For each $k\in\N$ we let $\mathcal{P}_k(\R \times\Rd)$ be the subset of~$\mathcal{P}(\R\times\Rd)$ of polynomials with parabolic degree at most~$k$. For $\alpha>0$, we say that a function $\phi:\R\times \Rd\to \R$ or $\phi: (-\infty,0]\times \Rd\to \R$ is \emph{parabolically $\alpha$-homogeneous} if
\index{parabolically homogeneous}
\begin{equation*} \label{}
\forall \lambda \in\R, \quad \phi(\lambda^2 t,\lambda x) = \lambda^\alpha \phi(t,x).
\end{equation*}
Any element of $\mathcal{P}_k(\R \times\Rd)$ can be written as a sum of at most $C(d,k)<\infty$ many parabolically homogeneous polynomials. 

\smallskip

We denote by $\Ahom_k(Q_\infty)$ the set of~$\ahom$-caloric functions on~$Q_\infty$ with growth which is strictly less than a polynomial of parabolic degree $k+1$:
\begin{multline*} \label{}
\Ahom_k (Q_\infty)
:=
\\
\left\{ 
w\in H^1_\pa(Q_\infty) \,:\,
\limsup_{r \to \infty} \, r^{-(k+1)}  \left\| w \right\|_{\underline{L}^2(Q_r) }  = 0, \ \partial_t w - \nabla \cdot \left( \ahom \nabla w \right) = 0 \ \mbox{in} \ Q_\infty
\right\}.
\end{multline*}
It turns out that $\Ahom_k(Q_\infty)$ coincides with the set of \emph{$\ahom$-caloric polynomials}\footnote{$\ahom$-caloric polynomials are often called \emph{heat polynomials} in the literature, in the case $\ahom=I_d$} of parabolic degree at most $k$. That is, 
\begin{equation} 
\label{e.identifyAhomkQinfty}
\Ahom_k(Q_\infty)
=
\left\{ w\vert_{Q_\infty}  \,:\, w \in \mathcal{P}_k(\R \times \Rd), \ \partial_t w- \nabla \cdot \left( \ahom \nabla w\right) = 0 \ \mbox{in} \ \R\times \Rd \right\}. 
\end{equation}
The vector space of~$n$-homogeneous $\ahom$-caloric polynomials is isomorphic to that of $n$-homogeneous polynomials of $\R^d$. 
This can be shown by backwards uniqueness and the fact that this vector space is spanned by products of homogeneous $\ahom$-caloric polynomials depending only on $t$ and one of the space variables (see for instance~\cite{widder} or \cite[Proposition~1.1.1]{nualart}). In any case, we have that~$\dim(\Ahom_k(Q_\infty))= \binom{d+k}{d}<\infty$. 

\smallskip

In the next result, we generalize the parabolic Liouville theorem implicit in~\eqref{e.identifyAhomkQinfty} to~$\a(x)$-caloric functions. At the same time we provide a \emph{quantitative} version of this Liouville principle, in other words, a~$C^{k,1}$-type regularity estimate. Denote, for every parabolic cylinder $I\times U\subseteq \R\times \Rd$,
\begin{equation*} \label{}
\A(I\times U):= 
\left\{
w\in H^1_{\pa,\,\mathrm{loc}} (I\times U)
\,:\,
 \partial_t w - \nabla \cdot \left( \a \nabla w \right) = 0 \ \mbox{in} \ I\times U
\right\}
\end{equation*}
and, for every $k\in\N$,
\begin{equation*} \label{}
\A_k(Q_\infty):= 
\left\{ 
w\in \A(Q_\infty)  \,:\,
\limsup_{r \to \infty} r^{-(k+1)}  \left\| w \right\|_{\underline{L}^2(Q_r) }  = 0
\right\}.
\end{equation*}
Note that these vector spaces are \emph{random} since they depend on~$\a\in\Omega$. 
The following theorem is a parabolic analogue of \cite[Theorem 3.6]{AKMBook}.

\begin{theorem}
[Parabolic higher regularity theory]
\label{t.regularity}
Fix $s \in (0,2+d)$. There exist an exponent $\delta(s,d,\Lambda)\in \left( 0, \frac12 \right]$ and a random variable $\X_s$ satisfying the estimate
\begin{equation}
\label{e.X.parab}
\X_s \leq \O_s\left(C(s,d,\Lambda)\right)
\end{equation}
such that the following statements hold, for every $k\in\N$:
\begin{enumerate}
\item[{$\mathrm{(i)}_k$}] There exists $C(k,d,\Lambda)<\infty$ such that, for every $u \in \A_k(Q_\infty)$, there exists $p\in \Ahom_k(Q_\infty)$ such that, for every $R\geq \X_s$,
\begin{equation} \label{e.liouvillec.parab}
\left\| u - p \right\|_{\underline{L}^2(Q_R)} \leq C R^{-\delta} \left\| p \right\|_{\underline{L}^2(Q_R)}.
\end{equation}

\item[{$\mathrm{(ii)}_k$}]For every $p\in \Ahom_k(Q_\infty)$, there exists $u\in \A_k(Q_\infty)$ satisfying~\eqref{e.liouvillec.parab} for every $R\geq \X_s$. 

\item[{$\mathrm{(iii)}_k$}]
There exists $C(k,d,\Lambda)<\infty$ such that, for every $R\geq \X_s$ and $u\in \A(Q_R)$, there exists $\phi \in \A_k(Q_\infty)$ such that, for every $r \in \left[ \X_s,  R \right]$, we have the estimate
\begin{equation}
\label{e.intrinsicreg.parab}
\left\| u - \phi \right\|_{\underline{L}^2(Q_r)} \leq C \left( \frac r R \right)^{k+1} 
\left\| u \right\|_{\underline{L}^2(Q_R)}.
\end{equation}
\end{enumerate}
In particular, we have, $\P$-almost surely, for every $k\in\N$,
\begin{equation} 
\label{e.dimensionofAk.parab}
\dim(\A_k(Q_\infty)) = \dim(\Ahom_k(Q_\infty)) =   \binom{d+k}{d}.
\end{equation}
\end{theorem}

Observe that, as in the elliptic case, even for $k=0$ the third statement of Theorem~\ref{t.regularity} gives us an important gradient estimate on solutions. Indeed, the combination of statement {$\mathrm{(iii)}_0$} and the Caccioppoli inequality yields that, for every $R\geq \X_s$, $u\in \A(Q_R)$ and $r \in [\X_s,R]$, we have 
\begin{equation*} \label{}
\left\| \nabla u \right\|_{\underline{L}^2(Q_r)} 
\leq 
C\left\| \nabla u \right\|_{\underline{L}^2(Q_R)}. 
\end{equation*}
This should be seen as a $C^{0,1}$-type estimate and compared to pointwise gradient bounds for the solutions of the heat equation. 

\smallskip

The proof of Theorem~\ref{t.regularity} is obtained as a consequence of Theorem~\ref{t.CDP} and a routine adaptation of the proof of~\cite[Theorem 3.6]{AKMBook}, which is the statement of the analogous result in the elliptic case. In Section~\ref{s.regularity}, we explain the modifications required in the parabolic setting. 

\smallskip

Soon after the first version of this paper was submitted and posted to the arxiv, a new preprint of Bella,~Chiarini and~Fehrman~\cite{BCF} appeared which contains a large-scale regularity result which has some overlap with Theorem~\ref{t.regularity}. In particular, under qualitative assumptions, they obtain the statement of Theorem~\ref{t.regularity} in the case $k=1$ with the estimate~\eqref{e.X.parab} on~$\X_s$ replaced by the qualitative bound $\P\left[ \X_s <\infty \right]=1$.

\subsection{Outline of the paper}
In the next section, we introduce the subadditive quantities inherited from the variational structure of the equation and record some of their basic properties. For convenience, the variational formulation of uniformly parabolic equations is recalled in a self-contained presentation in Appendix~\ref{s.app}. In Section~\ref{s.funineqs}, we present several functional inequalities which are needed later in the paper. Of particular interest are inequalities giving us control of certain weak norms of functions in terms of the spatial averages of the functions in cubes as well as Caccioppoli-type inequalities giving us control of strong norms of solutions in terms of weak norms. Section~\ref{s.iteration} is the heart of the paper, where we prove the convergence of the subadditive quantities by an iteration over the length scales. In Section~\ref{s.CDP}, we demonstrate how to pass from control of the convergence of the subadditive quantities to general homogenization results. Finally, in Section~\ref{s.regularity} we summarize the passage from the quantitative homogenization results to the higher regularity theory (which is entirely analogous to the elliptic setting). In Appendix~\ref{s.meyers}, we give local and global versions of the Meyers higher integrability estimate for gradients of solutions. We remark that the statement of the global Meyers estimate we prove appears to be new and somewhat sharper compared to what has previously appeared in the literature.

\section{Variational structure and subadditive quantities}
\label{s.varstructure}

\subsection{Variational formulation of parabolic equations}
As we now explain, the solution of a parabolic equation can be obtained as the minimizer of a uniformly convex functional. This is an entirely deterministic statement, valid for an arbitrary fixed coefficient field $\a \in \Omega$. 

\smallskip

The following proposition states the solvability of parabolic equations. It relies on convex analysis and calculus of variations, and is close to the main result of~\cite{ghoussoub-tzou} (see also the monograph \cite{ghoussoub-book}). We provide a self-contained proof in the appendix in the more general setting of maximal monotone operators, and for a larger set of pairs $(w,w^*)$; see Proposition~\ref{p.parabolic.min.app}.
\begin{proposition}[Parabolic variational principle]
\label{p.parabolic.min}
Let $\mcl J$ be defined below in \eqref{e.def.mclJ}. For each $w \in H^1_\pa(I\times U)$ and $w^* \in L^2(I;H^{-1}(U))$, the mapping
\begin{equation*}  
\Ll\{
\begin{array}{ccc}
w + H^1_{\pa,\sqcup}(I\times U)  & \to & \R \\
u & \mapsto & \mcl J[u,w^*]
\end{array}
\Rr.
\end{equation*}
is uniformly convex. Moreover, its minimum is zero, and the associated minimizer is the unique $u \in w + H^1_{\pa,\sqcup}(I\times U)$ solution of
\begin{equation}
\label{e.parab.rhs}
(\partial_t - \nabla \cdot \a \nabla)u = w^*  \quad \text{in } I \times U.
\end{equation}
\end{proposition}
Equation \eqref{e.parab.rhs} is interpreted as
\begin{equation}
\label{e.parab.rhs.weak}
\forall \phi \in L^2(I;H^1_0(U)), \quad \int_{I \times U} \nabla \phi \cdot \a \nabla u = \int_{I \times U} \phi \Ll( w^* - \partial_t u \Rr).
\end{equation}
The left side of \eqref{e.parab.rhs.weak} can be more explicitly written as
\begin{equation*}  
\int_{I \times U} \nabla \phi(t,x) \cdot \a(t,x) \nabla u(t,x) \, dt \, dx,
\end{equation*}
while the right side of \eqref{e.parab.rhs.weak} could be more properly written as 
\begin{equation*}  
\int_I \langle \phi(t,\cdot), (w^* - \partial_t u)(t,\cdot) \rangle \, dt,
\end{equation*}
with $\langle \cdot, \cdot \rangle$ the duality pairing between $H^1_0(U)$ and $H^{-1}(U)$. 

\smallskip

We proceed to define the functional $\J$ appearing in Proposition~\ref{p.parabolic.min}. To start with, we decompose the matrix $\a$ into its symmetric and skew-symmetric parts:
\begin{equation*}  
\sym(t,x) := \frac {\a(t,x) + \a^\T(t,x)} 2, \qquad \skew(x) := \frac {\a(t,x) - \a^\T(t,x)} 2,
\end{equation*}
and set
\begin{equation}  
\label{e.def.F}
A(p,q,t,x) := \frac 1 2 p \cdot \sym(t,x) p + \frac 1 2 (q - \skew(t,x) p) \cdot \sym^{-1}(t,x) (q - \skew(t,x) p),
\end{equation}
so that the following lemma holds.
\begin{lemma}
\label{l.lin.Fitz}
 There exists a constant $C(\Lambda) < \infty$ such that, for every $(t,x) \in \R \times \R^d$, 
\begin{equation*}  
(p,q) \mapsto A(p,q,t,x) - C^{-1} (|p|^2 + |q|^2) \quad \text{is convex}, 
\end{equation*}
and
\begin{equation*}  
(p,q) \mapsto A(p,q,t,x) - C (|p|^2 + |q|^2) \quad \text{is concave}.
\end{equation*}
Moreover, for every $(t,x) \in \R \times \Rd$ and $p,q \in \Rd$,
\begin{equation*}  
A(p,q,t,x) \ge p\cdot q,
\end{equation*}
with equality if and only if $q = \a(t,x) p$.
\end{lemma}
\begin{proof}
We briefly recall the proof---see also \cite[(2.6)]{AM}. The fact that $(p,q) \mapsto A(p,q,t,x)$ is uniformly convex and $C^{1,1}$ follows from the definition of $\Omega$ in \eqref{e.def.Omega}. The second part of the lemma is a consequence of the identity
\begin{equation*}  
A(p,q,t,x) - p\cdot q = \frac 1 2 (\a(t,x) p - q) \sym^{-1}(t,x)  (\a(t,x) p - q). \qedhere
\end{equation*}
\end{proof}
The functional $\J$ appearing in Proposition~\ref{p.parabolic.min} is defined, for every $u \in H^1_\pa(I\times U)$ and $u^* \in L^2(I;H^{-1}(U))$, by
\begin{equation}
\label{e.def.mclJ}
\mcl J[u,u^*] := \inf \Ll\{ \int_{I \times U} \Ll(A(\nabla u,\g,\cdot) - \nabla u \cdot \g\Rr) \ : \ -\nabla \cdot \g = u^* - \partial_t u \Rr\}.
\end{equation}
In the infimum above, we understand that $\g \in L^2(I\times U;\Rd)$, and the last condition is interpreted as
\begin{equation*}  
\forall \phi \in L^2(I;H^1_0(U)), \quad \int_{I\times U} \nabla \phi \cdot \g = \int_{I \times U} \phi \Ll( u^* - \partial_t u \Rr) .
\end{equation*}
In the integral on the right side of \eqref{e.def.mclJ}, the dot ``\,$\cdot$\,'' in the expression $A(\nabla u, \g,\cdot)$ stands for the time-space variable, that is,
\begin{multline*}  
\int_{I \times U} \Ll(A(\nabla u,\g,\cdot) - \nabla u \cdot \g\Rr) \\
= \int_{I} \int_U \Ll(A(\nabla u(t,x),\g(t,x),t,x) - \nabla u(t,x) \cdot \g(t,x) \Rr) \, d x \, dt.
\end{multline*}

\subsection{Subadditive quantities and basic properties}
In this subsection, we define the subadditive quantities and collect their basic properties. Although their definitions are actually very natural and intuitive, many readers will not find them to be on first reading. In order to understand the motivation for studying them, it is best to first have some familiarity with the elliptic case with symmetric coefficients, which is described in~\cite{AKMBook}. Indeed, much of what appears below can be compared to Chapter~2 of~\cite{AKMBook}, and in fact this paper can be seen as a generalization of~\cite[Chapters 1-3]{AKMBook} to the parabolic setting. Now, since the subadditive quantities are endowed from the variational structure of the equation, it is natural that the parabolic versions should be somewhat more complicated than the elliptic ones. A similar issue was encountered in~\cite{AM}, where subadditive quantities were defined and analyzed for ``non-variational'' elliptic equations.

\smallskip

In any case, the most convincing demonstration that these are the ``right'' quantities will have to wait until Section~\ref{s.CDP}, where we prove that quantitative information about the convergence of the subadditive quantities can be translated directly into control of the first-order correctors and therefore into estimates on the rate of homogenization. 

\smallskip

Without further ado, we give the definitions of the subadditive quantities. For every Lipschitz domain $U\subseteq \Rd$, bounded interval $I\subseteq \R$ and $p,q \in \Rd$, we define
\begin{equation}
\label{e.def.mu}
\mu(I\times U, p,q) := \inf_{(\nabla v,\mathbf{h})\in \mcl{C}_0(I\times U)} \fint_{I\times U} A(p+\nabla v, q + \h, \cdot),
\end{equation}
where the infimum is taken over $(\nabla v,\h)$ ranging in the space
\begin{multline}
\label{e.def.candidate}
\mcl{C}_0(I\times U) := \bigg\{ (\nabla v , \h) \in L^2(I \times U, \Rd)^2 \ : \ v \in H^1_{\pa,0}(I \times U) \ \ \text{and} \\
\forall \phi \in L^2(I;H^1(U)), \quad \int_{I \times U} \nabla \phi \cdot \h = - \int_{I \times U}  \phi \, \partial_t v
\bigg\}.
\end{multline}
Since $\phi \in L^2(I;H^1(U))$ and $\partial_t v \in L^2(I;\hat H^{-1}(U))$, the last integral is well-defined, in the usual interpretation as 
$$
\int_I \langle \phi(t,\cdot),\partial_t v(t,\cdot) \rangle \, dt,
$$
where here $\langle \cdot, \cdot \rangle$ denotes the duality pairing between $H^1(U)$ and $\hat H^{-1}(U)$. Testing the condition in \eqref{e.def.candidate} against the function $\phi(t,x) := p\cdot x$ and integrating by parts in time, we see that any candidate $\h$ must satisfy
\begin{equation}
\label{e.h.mean.zero}
\int_{I \times U} \h = 0.
\end{equation}
The dual subadditive quantity~$\mu^*$ is defined, for every $p^*, q^* \in \Rd$, by
\begin{equation}  
\label{e.def.mu*}
\mu^*(I\times U, q^*,p^*) := \sup_{(\nabla u,\g) \in \C(I\times U)} \fint_{I \times U} \Ll( -A(\nabla u, \g, \cdot) + q^* \cdot \nabla u + p^* \cdot \g \Rr) ,
\end{equation}
where the supremum is taken over $(\nabla u,\g)$ ranging  in the space
\begin{multline}
\label{e.def.candidate*}
\C(I\times U) := \bigg\{ (\nabla u , \g) \in L^2(I \times U ;  \Rd)^2 \,:\, u\in H^1_\pa(I \times U), \\
\forall \phi \in L^2(I;H^1_0(U)), \quad \int_{I \times U} \nabla \phi \cdot \g = - \int_{I \times U}  \phi \, \partial_t u
\bigg\}.
\end{multline}
Note that for each $p \in \Rd$ and $(\nabla v,\h) \in \C_0(I\times U)$, we have $(p + \nabla v, q + \h) \in \C(I\times U)$. Using also \eqref{e.dtv.mean.zero} and \eqref{e.h.mean.zero}, we thus deduce that for every $p,q,p^*,q^* \in \Rd$,
\begin{equation}  
\label{e.almost.dual}
\mu^*(I\times U,q^*,p^*) \ge q^* \cdot p + p^* \cdot q - \mu(I\times U,p,q).
\end{equation}
That is, the function $(q^*,p^*) \mapsto \mu^*(I\times U,q^*,p^*)$ is bounded below by the convex dual of the function $(p,q) \mapsto \mu(I\times U,p,q)$. 
As in the elliptic case (see \cite[Lemma~3.1]{AKM1} and \cite{AKM2}), we will combine $\mu$ and~$\mu^*$ into a master quantity denoted by~$J$ which monitors the defect in this convex duality pairing. 
For concision, we set
\begin{equation}
\label{e.def.V}
V := I \times U,
\end{equation}
and define a $2d$-by-$2d$ matrix field $\Abf:\R\times \Rd \to \R^{2d\times 2d}$ by 
\begin{equation*} \label{}
\Abf(t,x)
:= 
 \begin{pmatrix}  \s(t,x) - \m(t,x) \s^{-1}(t,x) \m(t,x) & \m(t,x) \s^{-1}(t,x) \\ -\s^{-1} (t,x) \m(t,x) & \s^{-1}(t,x) \end{pmatrix},
\end{equation*}
so that 
\begin{equation*}  
A(p,q,t,x) = \frac 1 2 \begin{pmatrix} p \\ q \end{pmatrix}\cdot \Abf \begin{pmatrix} p \\ q \end{pmatrix}.
\end{equation*}
This notation allows to rewrite the definitions of $\mu$ and $\mu^*$ in \eqref{e.def.mu} and \eqref{e.def.mu*} in more compact notation: for every $X, X^* \in \R^{2d}$, we have
\begin{equation}
\label{e.def.mu.compact}
\mu(V,X) = \inf_{S \in X + \C_0(V)} \fint_V \frac 1 2 S \cdot \Abf S,
\end{equation}
\begin{equation}
\label{e.def.mu*.compact}
\mu^*(V,X^*) = \sup_{S \in \C(V)} \fint_V \Ll( -\frac 1 2 S \cdot \Abf S + X^* \cdot S \Rr) ,
\end{equation}
and the inequality \eqref{e.almost.dual} can be rewritten as
\begin{equation}
\label{e.ordering}
\mu^*(V,X^*) \ge  X\cdot X^* - \mu(V,X).
\end{equation}
We now set
\begin{equation}  
\label{e.def.S}
\S(V) := \bigg\{(\nabla v, \h) \in \C(V) \ : \ 
\forall (\nabla \phi, \f) \in \C_0(V), \quad \int_{V} 
\begin{pmatrix}
\nabla \phi \\
 \f 
 \end{pmatrix}
 \cdot \Abf 
 \begin{pmatrix}
 \nabla v\\
  \h
  \end{pmatrix} 
  = 0\bigg\},
\end{equation}
and for every $X, X^* \in \R^{2d}$,
\begin{equation} 
\label{e.compact.J}
J( V, X,X^* ) 
:=
\sup_{S \in \S(V)} \fint_{V} \left( -\frac 1 2 S \cdot \Abf S - X\cdot \Abf S + X^* \cdot S \right).
\end{equation}
The master quantity $J$ can be rewritten in the following more explicit notation:
\begin{multline}
\label{e.def.J}
J\left( V,  \begin{pmatrix}  p \\ q \end{pmatrix} ,  \begin{pmatrix}  q^* \\ p^* \end{pmatrix}  \right) := 
\\
\sup_{(\nabla v, \g) \in \S(V)} 
\fint_{V}\left( 
-\frac12  \begin{pmatrix}  \nabla v \\ \g \end{pmatrix}  \cdot \mathbf{A}  \begin{pmatrix} \nabla v \\ \g \end{pmatrix} 
-  \begin{pmatrix}  p \\ q \end{pmatrix} \cdot \mathbf{A}  \begin{pmatrix}  \nabla  v \\ \g \end{pmatrix}  
+  \begin{pmatrix}  q^* \\ p^* \end{pmatrix} \cdot  \begin{pmatrix}  \nabla v \\ \g \end{pmatrix}
\right).
\end{multline}
The next lemma shows that $J$ indeed monitors the defect in convex duality between $\mu$ and $\mu^*$.
\begin{lemma}  
\label{l.Jsplitting}
For every $X,X^*\in\R^{2d}$,
\begin{equation} 
\label{e.Jsplitting}
J(V,X,X^*)= \mu(V,X)+\mu^*(V,X^*) - X\cdot X^*.  
\end{equation}
Moreover, the maximizer $S(\cdot,V,X,X^*)$ in \eqref{e.compact.J} is the difference between the maximizer of~$\mu^*(V,X^*)$ in~\eqref{e.def.mu*.compact} and the minimizer of $\mu(V,X)$ in~\eqref{e.def.mu.compact}.
\end{lemma}
\begin{proof}
We first argue that, for every $X^*\in \R^{2d}$, 
\begin{equation}
\label{e.mu*.S}
\mu^*(V,X^*) = \sup_{S \in \S(V)} \fint_V \Ll( - \frac 1 2 S \cdot \Abf S + X^* \cdot S \Rr) .
\end{equation}
Let $S^*  \in \mcl{C}(V)$ denote the maximizer in the definition of $\mu^*(V,X^*)$.
Note that, for every $X^* \in \R^{2d}$ and $S \in \mcl{C}_0(V)$,
\begin{equation}  
\fint_U X^* \cdot S = 0.
\end{equation}
By the first variation for~$\mu^*$, we deduce that 
\begin{equation*}  
\forall\, S' \in  \mcl{C}_0(V), \quad \int_V S' \cdot \Abf S^* = 0.
\end{equation*}
That is, $S^*\in \S(U)$, and thus~\eqref{e.mu*.S} holds.

\smallskip

Let $X = (p,q) \in \R^{2d}$ and 
\begin{equation}
\label{e.def.S0}
S_0 = (\nabla v, \h) \in X + \mcl{C}_0(V)
\end{equation}
denote the minimizer in the definition of $\mu(V,p,q) = \mu(V,X)$. For every $S \in \S(V)$, 
\begin{multline}
\label{e.mu.nu.bigJ}
\mu(V,X) + \fint_V  \Ll( X^* \cdot S - \frac 1 2 S \cdot \Abf S \Rr)  - X \cdot X^* \\
 = \fint_V \Ll( \frac 1 2 S_0 \cdot \Abf S_0 - \frac 1 2 S \cdot \Abf S + X^* \cdot S \right) - X \cdot X^*.
\end{multline}
By~\eqref{e.def.S0}, we have $X  = \fint_V S_0$. For each $S \in \S(V)$, 
\begin{equation*}  
\fint_V S \cdot \Abf S_0 = \fint_V S \cdot \Abf X,
\end{equation*}
and this last identity holds true in particular for $S = S_0$. 
We obtain that the left side of~\eqref{e.mu.nu.bigJ} is equal to 
\begin{equation*}  
\fint_V  \Ll( -\frac 1 2 \Ll( S - S_0 \Rr) \cdot \Abf \Ll( S - S_0 \Rr) 
- X \cdot \Abf\Ll( S - S_0 \Rr) + X^* \cdot (S - S_0)\Rr).
\end{equation*}
We compare this result to the identites~\eqref{e.mu*.S} and \eqref{e.compact.J} to obtain the lemma.
\end{proof}

The next lemma collects elementary properties of $J$ and its minimizer. It can be compared with~\cite[Lemma 2.2]{AKMBook}.

\begin{lemma}[Basic properties of $J$]
\label{l.basicJ}
The quantity $J(V,X,X^*)$ and its maximizer $S(\cdot,V,X,X^*)$ satisfy the following properties:

\begin{itemize}

\item The mapping $(X,X^*) \mapsto J(V,X,X^*)$ is quadratic. 

\item \emph{Uniformly convex and $C^{1,1}$ in $X$ and $X^*$ separately.} There exists a constant $C(d,\Lambda)<\infty$ such that, for every  $X_1,X_2,X^*\in\R^{2d}$, 
\begin{multline} 
\label{e.JunifconvC11p}
\frac1C \left| X_1-X_2\right|^2
\leq
\frac12 J\left( V,X_1,X^* \right) + \frac12 J\left(V,X_2, X^* \right)  - J\left(V,\frac12X_1+\frac12 X_2 , X^*\right)
\\
\leq
C \left| X_1-X_2\right|^2
\end{multline}
and, for every $X_1^*,X_2^*,X\in\R^{2d}$, 
\begin{multline} 
\label{e.JunifconvC11q}
\frac1C \left| X_1^*-X_2^*\right|^2
\leq
\frac12 J\left(V,X,X_1^* \right) + \frac12 J\left( V,X,X_2^* \right)  - J\left( V,X,\frac12X_1^*+\frac12 X_2^* \right)
\\
\leq
C \left| X_1^* - X_2^*\right|^2.
\end{multline}

\item \emph{Subadditivity.} Let $V_1, \ldots, V_N \subset V$ be parabolic cylinders that partition $V$, in the sense that $V_i \cap V_j = \emptyset$ if $i\neq j$ and 
\begin{equation*} \label{}
\left| V \setminus \bigcup_{i=1}^N V_i \right| = 0. 
\end{equation*}
For every $X,X^*\in\R^{2d}$, we have
\begin{equation}
\label{e.Jsubadd}
J(V,X,X^*) \leq \sum_{i=1}^N \frac{\left|V_i\right|}{|V|} J(V_i,X,X^*). 
\end{equation}

\item \emph{First variation for $J$.} For $X,X^*\in\R^{2d}$, the function $S(\cdot,V,X,X^*)$ is the unique element of $\S(V)$ such that
\begin{equation}
\label{e.Jfirstvar} 
\forall\, T \in \S(V), \qquad \fint_V  T \cdot \Abf S(\cdot,V,X,X^*) = \fint_V \left( - X\cdot \Abf T + X^*\cdot T \right). 
\end{equation}

\item \emph{Quadratic response.} For every $X,X^*\in\R^{2d}$ and $T\in \S(V)$, 
\begin{multline} 
\label{e.Jquadresponse}
\frac 1C \fint_V \left|   T -   S(\cdot,V,X,X^*) \right|^2
\\
  \leq
J(V,X,X^*) - \fint_V \left( -\frac12  T \cdot \Abf  T  -X\cdot\Abf   T + X^*\cdot   T  \right)
\\
\leq
C \fint_V \left|   T - S(\cdot,V,X,X^*) \right|^2.
\end{multline}

\item \emph{Formulas for derivatives of $J$.} For every $X,X^*\in\R^{2d}$,
\begin{equation} 
\label{e.JderX}
\nabla_X J(V,X,X^*)  = -\fint_V \Abf S(\cdot,V,X,X^*)
\end{equation}
and
\begin{equation} 
\label{e.JderX*}
\nabla_{X^*} J(V,X,X^*)  =  \fint_V  S(\cdot,V,X,X^*).
\end{equation}
\end{itemize}
\end{lemma}
\begin{proof}
Since these properties are easy to check and their proofs are almost the same of those of~\cite[Lemma 2.2]{AKMBook}, we omit the details. 
\end{proof}

\begin{remark}
\label{r.spatav.determined}
Since $X^* \mapsto \mu^*(V,X^*)$ is a quadratic form, we obtain from \eqref{e.JderX} and Lemma~\ref{l.Jsplitting} that
\begin{equation}  
\label{e.spatav.X0}
\fint_{V} S(\cdot,V,X,0) = \nabla_{X^*} J(V,X,0) = -X,
\end{equation}
a property which also follows directly from the definition of $\mu$ in \eqref{e.def.mu.compact} and the identification of $- S(\cdot,V,X,0)$ as the minimizer in this definition. From \eqref{e.JderX*} and Lemma~\ref{l.Jsplitting}, we also obtain the dual identity
\begin{equation*}  
\fint_V \Abf S(\cdot,V,0,X^*) = X^*.
\end{equation*}
\end{remark}

In the next lemma, we relate the space $\S(I\times U)$ with the space of solutions of the parabolic equation and of its dual. Define the vector space $\A(I\times U)$ to be the set of weak solutions $u\in H^1_\pa(I \times U)$ of the equation
\begin{equation*} 
\partial_t u - \nabla \cdot \left( \a\nabla u \right) = 0 \quad \mbox{in} \ I \times U, 
\end{equation*}
and the vector space $\A^*(I\times U)$ to be the set of weak solutions $u^*\in H^1_\pa(I\times U)$ of the dual equation
\begin{equation*} 
\partial_t u^* + \nabla \cdot \left( \a^\T\nabla u^* \right) = 0 \quad \mbox{in} \ I \times U.
\end{equation*}
Note that the direction of time is reversed in the dual equation. Precisely,
\begin{multline*} 
\A(I\times U) := \\
\left\{ u\in H^1_\pa(I\times U) \,:\, \forall w\in L^2(I;H^1_0(U)), \ \ \int_{I\times U} w \, \partial_t u =- \int_{I\times U} \nabla w \cdot \a \nabla u  \right\},
\end{multline*}
\begin{multline*} 
\A^*(I\times U) := \\
\left\{ u^* \in H^1_\pa(I\times U) \,:\, \forall w\in L^2(I;H^1_0(U)), \ \  \int_{I\times U} w \, \partial_t u^* = \int_{I\times U} \nabla w \cdot \a^\T \nabla u^*  \right\}.
\end{multline*}

\begin{lemma} We have
\label{l.identif.S}
\begin{equation} 
\label{e.identif.S}
\S(V) := \left\{ \left( \nabla u + \nabla u^*, \a \nabla u- \a^\T\nabla u^* \right) \,:\, u\in \A(V), \ u^* \in \A^*(V) \right\}.
\end{equation}
\end{lemma}
\begin{proof}
Recall that $V = I \times U$, and denote by $\mcl S'(I\times U)$ the set on the right side of~\eqref{e.identif.S}. The condition
\begin{equation*}  
\int_{I\times U} (\nabla \phi, \f) \cdot \Abf (\nabla v, \h) = 0
\end{equation*}
appearing in \eqref{e.def.S}
can be rewritten more explicitly as
\begin{equation}  
\label{e.explicit.orthog}
\int_{I \times U} \Ll( \nabla \phi \cdot \s \nabla v  + \Ll( \f - \m \nabla \phi  \Rr) \cdot \s^{-1} \Ll( \h - \m \nabla v \Rr) \Rr) = 0.
\end{equation}
We first verify that $\S(I \times U) \subset \S'(I\times U)$. The space $\{0\} \times L^2(I,\Lso(U))$ is a subspace of $\C_0(I\times U)$. Hence, if \eqref{e.explicit.orthog} holds for every $(\nabla \phi, \h) \in \C_0(I \times U)$, then in particular
\begin{equation*}  
\forall \f \in L^2(I; \Lso(U)), \qquad \int_{I \times U} \f \cdot \s^{-1} (\h - \m \nabla v) = 0.
\end{equation*}
In other words, $\s^{-1} (\h - \m \nabla v)$ belongs to the space orthogonal to $L^2(I; \Lso(U))$ in $L^2(I\times U)$. That is, there exists $w \in L^2(I; H^1(U))$ such that
\begin{equation*}  
\s^{-1} (\h - \m \nabla v) = \nabla w,
\end{equation*}
and we deduce that for every $(\nabla \phi, \f) \in \C_0(I\times U)$,
\begin{equation*}  
\int_{I \times U} \Ll(\nabla \phi \cdot ( \s \nabla v + \m \nabla w) - w \, \partial_t \phi \Rr) = 0.
\end{equation*}
Denoting by $\Delta_N^{-1}$ the solution operator for the Laplace equation on $U$ with null Neumann boundary condition, we observe that for each $\phi \in H^1_{\pa,0}(I\times U)$, the pair $(\nabla \phi,\nabla \Delta_N^{-1} (\partial_t \phi))$ belongs to $\C_0(I\times U)$. The identity above therefore holds for arbitrary $\phi \in H^1_{\pa,0}(I\times U)$, and we thus deduce that $\partial_t w \in L^2(I,H^{-1}(U))$. We can then integrate by parts in time and obtain that
\begin{equation}  
\label{e.twist1}
\forall \phi \in H^1_{\pa,0}(I\times U), \quad \int_{I \times U} \Ll(\nabla \phi \cdot ( \s \nabla v + \m \nabla w) + \phi \, \partial_t w \Rr) = 0.
\end{equation}
This property can be extended to arbitrary $\phi \in L^2(I;H^1_0(U))$ by density. 
The additional requirement that $(\nabla v, \h) \in \C(I\times U)$ brings 
\begin{equation}  
\label{e.twist2}
\forall \psi \in L^2(I;H^1_0(U)), \quad \int_{I \times U} \Ll(\nabla \psi \cdot (\s \nabla w + \m \nabla v) +  \psi \, \partial_t v \Rr) = 0.
\end{equation}
Setting 
\begin{equation}
\label{e.here.are.u.u*}
u := \frac 1 2 (v + w), \qquad u^* := \frac 1 2 (v - w),
\end{equation}
we deduce that $u \in \A(I \times U), u^* \in \A^*(I \times U)$, with
\begin{equation*}  
v = \frac 1 2 (u  + u^*), \qquad \h = \a \nabla u - \a^\T \nabla u^*,
\end{equation*}
and this completes the proof that $\S(I \times U) \subset \S'(I\times U)$. 

\smallskip

Conversely, given $u \in \A(I \times U)$ and $u^* \in \A^*(I \times U)$, we set
\begin{equation*}  
v = u+u^*, \qquad w := u - u^*, \qquad \h := \a \nabla u - \a^\T \nabla u^* = \s  \nabla w + \m \nabla v,
\end{equation*}
and observe that
\begin{equation*}  
\a \nabla u + \a^\T \nabla u^* = \s \nabla v + \m \nabla w.
\end{equation*}
The identities \eqref{e.twist1} and \eqref{e.twist2} follow. This implies that the condition \eqref{e.explicit.orthog} is satisfied for every $(\nabla \phi,\f) \in \C_0(I\times U)$, and hence that $(\nabla v, \h) \in \S(I\times U)$. We have thus shown that $\S'(I\times U) \subset \S(I\times U)$, which completes the proof.
\end{proof}

\begin{remark}
\label{r.deconstruction}
Note that for $S = \left( \nabla u + \nabla u^*, \a \nabla u - \a^\T\nabla u^* \right)\in \S(V)$, we have
\begin{equation}  \label{e.diff.recovery}
\Abf S = (\a \nabla u + \a^\T \nabla u^*, \nabla u - \nabla u^*).
\end{equation}
Indeed,~\eqref{e.diff.recovery} is implicit in the proof of Lemma~\ref{l.identif.S} above and can also be checked by a direct computation. 
In particular,~$\nabla u$ can be written as one half the sum of the first component of~$S$ and second component of~$\Abf S$, and $\a\nabla u$ can be recovered similarly. This observation is needed in Section~\ref{s.CDP} in the construction of (approximate) correctors.  
\end{remark}

\section{Functional inequalities}
\label{s.funineqs}

We collect here some functional inequalities which will be useful in the rest of the paper. The two main results are a ``multiscale'' version of the Poincar\'e inequality, and a Caccioppoli-type inequality for elements of $\S(\cud_n)$. The proof of the latter is based on a parabolic version of the Helmhotz-Hodge decomposition of vector fields, which is of independent interest. 

\smallskip

We first recall a useful version of the Poincar\'e inequality, for functions of the space variable only.
\begin{lemma}
\label{l.Poincare.psi}
Let $\psi \in L^2(\cu_n)$ satisfy 
\begin{equation*} \label{}
\fint_{\cu_n} \psi = 1. 
\end{equation*}
There exists $C(d)<\infty$ such that, for every $u\in H^1(\cu_n)$,  
\begin{equation} 
\label{e.whateverPoincare}
\left\| u - \fint_{\cu_n} u\, \psi \right\|_{\underline{L}^2(\cu_n)} 
\leq C \left\| \psi \right\|_{\underline{L}^2(\cu_n)} \left\| \nabla u \right\|_{\underline{\hat H}^{-1}(\cu_n)}. 
\end{equation}
\end{lemma}
\begin{proof}
By the usual Poincar\'e inequality, all we need to show is that 
\begin{equation} 
\label{e.adjustmean}
\left|  \fint_{\cu_n}u\left( 1 - \psi \right) \right| \leq C\left\| \psi \right\|_{\underline{L}^2(\cu_n)} \left\| \nabla u \right\|_{\underline{H}^{-1}(\cu_n)}.
\end{equation}
Let $w$ be the solution of the Neumann problem
\begin{equation*} \label{}
\left\{
\begin{aligned}
& - \Delta w = 1-\psi & \mbox{in} & \ \cu_n, \\
& \mathbf{n} \cdot \nabla w = 0 & \mbox{on} & \ \partial \cu_n. 
\end{aligned}
\right.
\end{equation*}
Notice that this has a solution because $\int_{\cu_n} (1-\psi) = 0$, and we have the $H^2$ estimate (see for instance \cite[Lemma~B.18]{AKMBook})
\begin{equation*} \label{}
\left\| \nabla w \right\|_{\underline{H}^1(\cu_n)} 
\leq C \left\| 1 - \psi \right\|_{\un L^2(\cu_n)}
\leq C \left( 1 + \left\| \psi \right\|_{\un L^2(\cu_n)}\right) 
\leq C\left\| \psi \right\|_{\un L^2(\cu_n)}.
\end{equation*}
Testing the equation for $w$ by $u$ thus yields
\begin{align*} \label{}
\left|  \fint_{\cu_n}u\left( 1 - \psi \right)  \right|
& 
= \left|  \fint_{\cu_n} \nabla u \cdot \nabla w \right|
\\ &
\leq C \left\| \nabla u \right\|_{\underline{\hat H}^{-1}(\cu_n)} \left\| \nabla w \right\|_{\underline{H}^{1}(\cu_n)}
\\ &
\leq C\left\| \psi \right\|_{\un L^2(\cu_n)}\left\| \nabla u \right\|_{\underline{\hat H}^{-1}(\cu_n)}. \qedhere
\end{align*}
\end{proof}

For every parabolic cylinder $V$ and $f \in L^1(V)$, we recall that we use the following shorthand notation for the spatial average of $f$ over $V$:
\begin{equation}
\label{e.def.spat.av}
(f)_{V} := \fint_V f.
\end{equation}
By the standard Poincar\'e inequality in $1+d$ coordinates, we have
\begin{equation}  
\label{e.naive.Poincare}
\|u - (u)_{\cud_n}\|_{\un L^2(\cud_n)} \le C 3^n \|\nabla u\|_{\un L^2(\cud_n)} + C 3^{2n} \|\partial_t u\|_{\un L^2(\cud_n)}.
\end{equation}
In the context of parabolic equations, it is natural to try to preserve a matching between the number of times a function is differentiated in space and \emph{half} the number of times it is differentiated in time. The estimate \eqref{e.naive.Poincare} is not consistent with this scaling. The purpose of the next proposition is to obtain such a bound---see also Corollary~\ref{c.parab.poincare} below.
\begin{proposition}
\label{p.L2.weak}
There exists $C(d) <\infty$ such that, for every $u\in H^1_\pa(\cud_n)$ and $\g \in L^2(\cud_n;\Rd)$ satisfying $\partial_t u = \nabla \cdot \g$, we have
\begin{equation*}
\left\| u - \left( u \right)_{\cud_n} \right\|_{\un L^2(\cud_n)}
\leq C \left(
 \left\| \nabla u \right\|_{\underline{L}^2\left(I_n;\underline{\hat H}^{-1}(\cu_n)\right)} 
+ \left\| \g \right\|_{\underline{L}^1\left(I_n;\underline{H}^{-1}(\cu_n)\right)}
\right).
\end{equation*}
\end{proposition}
\begin{remark}
\label{r.explicit.g}
In the statement of Proposition~\ref{p.L2.weak} (and similarly for Corollary~\ref{c.parab.poincare} and Proposition~\ref{p.integrate} below), the condition 
$\partial_t u = \nabla \cdot \g$
is interpreted as 
\begin{equation*}  
\forall \phi \in L^2(I_n;H^1_0(\cu_n)), \quad \int_{\cud_n} \nabla \phi \cdot \g = - \int_{\cud_n}  \phi \, \partial_t u.
\end{equation*}
Equivalently, this amounts to saying that $(\nabla u,\g) \in \mcl C(\cud_n)$. As an example, we can always take $\g = \nabla \Delta_{\cu_n}^{-1} \partial_t u$, where $\Delta_{\cu_n}^{-1}$ is the solution operator for the Laplacian in $\cu_n$ with null Dirichlet boundary condition.
\end{remark}
\begin{proof}[Proof of Proposition~\ref{p.L2.weak}]
Let $\psi \in C^\infty_c(\cu_n)$ be a smooth function of compact support in $\cu_n$ such that $\fint_{\cu_n} \psi(x)\,dx = 1$, $0 \leq \psi \leq 2$, and
\begin{equation} 
\label{e.movemeanaround}
3^{-n} \left\| \nabla \psi \right\|_{L^\infty(\cu_n)} + \left\| \nabla^2 \psi \right\|_{L^\infty(\cu_n)} \leq C3^{-2n}.
\end{equation}
We write $\left( u \right)_{\cu_n,\psi}:= \fint_{\cu_n} u(x)\psi(x) \, dx$ and $\left( u \right)_{\cu_n} : = \fint_{\cu_n} u(x) \, dx$. 
Using the Poincar\'e inequality (in the form given by Lemma~\ref{l.Poincare.psi}) in time slices gives, for every $t\in I_n$, 
\begin{equation*} \label{}
\left\| u(t,\cdot) - \left( u(t,\cdot) \right)_{\cu_n,\psi} \right\|_{\underline{L}^2(\cu_n)} 
\leq 
C \left\| \nabla u(t,\cdot) \right\|_{\underline{\hat H}^{-1}(\cu_n)}.
\end{equation*}
Thus
\begin{equation} 
\label{e.pp0}
\fint_{I_n} \fint_{\cu_n} \left| u(t,x) - \left( u(t,\cdot) \right)_{\cu_n,\psi} \right|^2 \,dx\,dt \leq C \fint_{I_n} \left\| \nabla u(t,\cdot) \right\|_{\underline{\hat H}^{-1}(\cu_n)}^2\,dt. 
\end{equation}
Since $\psi\in H^1_0(U)$ and $\nabla \psi \in H^1_0(U;\Rd)$, we have, for every $t\in I_n$,
\begin{align*}
\left| \partial_t \left( u(t,\cdot) \right)_{\cu_n,\psi} \right| 
& 
= \left|  \fint_{\cu_n} \psi(x) \partial_t u(t,x)\,dx \right|
\\ & 
= \left| \fint_{\cu_n} \nabla \psi(x) \cdot \g(t,\cdot) (x)\,dx \right|
\\ & 
\leq C \left\| \nabla \psi \right\|_{\underline{H}^1(\cu_n)} \left\| \g(t,\cdot) \right\|_{\underline{H}^{-1}(\cu_n)} 
\\ & 
\leq C3^{-2n} \left\|\g(t,\cdot)\right\|_{\underline{H}^{-1}(\cu_n)}. 
\end{align*}
Thus 
\begin{align*} \label{}
\lefteqn{
\sup_{t\in I_n} \left| \left( u(t,\cdot) \right)_{\cu_n,\psi} - \fint_{I_n\times\cu_n} \psi(x)u(t,x)\,dx\,dt \right|
} \qquad & 
\\ &
\leq \int_{I_n} \left| \partial_t  \left( u(t,\cdot) \right)_{\cu_n,\psi} \right| \,dt
\\ & 
\leq C3^{-2n} \int_{I_n} \left\| \g(t,\cdot) \right\|_{\underline{H}^{-1}(\cu_n)} \,dt 
\\ & 
= C \fint_{I_n} \left\| \g(t,\cdot) \right\|_{\underline{H}^{-1}(\cu_n)} \,dt .
\end{align*}
Combining this with~\eqref{e.pp0}, we obtain
\begin{align*} \label{}
\lefteqn{
\fint_{I_n} \fint_{\cu_n} \left| u(t,x) - \fint_{I_n\times \cu_n} \psi(y) u(s,y)\,ds\,dy \right|^2 \,dx\,dt 
} \quad &
\\ &
\leq C \fint_{I_n} \left\| \nabla u(t,\cdot) \right\|_{\underline{\hat H}^{-1}(\cu_n)}^2\,dt + \left(  \fint_{I_n} \left\| \g(t,\cdot) \right\|_{\underline{H}^{-1}(\cu_n)} \,dt \right)^2.
\end{align*}
Since 
\begin{equation*}  
\|u - (u)_{\cud_n}\|_{\un L^2(\cud_n)} = \inf_{c \in \R} \|u - c\|_{\un L^2(\cud_n)} ,
\end{equation*}
this yields the announced result.
\end{proof}
\begin{corollary}
\label{c.parab.poincare}
There exists $C(d) <\infty$ such that, for every $u\in H^1_\pa(\cud_n)$ and $\g \in L^2(\cud_n,\Rd)$ satisfying $\partial_t u = \nabla \cdot \g$,
\begin{equation}
\label{e.parab.poincare}
\left\| u - \left( u \right)_{\cud_n} \right\|_{\un L^2(\cud_n)}
\leq C3^n \left(
 \left\| \nabla u \right\|_{\underline{L}^2\left(\cud_n\right)} 
+ \left\| \g \right\|_{\underline{L}^2\left(\cud_n\right)}
\right).
\end{equation}
\end{corollary}
\begin{proof}
This follows from Proposition~\ref{p.L2.weak} and the inequalities
\begin{equation*} \label{}
\left\| v \right\|_{\underline{L}^2\left(I_n;\underline{H}^{-1}(\cu_n)\right)} 
\leq \left\| v \right\|_{\underline{L}^2\Ll(I_n;\underline{\hat H}^{-1}(\cu_n)\Rr)} \le
C3^n \left\| v \right\|_{\underline{L}^2(\cud_n)}. \qedhere
\end{equation*}
\end{proof}
\begin{remark}  
\label{r.lapl.lapl}
Recall from Remark~\ref{r.explicit.g} that in the statement of Corollary~\ref{c.parab.poincare}, we can take $\g = \nabla \Delta^{-1}_{\cu_n} \partial_t u$. Moreover, there exists $C(d) < \infty$ such that, for every $f \in L^2(I_n;H^{-1}(\cu_n))$,
\begin{equation*}  
C^{-1} \, \|f\|_{\un L^2(I_n;\un H^{-1}(\cu_n))} \le \|\nabla \Delta_{\cu_n}^{-1} f\|_{\un L^2(\cud_n)} \le C \|f\|_{\un L^2(I_n;\un H^{-1}(\cu_n))}.
\end{equation*}
Indeed, by the standard Poincar\'e inequality, the norm~$\left\| \cdot \right\|_{\underline{H}^{1}(\cu_n)}$ is equivalent to the norm $u \mapsto \|\nabla u\|_{\un L^2(\cu_n)}$ on $H^1_0(\cu_n)$, and moreover,
\begin{align*}  
& \sup \Ll\{ \fint_{\cud_n} f \phi \ : \ \phi \in L^2(I_n; H^1_0(\cu_n)), \ \|\nabla \phi\|_{L^2(\cud_n)} \le 1 \Rr\} \\
& \qquad = \sup \Ll\{ \fint_{\cud_n} \nabla \phi \cdot \nabla \Delta_{\cu_n}^{-1} f  \ : \ \phi \in L^2(I_n; H^1_0(\cu_n)), \ \|\nabla \phi\|_{L^2(\cud_n)} \le 1 \Rr\}  \\
& \qquad = \|\nabla \Delta_{\cu_n}^{-1} f\|_{\un L^2(\cud_n)}.
\end{align*}
In particular, on the right side of \eqref{e.parab.poincare}, we can replace the term
$
\left\| \g \right\|_{\underline{L}^2\left(\cud_n\right)}
$
with the quantity 
$
\|\partial_t u\|_{\un L^2(I_n;\un H^{-1}(\cu_n))}
$.
\end{remark}

The next proposition allows to obtain a control of the $H^{-1}_{\pa}(V)$ norm of a function from a knowledge of its spatial averages over large scales. 
For each $m,n \in \N$, $m \le n$, we set
\begin{equation}
\label{e.def.mclZ}
\mcl Z_m := \Ll[\Ll(3^{2m} \Z\Rr) \times \Ll(3^m \Z^d\Rr) \Rr]\cap \cud_n.
\end{equation}
Although $\mcl Z_m$ depends on $n$, we keep this dependence implicit in the notation, since its identity will be clear from the context. This is a parabolic version of the inequality which first appeared in~\cite[Proposition 6.1]{AKM1}.

\begin{proposition}[Multiscale Poincar\'e inequality]
\label{p.msp}
There exists $C(d,\Lambda) < \infty$ such that, for every $n \ge 1$ and $f \in L^2(\cud_n)$,
\begin{equation*}  
\|f\|_{\un {\hat H}^{-1}_{\pa}(\cud_n)} 
\leq
C \|f\|_{\un L^2(\cud_n)}  + C \sum_{m = 0}^{n-1} 3^{m} \Ll( |\mcl Z_m|^{-1} \sum_{z \in \mcl Z_m} |(f)_{z + \cud_m}|^2 \Rr)^\frac 1 2. 
\end{equation*}
\end{proposition}
\begin{proof}[Proof of Proposition~\ref{p.msp}]
Recalling \eqref{e.def.H-1par}, we fix $g \in H^{1}_{\pa}(\cud_n)$ such that 
\begin{equation}  
\label{e.normalize.g}
\|g\|_{\un H^{1}_{\pa}(\cud_n)} \le 1,
\end{equation}
and decompose the proof into two steps.

\smallskip

\emph{Step 1.} In this step, we show that there exists a constant $C(d) < \infty$ such that, for every $m \in \{0,\ldots,n\}$,
\begin{equation}
\label{e.local.Poincare}
\sum_{z \in \mcl Z_{m}} \int_{z + \cud_m} |g - (g)_{z + \cud_m}|^2 \le C |\cud_n| \, 3^{2m}.
\end{equation}
By Corollary~\ref{c.parab.poincare} and Remark~\ref{r.lapl.lapl}, the left side above is bounded by
\begin{equation*}  
C 3^{2m} \, |\cud_m| \sum_{z \in \mcl Z_m} \Ll( \|\nabla g\|_{\un L^2(z + \cud_m)}^2 + \|\partial_t g\|_{\un L^2(z_0+I_m,\un H^{-1}(z' + \cu_m))}^2 \Rr) ,
\end{equation*}
where we write $z = (z_0,z') \in \Z\times \Zd$. The contribution of the first term is easily estimated, since by \eqref{e.normalize.g},
\begin{equation*}  
|\cud_m| \sum_{z \in \mcl Z_m}  \|\nabla g\|_{\un L^2(z + \cud_m)}^2 = \int_{\cud_n} |\nabla g|^2 \le |\cud_n|.
\end{equation*}
For the second term, we write
\begin{align*}  
& \sum_{z \in \mcl Z_m} \|\partial_t g\|_{\un L^2(z_0+I_m,\un H^{-1}(z' + \cu_m))}^2
 \\
&  = \sup \Ll\{ \sum_{z \in \mcl Z_m} \Ll(\fint_{z + \cud_m} \phi_z \, \partial_t g\Rr)^2 \ : \ \phi_z \in L^2(z_0 + I_m; H^1_0(z'+\cu_m)), \ \|\nabla \phi_z\|_{\un L^2(z + \cud_m)} \le 1 \Rr\} .
\end{align*}
For $\phi_z$ satisfying the conditions in the supremum above, we have
\begin{multline*}  
|\cud_n|^{-1} \sum_{z \in \mcl Z_m} \Ll(\int_{z + \cud_m} \phi_z \, \partial_t g\Rr)^2 = \fint_{\cud_n} \partial_t g \Ll( \sum_{z \in \mcl Z_m} \phi_z \int_{z + \cud_n} \phi_z \partial_t g \Rr) 
\\
\le \|\partial_t g\|_{\un L^2(I_n; \un H^{-1}(\cu_n))} \Ll\| \sum_{z \in \mcl Z_m}  \phi_z \int_{z + \cud_n} \phi_z \partial_t g \Rr\|_{\un L^2(I_n;\un H^1(\cu_n))}.
\end{multline*}
Notice that, by \eqref{e.normalize.g}, the first term on the right side of the previous inequality is bounded by~$1$. 
Moreover, by the normalization of the functions $\phi_z$,
\begin{align*}  
\Ll\| \sum_{z \in \mcl Z_m}  \phi_z \int_{z + \cud_n} \phi_z \partial_t g \Rr\|_{\un L^2(I_n;\un H^1(\cu_n))}^2 
&\le C \Ll\| \sum_{z \in \mcl Z_m}  \nabla \phi_z  \int_{z + \cud_n} \phi_z \partial_t g\Rr\|_{\un L^2(\cud_n)}^2 \\
& \le C \frac{|\cud_m|}{|\cud_n|} \sum_{z \in \mcl Z_m} \Ll(\int_{z + \cud_m} \phi_z \, \partial_t g\Rr)^2.
\end{align*}
Combining the last three displays, we arrive at 
\begin{equation*}  
\sum_{z \in \mcl Z_m} \|\partial_t g\|_{\un L^2(z_0+I_m,\un H^{-1}(z' + \cu_m))}^2 \le C \frac {|\cud_n|}{|\cud_m|},
\end{equation*}
and this completes the proof of \eqref{e.local.Poincare}.

\smallskip

\emph{Step 2.}
We aim to control $\fint_{\cud_n} fg$, which we decompose into
\begin{equation}  
\label{e.decomp.intfg}
\fint_{\cud_n} fg = \fint_{\cud_n} f (g - (g)_{\cud_n}) + (f)_{\cud_n} \, (g)_{\cud_n}.
\end{equation}
By the definition of the $H^{1}_{\pa}$ norm in \eqref{e.def.Xnorm}, we have $(g)_{\cud_n} \le 3^n$, and therefore, by Jensen's inequality,
\begin{equation}  
\label{e.msp.firstbit}
(f)_{\cud_n} \, (g)_{\cud_n} \le 3^{n} \Ll|(f)_{\cud_n}\Rr| \le C 3^{n-1} \Ll( |\mcl Z_{n-1}|^{-1} \sum_{z \in \mcl Z_{n-1}} |(f)_{z + \cud_{n-1}}|^2 \Rr)^\frac 1 2.
\end{equation}
We then proceed to decompose the first integral on the right side of \eqref{e.decomp.intfg} recursively. For every $m \in \{0,\ldots, n-1\}$ and $z\in \mcl Z_{m+1}$, we have
\begin{multline}
\label{e.iter.msp}
\int_{z+\cud_{m+1}} f \cdot \left( g - \left( g\right)_{z+\cud_{m+1}} \right) 
= \sum_{y\in \mcl Z_{m} \cap (z+\cud_{m+1})} \int_{y+\cud_{m}} f \cdot \left( g - \left( g\right)_{y+\cud_{m}} \right) \\
+ |\cud_{m}|\sum_{y\in \mcl Z_{m} \cap (z+\cud_{m+1})} \left(  \left( g\right)_{y+\cud_{m} }  -  \left( g\right)_{z+\cud_{m+1}} \right) \cdot \left( f \right)_{y+\cud_{m}}.
\end{multline}
Summing over $z \in \mcl Z_{m+1}$ and using H\"older's inequality, we get
\begin{multline*}  
\sum_{z\in \mcl Z_{m+1}} \int_{z+\cud_{m+1}} f \cdot \left( g - \left( g\right)_{z+\cud_{m+1}} \right)  \leq \sum_{y\in \mcl Z_{m}} \int_{y+\cud_{m}} f \cdot \left( g - \left( g\right)_{y+\cud_{m}} \right) \\
+|\cud_{m}| \Bigg(\sum_{\substack{z \in \mcl Z_{m+1} \\ y \in \mcl Z_{m} \cap (z+\cud_{m+1})}} \Ll|   \left( g\right)_{y+\cud_{m} } -  \left( g\right)_{z+\cud_{m+1}}\Rr|^2\Bigg)^{\frac 1 2} \Ll( \sum_{y \in \mcl Z_{m}}\left| \left( f\right)_{y+\cud_{m} } \right|^2  \Rr)^\frac 1 2.
\end{multline*}
By Jensen's inequality, we have, for each $z \in \mcl Z_{m+1}$,
\begin{equation*}  
\sum_{y \in \mcl Z_{m} \cap (z+\cud_{m+1})} \Ll|   \left( g\right)_{y+\cud_{m} } -  \left( g\right)_{z+\cud_{m+1}}\Rr|^2 \le \sum_{y \in \mcl Z_{m} \cap (z+\cud_{m+1})} \fint_{y + \cud_m} \Ll| g -  \left( g\right)_{z+\cud_{m+1}}\Rr|^2,
\end{equation*}
and thus, by \eqref{e.local.Poincare}, 
\begin{equation*}  
\sum_{\substack{z \in \mcl Z_{m+1} \\ y \in \mcl Z_m \cap (z + \cud_{m+1})}} \Ll| (g)_{y + \cud_m} - (g)_{z + \cud_{m+1}} \Rr|^2 \le C \, \frac{|\cud_n|}{|\cud_{m}|} \, 3^{2m}.
\end{equation*}
Using also that $|\mcl Z_m| = |\cud_n|/|\cud_m|$ and combining with \eqref{e.iter.msp}, we obtain
\begin{multline*}
\sum_{z\in \mcl Z_{m+1}} \int_{z+\cud_{m+1}} f \cdot \left( g - \left( g\right)_{z+\cud_{m+1}} \right) \\ \leq \sum_{y\in \mcl Z_{m}} \int_{y+\cud_{m}} f \cdot \left( g - \left( g\right)_{y+\cud_{m}} \right)
+C\, |\cud_n| \, 3^{m} \left( |\mcl Z_{m}|^{-1}\sum_{y\in \mcl Z_{m}}\left| \left( f\right)_{y+\cud_{m} } \right|^2 \right)^{\frac12}.
\end{multline*}
Summing over $m \in \{0,\ldots,n-1\}$ yields
\begin{multline*}
\int_{\cud_n} f \Ll(g - (g)_{\cud_n}\Rr)\\
 \leq \sum_{z\in \mcl Z_0} \int_{z+\cud_0} f \cdot \left( g - \left( g\right)_{z+\cud_{0}} \right) + C \, \left| \cud_m\right| \, \sum_{m=0}^{n-1} 3^{ m}    \left( \left| \mcl Z_m \right|^{-1} \sum_{y\in \mcl Z_m}\left| \left( f \right)_{y+\cud_{m} } \right|^2 \right)^{\frac12}.
\end{multline*}
Hence, by H\"older's inequality and \eqref{e.local.Poincare},
\begin{multline*}
\int_{\cud_n} f \Ll(g - (g)_{\cud_n}\Rr)  \\
\le C|\cud_n|^{\frac1{2}} \left(  \int_{\cud_n} \left| f\right|^2 \right)^{\frac12}  
+ C \, \left| \cud_n\right| \, \sum_{m=0}^{n-1} 3^{m}    \left( \left| \mcl Z_m \right|^{-1} \sum_{y\in \mcl Z_m}\left| \left( f \right)_{y+\cud_{m} } \right|^2 \right)^{\frac12}.
\end{multline*}
Dividing by $|\cud_n|$ and combining with \eqref{e.decomp.intfg}-\eqref{e.msp.firstbit}, we obtain
\begin{equation*}  
\fint_{\cud_n} fg \le 
C  \|f\|_{\un L^2(\cud_n)} 
+ C\sum_{m=0}^{n} 3^{m}    \left( \left| \mcl Z_m \right|^{-1} \sum_{y\in \mcl Z_m}\left| \left( f \right)_{y+\cud_{m} } \right|^2 \right)^{\frac12}.
\end{equation*}
Taking the supremum over all $g$ satisfying \eqref{e.normalize.g} yields the result.
\end{proof}

The name ``multiscale Poincar\'e inequality'' for Proposition~\ref{p.msp} is best understood in conjunction with the following statement.
\begin{proposition}
\label{p.integrate}
There exists a constant $C(d) < \infty$ such that, for every $n \in \N$, $u \in H^1_{\pa}(\cud_{n+1})$ and $\g \in L^2(\cud_{n+1};\Rd)$ satisfying $\partial_t u = \nabla \cdot \g$, we have
\begin{equation*}  
\|u - (u)_{\cud_n}\|_{\un L^2(\cud_n)} \le C \|\nabla u\|_{\un {\hat H}^{-1}_{\pa}(\cud_{n+1})} + C \|\g\|_{\un {\hat H}^{-1}_{\pa}(\cud_{n+1})}.
\end{equation*}
\end{proposition}
\begin{remark}
\label{r.change.interior}
It is clear that the proof of Proposition~\ref{p.integrate} can be adapted to show that for every $r > 0$, there exists a constant $C(r,d) < \infty$ such that, for every $n \in \N$, $u \in H^1_{\pa}(\cud_{n})$ and $\g \in L^2(\cud_{n};\Rd)$ satisfying $\partial_t u = \nabla \cdot \g$, we have
\begin{equation*}  
\|u - (u)_{\cud_{n-r}}\|_{\un L^2(\cud_{n-r})} \le C \|\nabla u\|_{\un {\hat H}^{-1}_{\pa}(\cud_{n})} + C \|\g\|_{\un {\hat H}^{-1}_{\pa}(\cud_{n})}.
\end{equation*}
Although one can expect that the estimate above still holds for $r = 0$, we leave it as an open question here, and content ourselves with an interior estimate.
\end{remark}
Combining Propositions~\ref{p.msp} and \ref{p.integrate} allows to estimate the (interior) $L^2$ oscillation of $u$ in terms of spatial averages of $\nabla u$ and $\g$ (see also \cite[Proposition~6.1]{AKM1}). The estimate yields better interior information than the ``single-scale'' Poincar\'e inequality provided by Proposition~\ref{p.L2.weak} as soon as the spatial averages of $\nabla u$ and $\g$ display non-trivial cancellations over large scales. This feature will be crucial to our subsequent arguments. 

\smallskip

Before turning to the proof of Proposition~\ref{p.integrate}, we recall the classical $H^2$ estimate for solutions of the heat equation. For simplicity, we state it using periodic boundary conditions in the space variable. We denote the corresponding function spaces by $H^1_\per(U)$, $H^1_{\pa,\per}(I\times U)$, etc. 

\begin{lemma}[$H^2$ estimate for the Cauchy problem]
\label{l.H2.heat}
There exists $C(d) < \infty$ such that, for every $n \in \N$, if $f \in L^2(\cud_n)$ and $u \in H^{1}_{\pa,\per}(\cud_n)$ solves the Cauchy problem
\begin{equation}
\label{e.H2.heat}
\Ll\{
\begin{aligned}
& (\partial_t-\Delta) u = f  \quad & \mbox{in}  \ & \cud_n, \\
& u = 0 \quad & \mbox{on} \ & \Ll\{-\frac{3^{2n}}{2}\Rr\}\times \cu_n,
\end{aligned}
\Rr.
\end{equation}
then $\nabla u \in H^1_{\pa,\per}(\cud_n)$ and 
\begin{equation}  \label{e.bndH2}
3^{-n} \|u\|_{\un H^1_{\pa,\per}(\cud_n)} + \|\nabla u\|_{\un {H}^1_{\pa,\per}(\cud_n)} \le C \|f\|_{\un L^2(\cud_n)}.
\end{equation}
\end{lemma}
\begin{proof}
By scaling, it suffices to prove the result for $n = 0$. For each $s \in I_0$, we test equation~\eqref{e.H2.heat} against the function $(t,x) \mapsto \1_{t < s} u(t,x)$ and get
\begin{equation*}  
\int_{\Ll( -\frac 1 2, s \Rr) \times \cu_0} \Ll( u \partial_t u + |\nabla u|^2 \Rr) = \int_{\Ll( -\frac 1 2, s \Rr) \times \cu_0} u f,
\end{equation*}
which implies
\begin{equation*}  
\frac 1 2 \|u(s,\cdot)\|_{L^2(\cu_0)}^2 + \int_{\Ll( -\frac 1 2, s \Rr) \times \cu_0} |\nabla u|^2 \le \|u\|_{L^2(\cud_0)} \, \|f\|_{L^2(\cud_0)}. 
\end{equation*}
Taking the supremum over $s \in I_0$, observing that 
\begin{equation*}  
\sup_{s \in I_0} \|u(s,\cdot)\|_{L^2(\cu_0)} \le \|u\|_{L^2(\cud_0)}
\end{equation*}
and using Young's inequality, we obtain
\begin{equation*}  
\|u\|_{L^2(\cud_0)} + \|\nabla u\|_{L^2(\cud_0)} \le C \|f\|_{L^2(\cud_0)}.
\end{equation*}
We now turn to the estimation of $\|\nabla^2 u\|_{L^2(\cud_0)}$. We first observe that by integration by parts, we have $\|\nabla^2 u\|_{L^2(\cud_0)} = \|\Delta u\|_{L^2(\cud_0)}$. Moreover, using \eqref{e.H2.heat}, we get
\begin{equation*}  
\|\Delta u\|_{L^2(\cud_0)}^2 =  \int_{\cud_0}\Delta u (f+\partial_t u) =\int_{\cud_0}f\Delta u  - \int_{\cud_0} \nabla u \cdot\partial_t\nabla u,
\end{equation*}
with 
\begin{equation*}  
 \int_{\cud_0} \nabla u \cdot\partial_t\nabla u = \frac 1 2 \Ll\|\nabla u \Ll( \tfrac 1 2, \cdot \Rr) \Rr\|_{L^2(\cu_0)}^2 \ge 0,
\end{equation*}
and therefore
\begin{equation}  
\label{e.bound.Deltau}
\|\Delta u\|_{L^2(\cu_0)} \le \|f\|_{L^2(\cu_0)}. 
\end{equation}
We also need bounds on the time derivatives of $u$ and $\nabla u$. Note that
\begin{align*}  
\|\partial_t\nabla u\|_{L^2(I_0;H^{-1}_\per(\cu_0))} 
& = \sup \Ll\{ \int_{\cud_0} \partial_t \nabla u \cdot  F \ : \ \|F\|_{L^2(I_0;H^1_\per(\cu_0))} \le 1 \Rr\} 
\\ 
& = \sup \Ll\{ \int_{\cud_0} \partial_t  u  \, \nabla \cdot F \ : \ \|F\|_{L^2(I_0;H^1_\per(\cu_0))} \le 1 \Rr\} \\
& \le \|\partial_t u\|_{L^2(\cud_0)},
\end{align*}
and we can estimate the $L^2$-norm of $\partial_t u$ using \eqref{e.H2.heat} and \eqref{e.bound.Deltau}:
\begin{equation*}
\|\partial_t u\|_{L^2(\cud_0)}=	\|f+\Delta u\|_{L^2(\cud_0)}\le 2 \| f \|_{L^2(\cud_0)}.
\end{equation*}
The obvious bound
\begin{equation*}  
\|\partial_t u\|_{L^2(I_0;H^{-1}_\per(\cu_0))} \le C \|\partial_t u\|_{L^2(\cud_0)}
\end{equation*}
completes the proof of \eqref{e.bndH2}.
\end{proof}

\begin{proof}[Proof of Proposition~\ref{p.integrate}]
By homogeneity, it suffices to show the result for $n = 0$. Let $\psi \in C^\infty_c(\R^{1+d})$ be a smooth function with compact support in $\cud_{-2}$ and such that $\int_{\R^{1+d}} \psi = 1$. We decompose the proof into three steps. 

\smallskip 

\emph{Step 1.} Let $\eta \in C^\infty_c(\cud_1)$ be a smooth function with compact support in $\cud_{3/4}$ and such that $\eta \equiv 1$ on $\cud_{1/2}$. In this step, we show that there exists a constant $C(d,\psi,\eta) < \infty$ such that, for every $u \in L^2(\cud_1)$,
\begin{equation}  
\label{e.stupid.estimate}
\|\eta\Ll(u - u \star \psi\Rr)\|_{\hat H^{-1}_\pa(\cud_{1})} \le C \Ll(\|\nabla u\|_{\hat H^{-1}_\pa(\cud_1)} + \|\g\|_{\hat H^{-1}_\pa(\cud_1)}\Rr),
\end{equation}
For each $n \in \N$, we set
\begin{equation*}  
\psi_n :=  3^{n(2+d)} \psi(3^{2n} t, 3^{n} x).
\end{equation*}
Since the function $\psi_{-1} - \psi_0$ is compactly supported and of mean zero, we can use e.g.\ \cite[Lemma~5.7]{AKMBook} (in $1+d$ dimensions) to rewrite it in the form
\begin{equation*}  
\psi_{-1} - \psi_0 = \partial_t \h^\circ + \nabla \cdot \h',
\end{equation*}
where $\h = (\h^\circ, \h') \in C^\infty_c(\R^{1+d} ; \R \times \Rd)$ is supported in $\cud_{-2}$ (with $\h^\circ$ taking values in $\R$ and $\h'$ in $\R^d$). 
For each $n \in \N$, we denote
\begin{equation*}  
\h_n(t,x) = (\h^\circ_n, \h'_n)(t,x) :=  3^{n(2+d)} \h(3^{2n} t, 3^{n} x),
\end{equation*}
so that
\begin{equation*}  
\psi_{n-1} - \psi_n = 3^{-2n}  \partial_t \h^\circ_n + 3^{-n} \, \nabla \cdot \h'_n.
\end{equation*}
Since $u \star \psi_n \to u$ in $L^2(\cud_{3/4})$, we can use the triangle inequality to bound
\begin{equation}  
\label{e.triangle.H-1}
\|\eta \Ll(u - u \star \psi\Rr)\|_{\hat H^{-1}_\pa(\cud_{1})} \le \sum_{n = 0}^{+\infty} \|\eta \Ll(u \star \psi_{n-1} - u \star \psi_n\Rr)\|_{\hat H^{-1}_\pa(\cud_{1})}.
\end{equation}
We next observe that, for every $z \in \cud_{3/4}$, 
\begin{align*}  
(u \star \psi_{n-1} - u \star \psi_n)(z) & = \int_{\cud_{1}} u(y) \Ll(3^{-2n}  \partial_t \h^\circ_n + 3^{-n} \, \nabla \cdot \h'_n\Rr)(z-y) \, dy \\
& = \int_{\cud_{1}} \Ll(3^{-2n} \partial_t u(y)  \h^\circ_n(z-y) + 3^{-n} \, \nabla u(y) \cdot \h'_n(z-y)\Rr) \, dy \\
& =  3^{-n} \int_{\cud_1}  \nabla u(y) \cdot \h'_n(z-y) \, dy  \\
& \qquad \qquad  \qquad- 3^{-2n} \int_{\cud_{1}}  \g(y)  \cdot \nabla \h^\circ_n(z-y) \, dy ,
\end{align*}
We fix $f \in H^1_{\pa}(\cud_{1})$, set $\td f := \eta f$, and compute
\begin{equation*}  
\int_{\cud_{1}} \eta (u \star \psi_{n-1} - u \star \psi_n) f =3^{-n} \int_{\cud_1} \nabla u \cdot (\td f \ast \h_n')\\
-  3^{-2n} \int_{\cud_1} \g \cdot (\td f \ast \nabla \h_n^\circ) .
\end{equation*}
One can check that there exists a constant $C(d, \psi,\h) < \infty$ such that
\begin{equation*}  
\|\td f\ast \h_n'\|_{H^1_\pa(\cud_1)} + 3^{-n} \|\td f\ast \nabla \h_n^\circ\|_{H^1_\pa(\cud_1)}  \le C  \|f\|_{H^1_\pa(\cud_1)}.
\end{equation*}
Summarizing, we have thus shown that
\begin{equation*}  
\|\eta \Ll(u \star \psi_{n-1} - u \star \psi_n\Rr)\|_{\hat H^{-1}_\pa(\cud_{1})} \le C 3^{-n} \Ll( \|\nabla u\|_{\hat H^{-1}_\pa(\cud_1)} + \|\g\|_{\hat H^{-1}_\pa(\cud_1)} \Rr) .
\end{equation*}
Summing over $n \in \N$ in \eqref{e.triangle.H-1}, we obtain \eqref{e.stupid.estimate}.

\smallskip

\emph{Step 2.} Define
\begin{equation*}  
c(u) := u \star \psi(0).
\end{equation*}
In this step, we show that there exists a constant $C(d,\psi,\eta) < \infty$ such that 
\begin{equation}
\label{e.cleanup.estimate}
\|\eta \Ll(u - c(u)\Rr)\|_{\hat H^{-1}_\pa(\cud_{1})} \le C \Ll(\|\nabla u\|_{\hat H^{-1}_\pa(\cud_1)} + \|\g\|_{\hat H^{-1}_\pa(\cud_1)}\Rr).
\end{equation}
This is an immediate consequence of the fact that there exists a constant $C(d,\psi) < \infty$ such that, for every $y,z \in \cud_{3/4}$, 
\begin{equation}  
\label{e.pointwise.est}
\Ll|u \star \psi(z) - u \star \psi(y)\Rr| \le C \Ll(\|\nabla u\|_{\hat H^{-1}_\pa(\cud_1)} + \|\g\|_{\hat H^{-1}_\pa(\cud_1)}\Rr).
\end{equation}
The proof of \eqref{e.pointwise.est} is very similar to the previous step, only simpler: we represent the function $\psi(\cdot - z) - \psi(\cdot - y)$ in the form
\begin{equation*}  
\partial_t \td \h^\circ + \nabla \cdot \td \h',
\end{equation*}
with $(\td \h^\circ, \td \h') \in C^\infty_c(\cud_1;\R\times \Rd)$, and then obtain \eqref{e.pointwise.est} thanks to an integration by parts.

\smallskip

\emph{Step 3.} 
For concision, we write 
\begin{equation*}  
\td u := u - c(u). 
\end{equation*}
Let $\chi \in C^\infty_c(\cud_1)$ be a smooth function with compact support in $\cud_{1/2}$ and such that $\chi \equiv 1$ on $\cud_0$. 
In this step, we show that there exists a constant $C(d, \psi,\eta,\chi) < \infty$ such that
\begin{equation} 
\label{e.thats.your.L2.bound}
\|\chi \td u \|_{L^2(\cud_1)} \le C  \Ll( \|\nabla u\|_{\hat H^{-1}_\pa(\cud_1)} + \|\g\|_{\hat H^{-1}_\pa(\cud_1)} \Rr) .
\end{equation}
Let $w \in H^1_{\pa,\per}(\cud_1)$ solve the Cauchy problem
\begin{equation}  
\label{e.L2.est.eq.v}
\Ll\{
\begin{aligned}
& (\partial_t - \Delta) w = \chi \td u  \quad & \mbox{in}  \ & \cud_1, \\
& w = 0 \quad & \mbox{on} \ & \Ll\{-\tfrac{9}{2}\Rr\}\times \cu_1.
\end{aligned}
\Rr.
\end{equation}
By Lemma~\ref{l.H2.heat}, there exists a constant $C(d) < \infty$ such that 
\begin{equation}  
\label{e.from.lemma.H2heat}
  \| w\|_{H^1_{\pa,\per}(\cud_1)} + \|\nabla w\|_{H^1_{\pa,\per}(\cud_1)} \le C \|\chi \td u\|_{L^2(\cud_1)}.
\end{equation}
Testing the equation \eqref{e.L2.est.eq.v} against $\chi \td u$ and integrating by parts gives
\begin{align*}  
\|\chi \td u\|_{L^2(\cud_1)}^2 & = \int_{\cud_1} \Ll( \nabla w \cdot \nabla \Ll( \chi \td u \Rr) - w \,  \partial_t \Ll( \chi \td u \Rr)  \Rr) \\
& = \int_{\cud_1} \Ll(\nabla w \cdot \nabla \Ll( \chi \td u \Rr) + \nabla \Ll( \chi w \Rr) \cdot \g - \partial_t \chi \, w \td u  \Rr) \\
& \le \|\nabla w\|_{H^1_{\per,\pa(\cud_1)}} \Ll( \|\nabla u\|_{\hat H^{-1}_\pa(\cud_1)} + \|\td u \nabla \chi\|_{\hat H^{-1}_\pa(\cud_1)} + \|\g\|_{\hat H^{-1}_\pa(\cud_1)} \Rr)  \\
& \qquad + \|w\|_{H^1_{\per,\pa(\cud_1)}} \Ll( \|\g\|_{\hat H^{-1}_\pa(\cud_1)} + \|\td u \nabla \chi\|_{\hat H^{-1}_\pa(\cud_1)} \Rr) .
\end{align*}
Using the result of the previous step and \eqref{e.from.lemma.H2heat}, we obtain \eqref{e.thats.your.L2.bound}. This completes the proof of Proposition~\ref{p.integrate}, since
\begin{equation*}  
\|u - (u)_{\cud_0}\|_{L^2(\cud_0)}  \le \|u - c(u)\|_{L^2(\cud_0)} \le \|\chi \Ll( u - c(u) \Rr) \|_{L^2(\cud_0)}. \qedhere
\end{equation*}
\end{proof}

Finally, we prove a Caccioppoli-type inequality for elements of $\S(\cud_n)$. 

\begin{proposition}
\label{p.caccio} 
There exists a constant $C(d,\Lambda) < \infty$ such that, for every $n \in \N$ and $S \in \S(\cud_{n+1})$, 
\begin{equation*}  
\|S\|_{\un L^2(\cud_n)} \le C 3^{-n} \|S\|_{\un {\hat H}^{-1}_\pa(\cud_{n+1})}.
\end{equation*}
\end{proposition}

In order to prove this result, we first describe more explicitly the structure of elements of $\hat H^{-1}_\pa(I\times U)$.

\begin{lemma}[Identification of $\hat H^{-1}_\pa$]
\label{l.identif.H-1par}
There exists a constant $C(I,U,d) < \infty$ and, for each $u^* \in L^2(I\times U)$, a pair $(w,w^*) \in L^2(I;H^1_0(U)) \times L^2(I;H^{-1}(U))$ such that
\begin{equation}  
\label{e.u*.decomp}
u^* = \partial_t w + w^*,
\end{equation}
with
\begin{equation}  
\label{e.estim.w.w*}
\|w\|_{L^2(I;H^1(U))} + \|w^*\|_{L^2(I;H^{-1}(U))} \le C \|u^*\|_{{\hat H}^{-1}_\pa(I\times U)}.
\end{equation}
Let $V \subset V' \subset  I\times U$ be subdomains of $I \times U$ such that $\bar V \subset V'$. If $u^*$ vanishes outside of $V$, then there exists a pair $(w,w^*) \in L^2(I;H^1_0(U)) \times L^2(I;H^{-1}(U))$ satisfying \eqref{e.u*.decomp}-\eqref{e.estim.w.w*} for a constant $C(V,V',I,U,d) < \infty$, and such that $w$ and $w^*$ vanish outside of $V'$.
\end{lemma}
\begin{proof}
Denote by $\Delta_{\cu_n}^{-1}$ the solution operator for the Laplacian in ${\cu_n}$ with null Dirichlet boundary condition. We observe that
\begin{equation*}  
(u,v) \mapsto \fint_{I\times U} \Ll(|U|^{-\frac 2 d}uv + \nabla u \cdot \nabla v + \nabla \Delta^{-1}_U \partial_t u \cdot \nabla \Delta^{-1}_U \partial_t v\Rr)
\end{equation*}
is a scalar product for the Hilbert space $H^1_\pa(I\times U)$. By the Riesz representation theorem, there exists a unique $u \in H^1_\pa(I\times U)$ such that, for every $v \in H^1_\pa(I\times U)$,
\begin{equation*}  
\fint_{I\times U} u^*  v = \fint_{I\times U} \Ll(|U|^{-\frac 2 d}uv + \nabla u \cdot \nabla v + \nabla \Delta^{-1}_U \partial_t u \cdot \nabla \Delta^{-1}_U \partial_t v\Rr),
\end{equation*}
and moreover, by testing this identity with $v = u$, we obtain
\begin{equation}  
\label{e.compare.u.u^*.norms}
\|u\|_{H^1_\pa(I\times U)} \le C \|u^*\|_{{\hat H}^{-1}_\pa(I\times U)}.
\end{equation}
We set
\begin{equation*}  
w := \Delta_U^{-1} \partial_t u \quad \text{and} \quad w^* := |U|^{-\frac 2 d} u - \Delta u.
\end{equation*}
The estimate \eqref{e.estim.w.w*} follows from \eqref{e.compare.u.u^*.norms}. 
For $v \in H^1_\pa(I\times U)$ with compact support in $I\times U$, we have
\begin{equation*}  
\fint_{I\times U} u^*  v = \fint_{I\times U} \Ll(w^* v  + \partial_t w  \, v \Rr) .
\end{equation*}
Since $u^* \in L^2(I\times U)$, we can argue by density to infer that $w^* + \partial_t w \in L^2(I\times U)$. The identity above then implies \eqref{e.u*.decomp}.

\smallskip

If $u^*$ vanishes outside of $V \subset I\times U$, then we select a smooth cutoff function $\eta$ such that $\eta \equiv 1$ on $V$ and $\eta \equiv 0$ outside of $V'$, and we write
\begin{align*}  
u^* & = \eta u^* = \eta (\partial_t w + w^*) \\
& = \partial_t (\eta w) + \eta w^* - w \partial_t \eta.
\end{align*}
This decomposition yields the second part of the statement, by standard comparisons of norms.
\end{proof}

\begin{proof}[Proof of Proposition~\ref{p.caccio}]
By scaling, we may fix $n = 0$. In order to localize an element $S \in \S(\cud_1)$ into an element of $\C_0(\cud_{1})$ and thus be able to use the orthogonality property in the definition of the set $\S(\cud_{1})$, see \eqref{e.def.S}, we introduce a version of the Helmholtz-Hodge decomposition of $S$ which is adapted to the parabolic setting. In order to minimize difficulties due to boundary conditions, we work with functions which are periodic in the space variable. 
In the course of the proof, we will use the elementary variant of Proposition~\ref{p.parabolic.min} for the standard heat operator with space-periodic boundary condition. 

\smallskip

We decompose the proof into four steps. The first two steps are devoted to the construction of the Helmholtz-Hodge decomposition of $S$, and its estimation in relevant norms. The last step uses this representation to localize $S$ to an element of $\C_0(\cud_1)$ and concludes the proof.

\smallskip

\emph{Step 1.}
We write $S = (\nabla u,\g) \in \S(\cud_1)$. We recall that since $\S(\cud_1) \subset \C(\cud_1)$, we have $\partial_t u = \nabla \cdot \g$, see \eqref{e.def.candidate*}. The function $u$ is determined up to an additive constant, which we fix so that
\begin{equation*}  
\fint_{\cud_{3/4}} u = 0.
\end{equation*}
Let $\eta \in C_c^\infty(\cud_1)$ be a smooth function with compact support in $\cud_{3/4}$, such that $0 \le \eta \le 1$ and $\eta \equiv 1$ on $\cud_{1/2}$. We set
\begin{equation*}  
\td u := \eta u, \quad \text{and} \quad \forall j \in \{1,\ldots, d\}, \ \td \g_j := \eta \g_j.
\end{equation*}
For each $j \in \{1,\ldots,d\}$, let $T_{0j} \in H^{1}_{\pa,\per}(\cud_1)$ be the unique solution of
\begin{equation}  
\label{e.def.T0j}
\Ll\{
\begin{aligned}
& (\partial_t - \Delta) T_{0j}  = \td \g_j - \partial_{x_j} \td u  \quad & \mbox{in}  \ & \cud_1, \\
& T_{0j}  = 0 \quad & \mbox{on} \ & \Ll\{-\tfrac{9}{2}\Rr\}\times \cu_1.
\end{aligned}
\Rr.
\end{equation}
By Lemma~\ref{l.identif.H-1par} there exist $(w_j, w_j^*) \in L^2(I_1;H^1_0(\cu_1)) \times  L^2(I_1;H^{-1}(\cu_1))$ which vanish in a neighborhood of $\partial \cud_1$ and satisfy
\begin{equation*}  
\td \g_j - \partial_{x_j} \td u  = \partial_t w_j +  w^*_j,
\end{equation*}
with 
\begin{equation*}  
\|w_j\|_{L^2(I_1;H^1(\cu_1))} + \|w^*_j\|_{L^2(I_1;H^{-1}(\cu_1))} \le C \|\td \g_j - \partial_{x_j} \td u \|_{{\hat H}^{-1}_\pa(\cud_1)}.
\end{equation*}
Since $(w_j,w^*_j)$ vanish in a neighborhood of $\partial \cud_1$, we can interpret this pair as an element of
\begin{equation*}  
L^2(I_1;H^1_\per(\cu_1)) \times  L^2(I_1;H^{-1}_\per(\cu_1)),
\end{equation*}
with the estimate
\begin{equation*}  
\|w_j\|_{L^2(I_1;H_\per^1(\cu_1))} + \|w^*_j\|_{L^2(I_1;H_\per^{-1}(\cu_1))} \le C \|\td \g_j - \partial_{x_j} \td u \|_{{\hat H}^{-1}_\pa(\cud_1)}.
\end{equation*}
Since $\td \g_j - \partial_{x_j} \td u \in L^2(\cud_1)$, it is clear that $\partial_t w_j \in L^2(I_1; H^{-1}_\per(\cud_1))$, and we therefore deduce that $w_j \in H^1_{\per,\pa,\sqcup}(\cud_1)$. 
Moreover, by Proposition~\ref{p.parabolic.min} applied to the standard heat operator, there exist a constant $C(d) < \infty$ and $T'_{0j} \in H^1_{\per,\pa,\sqcup}(\cud_1)$ such that 
\begin{equation*}  
(\partial_t - \Delta) T'_{0j} = -\Delta w_j + w^*_j,
\end{equation*}
with
\begin{align*}  
\|T'_{0j}\|_{L^2(\cud_1)} + \|\nabla T'_{0j}\|_{L^2(\cud_1)} & \le C \Ll(\|w_j\|_{L^2(I_1;H^1(\cu_1))}  +  \|w^*_j\|_{L^2(I_1;H^{-1}(\cu_1))} \Rr) \\
& \le C \|\td \g_j - \partial_{x_j} \td u \|_{{\hat H}^{-1}_\pa(\cud_1)}.
\end{align*}
We thus have 
\begin{equation*}  
(\partial_t - \Delta) (w_j + T'_{0j} - T_{0j}) = 0,
\end{equation*}
with $w_j + T'_{0j} - T_{0j} \in H^1_{\per,\pa, \sqcup}(\cud_1)$. Therefore,
\begin{equation*}  
T_{0j} = w_j + T'_{0j},
\end{equation*}
and 
\begin{equation*}  
\|T_{0j}\|_{L^2(\cud_1)} + \|\nabla T_{0j}\|_{L^2(\cud_1)} \le C \|\td \g_j - \partial_{x_j} \td u\|_{{\hat H}^{-1}_\pa(\cud_1)}.
\end{equation*}
It is clear that $\|\td \g_j\|_{H^{-1}_\pa(\cud_1)} \le C \|\g_j\|_{H^{-1}_\pa(\cud_1)}$. 
By Proposition~\ref{p.integrate}, Remark~\ref{r.change.interior} and the comparison 
$$
\|\td u\|_{H^{-1}_\pa(\cud_1)} \le \|\td u\|_{L^2(\cud_1)} \le  \|u\|_{L^2(\cud_{3/4})} = \|u - (u)_{\cud_{3/4}}\|_{L^2(\cud_{3/4})} ,
$$
we obtain
\begin{equation}  
\label{e.estim.T0j}
\|T_{0j}\|_{L^2(\cud_1)} + \|\nabla T_{0j}\|_{L^2(\cud_1)} \le C \Ll(\|\nabla u\|_{{\hat H}^{-1}_\pa(\cud_1)} + \|\g\|_{{\hat H}^{-1}_\pa(\cud_1)} \Rr).
\end{equation}

\smallskip

\emph{Step 2.}
For each $i,j \in \{1,\ldots,d\}$, we define $T_{ij} \in H^{1}_{\pa,\per}(\cud_1)$ as the solution of
\begin{equation*}  
\Ll\{
\begin{aligned}
& (\partial_t - \Delta) T_{ij} = \partial_{x_i} \td \g_j - \partial_{x_j} \td \g_i \quad & \mbox{in}  \ & \cud_1, \\
& T_{ij} = 0 \quad & \mbox{on} \ & \Ll\{-\tfrac{9}{2}\Rr\}\times \cu_1.
\end{aligned}
\Rr.
\end{equation*}
The solution $T_{ij}$ is well-defined since the right-hand side belongs to $L^2(I_1;H^{-1}_\per(\cu_1))$. We now estimate the $L^2$ norm of $T_{ij}$ using Lemma~\ref{l.H2.heat} and duality. We define $\phi_{ij} \in H^1_{\pa,\per}(\cud_1)$ the solution of the backwards heat equation
\begin{equation}  
\label{e.def.phiij}
\Ll\{
\begin{aligned}
& -\Ll(\partial_t + \Delta\Rr) \phi_{ij} = T_{ij} \quad & \mbox{in}  \ & \cud_1, \\
& \phi_{ij} = 0 \quad & \mbox{on} \ & \Ll\{\tfrac{9}{2}\Rr\}\times \cu_1.
\end{aligned}
\Rr.
\end{equation}
By Lemma~\ref{l.H2.heat}, we have
\begin{equation*}  
\|\nabla \phi_{ij}\|_{H^1_{\pa,\per}(\cud_1)} \le C\|T_{ij}\|_{L^2(\cud_1)}.
\end{equation*}
Testing the equation \eqref{e.def.phiij} against $T_{ij}$, we get
\begin{equation*}  
\|T_{ij}\|_{L^2(\cud_1)}^2 = \int_{\cud_1} \Ll(\td \g_i \, \partial_{x_j} \phi_{ij} - \td \g_j \, \partial_{x_i} \phi_{ij} \Rr)\le \|\nabla \phi_{ij}\|_{H^1_{\pa,\per}(\cud_1)}  \, \|\td \g\|_{{\hat H}^{-1}_\pa(\cud_1)} .
\end{equation*}
Combining the two previous displays yields
\begin{equation}
\label{e.estim.Tij}
\|T_{ij}\|_{L^2(\cud_1)} \le  C \| \td \g\|_{{\hat H}^{-1}_\pa(\cud_1)} \le  \| \g\|_{{\hat H}^{-1}_\pa(\cud_1)}.
\end{equation}

\smallskip

\emph{Step 3.}
For notational convenience,  for each $j \in \{1,\ldots, d\}$, we denote $T_{j0} := - T_{0j}$, $\partial_0 := \partial_t$, $\partial_j := \partial_{x_j}$, 
\begin{equation*}  
R_0 := \td u - \sum_{j = 1}^d \partial_j T_{0j},
\end{equation*}
and, for each $i \in \{1,\ldots,d\}$,
\begin{equation*}  
R_i := \td \g_i  - \sum_{j = 0}^d \partial_j T_{ij} .
\end{equation*}
In this step, we show that
\begin{equation}  
\label{e.estim.R0}
\|R_0 \|_{L^2(\cud_1)} + \|\nabla R_0\|_{L^2(\cud_1)} \le C \|\nabla u\|_{{{\hat H}^{-1}_\pa(\cud_1)}} + C \|\g\|_{{\hat H}^{-1}_\pa(\cud_1)}.
\end{equation}
and that for every $i \in \{1,\ldots,d\}$, 
\begin{equation}  
\label{e.estim.R1}
\|R_i\|_{L^2(\cud_1)} \le C \|\nabla u\|_{{{\hat H}^{-1}_\pa(\cud_1)}} + C \|\g\|_{{\hat H}^{-1}_\pa(\cud_1)}.
\end{equation}
Recalling that $\partial_t u = \nabla \cdot \g$, we note that
\begin{align*}  
(\partial_t - \Delta) R_0
 = (\partial_t - \Delta) \td u - \sum_{j = 1}^d \partial_j \Ll(  \td \g_j - \partial_j \td u \Rr) = u \, \partial_t \eta + \g \cdot \nabla \eta ,
\end{align*}
and, for each $i \in \{1,\ldots,d\}$,
\begin{align*}  
 (\partial_t - \Delta) R_i 
&  = (\partial_t - \Delta) \td \g_i - \partial_t \Ll( \td  \g_i - \partial_i \td u \Rr) - \sum_{j = 1}^d \partial_j \Ll( \partial_i \td \g_j - \partial_j \td \g_i \Rr) \\
& = (\partial_i u) (\partial_t \eta) + (\partial_t u)(\partial_i \eta) + (\nabla \cdot \g) (\partial_i \eta)  + (\partial_i \g) \cdot \nabla \eta.
\end{align*}
Moreover, it is clear from their definitions that $T_{0j}$ and $T_{ij}$ vanish in a neighborhood of the initial time slice $\{-\frac 9 2\} \times \cud_1$. The estimates \eqref{e.estim.R0} and \eqref{e.estim.R1} are thus obtained by following the steps to the derivations of \eqref{e.estim.T0j} and \eqref{e.estim.Tij} respectively.

\smallskip

\emph{Step 4.}
We now select a cutoff function $\chi \in C^\infty_c(\cud_1)$ such that $\chi \equiv 1$ on $\cud_0$ and $\chi \equiv 0$ outside of $\cud_{1/2}$, and observe that
\begin{equation*}  
\begin{pmatrix}
\displaystyle{\sum_{j = 1}^d \nabla \partial_j(\chi^4 T_{0j})} \\
\displaystyle{\sum_{j = 0}^d \partial_j (\chi^4 T_{ij})} 
\end{pmatrix}
\in \C_0(\cud_{1}),
\end{equation*}
where we understand that the second component above denotes a $d$-dimensional vector field with components indexed by $i \in \{1,\ldots, d\}$. By the definition of $\S(\cud_1)$, we deduce that
\begin{equation}  
\label{e.orthogonality}
\int_{\cud_1} 
\begin{pmatrix}
\displaystyle{\sum_{j = 1}^d \nabla \partial_j(\chi^4 T_{0j})} \\
\displaystyle{\sum_{j = 0}^d \partial_j (\chi^4 T_{ij})} 
\end{pmatrix}
\cdot \Abf
\begin{pmatrix}
\nabla u
\\
\g
\end{pmatrix}
= 0.
\end{equation}
In the display above, the first vector is of dimension $2d$: the gradient appearing on the first raw carries the first $d$ components, while the other $d$ components are represented by the second raw and indexed by $i \in \{1,\ldots,d\}$.
Applying the chain rule in the identity \eqref{e.orthogonality} yields a number of terms, one of which is
\begin{align*}  
\int_{\cud_1} 
\chi^4
\begin{pmatrix}
\displaystyle{\sum_{j = 1}^d \nabla \partial_j T_{0j}} \\
\displaystyle{\sum_{j = 0}^d \partial_j  T_{ij}} 
\end{pmatrix}
\cdot \Abf
\begin{pmatrix}
\nabla u
\\
\g
\end{pmatrix} 
& = 
\int_{\cud_1} 
\chi^4
\begin{pmatrix}
\nabla \td u
\\
\td \g
\end{pmatrix}
\cdot \Abf
\begin{pmatrix}
\nabla u
\\
\g
\end{pmatrix}
+ 
\int_{\cud_1} 
\chi^4
\begin{pmatrix}
\nabla R_0
\\
R_i
\end{pmatrix}
\cdot \Abf
\begin{pmatrix}
\nabla u
\\
\g
\end{pmatrix}
\\
&
=
\int_{\cud_1} 
\chi^4
\begin{pmatrix}
\nabla u
\\
\g
\end{pmatrix}
\cdot \Abf
\begin{pmatrix}
\nabla u
\\
\g
\end{pmatrix}
+
\int_{\cud_1} 
\chi^4
\begin{pmatrix}
\nabla R_0
\\
R_i
\end{pmatrix}
\cdot \Abf
\begin{pmatrix}
\nabla u
\\
\g
\end{pmatrix}
.
\end{align*}
We are interested in estimating the first term in this sum. By the uniform boundedness of $\Abf$, the absolute value of the second term in this sum is bounded by a constant times
\begin{equation*}  
\Ll( \|\nabla R_0\|_{L^2(\cud_1)} + \sum_{i = 1}^d \|R_i\|_{L^2(\cud_1)} \Rr) 
\Ll( \int_{\cud_1} 
\chi^4
\begin{pmatrix}
\nabla u
\\
\g
\end{pmatrix}
\cdot \Abf
\begin{pmatrix}
\nabla u
\\
\g
\end{pmatrix} \Rr) ^{\frac 1 2}.
\end{equation*}
When applying the chain rule in the identity \eqref{e.orthogonality}, the leftover terms are 
\begin{equation*}  
4\int_{\cud_1} \chi^2
\begin{pmatrix}
\displaystyle{\sum_{j = 1}^d T_{0j}\Ll(3 \nabla \chi\partial_j \chi + \chi \nabla \partial_j \chi \Rr) + 2\nabla T_{0j} \chi \partial_j \chi} \\
\displaystyle{\sum_{j = 0}^d T_{ij}\chi \partial_j \chi  } 
\end{pmatrix}
\cdot \Abf
\begin{pmatrix}
\nabla u
\\
\g
\end{pmatrix} .
\end{equation*}
Using once more the uniform boundedness of $\Abf$, we obtain that the absolute value of the quantity above  is bounded by a constant times
\begin{equation*}  
\Ll( \sum_{j=1}^d \Ll(\|T_{0j}\|_{L^2(\cud_1)} + \| \nabla T_{0j}\|_{L^2(\cud_1)} \Rr) + \sum_{j = 0}^d \|T_{ij}\|_{L^2(\cud_1)} \Rr)\Ll( \int_{\cud_1} 
\chi^4
\begin{pmatrix}
\nabla u
\\
\g
\end{pmatrix}
\cdot \Abf
\begin{pmatrix}
\nabla u
\\
\g
\end{pmatrix} \Rr) ^{\frac 1 2}.
\end{equation*}
Combining the previous displays with the estimates \eqref{e.estim.T0j}, \eqref{e.estim.Tij}, \eqref{e.estim.R0} and \eqref{e.estim.R1}, we arrive at
\begin{equation*}  
\Ll( \int_{\cud_1} 
\chi^4
\begin{pmatrix}
\nabla u
\\
\g
\end{pmatrix}
\cdot \Abf
\begin{pmatrix}
\nabla u
\\
\g
\end{pmatrix} \Rr) ^{\frac 1 2}
\le 
C \Ll( \|\nabla u\|_{{\hat H}^{-1}_\pa(\cud_1)} + \|\g\|_{{\hat H}^{-1}_\pa(\cud_1)} \Rr) .
\end{equation*}
By the uniform ellipticity of $\Abf$, the left side is an upper bound for $\|S\|_{L^2(\cud_0)}$, up to a multiplicative constant, and therefore the proof is complete.
\end{proof}

\section{Convergence of subadditive quantities}
\label{s.iteration}

In this section, we obtain an algebraic rate of convergence for the limits of the subadditive quantities by adapting the approach of~\cite{AS,AM}, following the presentation of~\cite[Chapter 2]{AKMBook}.

\smallskip

We let $\Abfh$ be the $2d$-by-$2d$ matrix characterized by the limit
\begin{equation}  
\label{e.defAbfh}
\lim_{n \to \infty} \E \Ll[ \mu(\cud_n,X) \Rr] = \frac 1 2 X\cdot \Abfh X.
\end{equation}
Note that the existence of the limit on the left side follows from the subadditivity of $\mu(\cdot,X) = J(\cdot,X,0)$ and stationarity, which together ensure that $\E \Ll[ \mu(\cud_n,X) \Rr]$ is a nonincreasing sequence. The fact that $X\mapsto \mu(\cud_n,X)$ is quadratic ensures that the limit is also quadratic in $X$ and can therefore be represented by a matrix. Moreover, by Lemma~\ref{l.basicJ}, there exists $C(\Lambda) < \infty$ such that 
\begin{equation}
\label{e.abigbarbounds}
\frac 1 C \, I_{2d} \le \Abfh \le C \, I_{2d},
\end{equation}
where $I_{2d}$ is the $2d$-by-$2d$ identity matrix. 
It is convenient to define 
\begin{equation*} \label{}
\overline{J}(X,X^*):=  \frac12 X \cdot \Abfh  X +  \frac12 X^* \cdot \Abfh ^{-1} X^* - X \cdot X^*.
\end{equation*}
The goal of this section is to prove the following theorem.

\begin{theorem}[Convergence of $J$]
\label{t.subadd}
There exist an exponent $\be(d,\Lambda) > 0$ and, for each $s \in (0,2+d)$, a constant $C(s,d,\Lambda) < \infty$ such that, for every $X,X^*\in B_1$ and $n\in\N$,  we have
\begin{equation}
\label{e.subadderror}
 \left| 
 J(\cud_n,X,X^*)
 -
 \overline{J}(X,X^*)
  \right| 
\leq C 3^{-n \be(2+d-s)} + \O_1\left( C3^{-ns} \right). 
\end{equation}
\end{theorem}
The next lemma (which should be compared to~\cite[Lemma 2.7]{AKMBook}) allows us to reduce  Theorem~\ref{t.subadd} to an estimate on the quantity~$J(\cud_n, X, \Abfh X)$. Note that, in view of Lemma~\ref{l.Jsplitting}, a control on the size of~$\inf_{X^*} J(V,X,X^*)$ can be interpreted as information on the ``convex duality defect'' between the quantities $\mu$ and $\mu^*$, quantifying how close these functions are to a convex dual pair.

\begin{lemma}[reduction to minimal set]
\label{l.minimalset}
For each $\Gamma\geq 1$, there exists a constant $C(\Gamma,d,\Lambda)<\infty$ such that, for every $2d$-by-$2d$ symmetric matrix $\tilde{\Abf}$ satisfying
\begin{equation}
 \label{e.Abf.Gamma}
\Gamma^{-1} \, I_{2d} \leq \tilde{\Abf} \leq \Gamma\, I_{2d}
\end{equation}
and every parabolic cylinder $V\subseteq\R^{d+1}$, we have 
\begin{multline}
\label{e.minimalset}
\sup_{X,X^*\in B_1} 
\left|  
J(V,X,X^*) - 
\left( 
\frac12X
\cdot \tilde{\Abf}X
+ \frac12 X^*
\cdot \tilde{\Abf}^{-1} X^*
- X\cdot X^*
 \right) 
 \right|
\\
 \leq C \sup_{X  \in B_1} 
 \Ll(J \left(V,X, \tilde{\Abf} X \right) \Rr)^\frac 12 . 
\end{multline}
\end{lemma}
\begin{proof}
Since the domain~$V$ plays no role in the argument, we drop the explicit dependence on~$V$. 
Denote 
\begin{equation*} \label{}
\delta^2 :=  \sup_{X  \in B_1} J \left(X, \tilde{\Abf} X \right).
\end{equation*}
To avoid a conflict in the notation, we denote the Legendre-Fenchel transform (convex dual function) of $\mu$ by
\begin{equation*} \label{}
H(X^*):= \sup_{X\in\R^{2d}} \left( X \cdot X^* - \mu(X) \right).  
\end{equation*}
It is clear from~\eqref{e.ordering} that 
\begin{equation*} \label{}
H(X^*) \leq \mu^*(X^*). 
\end{equation*}
Thus, by~\eqref{e.Jsplitting}, for every $X\in B_1$, 
\begin{multline} 
\label{e.zerodelta2}
0\leq \mu(X) + H(\tilde{\Abf} X) - X \cdot \tilde{\Abf} X 
\\
\leq 
\mu(X) + \mu^*(\tilde{\Abf} X) - X \cdot \tilde{\Abf} X 
= J(X,\tilde{\Abf} X) \leq \delta^2. 
\end{multline}
This implies that, for every $X \in B_1$, 
\begin{equation} 
\label{e.calledthatdual}
\left| \mu^*(\tilde{\Abf} X) -  H(\tilde{\Abf} X) \right| \leq \delta^2.
\end{equation}
For each $X\in \R^{2d}$, the minimum of the map $X^* \mapsto \mu(X) + H(X^*) - X\cdot X^*$ is zero and it is achieved at $X^*$ for which~$X^* = \nabla \mu(X)$. By uniform convexity (quadratic response) and \eqref{e.zerodelta2}, we deduce that, for every $X \in B_1$,
\begin{equation} \label{}
\left| \tilde{\Abf}X - \nabla \mu(X) \right|^2 \leq C\delta^2. 
\end{equation}
Using the expression
\begin{equation*} \label{}
\mu(X) =  \frac12 X \cdot \nabla \mu(X) 
\end{equation*}
we obtain, for every $X\in B_1$, 
\begin{equation*} \label{}
\left| \frac12 X \cdot \tilde{\Abf}X - \mu(X) \right| \leq C\delta. 
\end{equation*}
From this, uniform convexity and~\eqref{e.Abf.Gamma}, we obtain, for every $X^*\in B_1$,
\begin{equation*} \label{}
\left| \frac12 X^* \cdot \tilde{\Abf}^{-1} X^* - H(X^*) \right| \leq C\delta. 
\end{equation*}
Hence by~\eqref{e.calledthatdual},~\eqref{e.Abf.Gamma} again,
\begin{equation*} \label{}
\left| \frac12 X^* \cdot \tilde{\Abf}^{-1} X^* - \mu^*(X^*) \right| \leq C\delta. 
\end{equation*}
The formula~\eqref{e.Jsplitting} now yields the lemma. 
\end{proof}

We decompose the estimate for $J(\cud_n,X, \Abfh X)$ into three steps. In the first step, we identify a convenient finite-volume approximation of the homogenized matrix $\Abfh$. We next control the expectation of $J(\cud_n,X,\Abfh X)$ in Subsection~\ref{ss.expec}. We finally use the subadditivity of $J$ in Subsection~\ref{ss.fluc} to deduce a control of the fluctuations of $J(\cud_n,X, \Abfh X)$, and complete the proof of Theorem~\ref{t.subadd}.

\subsection{The coarsened mapping}

Recall that $S(\cdot, V, X, X^*)$ denotes the unique maximizer in the definition of $J(V, X, X^*)$, see \eqref{e.compact.J}.
We let $\Abfh_V \in \R^{2d \times 2d}$ be the symmetric matrix such that, for every $X^* \in \R^{2d}$,
\begin{equation*}  
\E[J(V,0,X^*)] = \frac 1 2 X^* \cdot \Abfh_V^{-1} X^*.
\end{equation*}
By \eqref{e.JunifconvC11q}, there exists $C(d,\Lambda) < \infty$ such that 
\begin{equation*}  
\frac 1 C \, I_{2d} \le \Abfh_{V} \le C \, I_{2d},
\end{equation*}
and by \eqref{e.JderX*}, 
\begin{equation*}  
\E \Ll[ \fint_V S(\cdot,V,0,X^*) \Rr] = \Abfh_V^{-1} X^*.
\end{equation*}
Recalling also \eqref{e.spatav.X0} and the linearity of the mapping $(X,X^*) \mapsto S(\cdot,V,X,X^*)$, we thus see that the matrix $\Abfh_V$ is such that, for every $X \in \R^{2d}$,
\begin{equation}
\label{e.spatav.Abfh}
\E \Ll[ \fint_V S(\cdot,V,X,\Abfh_V X) \Rr] = 0.
\end{equation}
We note that by Lemmas~\ref{l.Jsplitting} and \ref{l.basicJ}, for each $X \in \R^{2d}$, the mapping 
\begin{equation*}  
X^* \mapsto \E[J(V,X,X^*)] = \E[\mu(V,X)] + \E[\mu^*(V,X^*)] - X \cdot X^*
\end{equation*}
is uniformly convex, and achieves its unique minimum at $X^*$ satisfying
\begin{equation*}  
\E[\nabla_{X^*} \mu^*(V,X^*)] = X.
\end{equation*}
Moreover, the latter condition is equivalent to $X^* = \Abfh_V X$. We thus deduce that for every $X,X^* \in \R^{2d}$,
\begin{equation}
\label{e.quad.resp.Abfh}
\E[J(V,X,\Abfh_V X)] \le \E[J(V,X,X^*)] \le \E[J(V,X,\Abfh_V X)] + C |X^* - \Abfh_V X|^2.
\end{equation}
We use the shorthand notation 
\begin{equation}  
\label{e.def.abfh}
\Abfh_n := \Abfh_{\cud_n}.
\end{equation}

\subsection{Control of the expectation of \texorpdfstring{$J$}{J}}
\label{ss.expec}

The goal of this subsection is to prove the following proposition.

\begin{proposition}[{Decay of $\E[J]$}]
\label{p.control.expec}
There exist~$\be(d,\Lambda) > 0$ and~$C(d,\Lambda) < \infty$ such that, for every $n\in\N$ and $X\in B_1(\R^{2d})$,
\begin{equation}
\label{e.control.expec}
\E \left[ J(\cud_n,X,\Abfh  X)\right] \leq C3^{-\beta n}. 
\end{equation}
\end{proposition}
The main step to prove this result is to control the size of $J$ near $(X,\Abfh X)$ in terms of the expected ``additivity defect'' of $J$ between successive triadic scales. We measure the latter using the quantity 
\begin{equation}
\label{e.def.taun}
\tau_n := \sup_{X, X^* \in B_1} \Ll( \E[J(\cud_{n},X, X^*)]  - \E[J(\cud_{n+1},X, X^*)] \Rr) .
\end{equation}

\begin{proposition}
\label{p.Jsmall}
There exist $\alpha(d) < \infty$ and $C(d,\Lambda) < \infty$ such that, for every $n \in \N$ and $X \in B_1(\R^{2d})$,
\begin{equation*}  
\E[J(\cud_{n},X, \Abfh_n X)] \le C \, 3^{-\alpha n} \Ll(1 + \sum_{k = 0}^{n-1} 3^{\alpha k } \tau_k \Rr) .
\end{equation*}
\end{proposition}
As will be explained below, Proposition~\ref{p.control.expec} follows from Proposition~\ref{p.Jsmall} by iteration, in analogy with an ODE argument. We focus for now on the proof of Proposition~\ref{p.Jsmall}, and start by rewriting the quadratic response \eqref{e.Jquadresponse} in a more convenient form.
\begin{lemma}[Quadratic response]
\label{l.quadresponse}
There exists a constant $C(d,\Lambda) < \infty$ such that the following holds. Let $V$, $V_1,\ldots, V_k$ be parabolic cylinders such that $\{V_1,\ldots,V_k\}$ forms a partition of $V$, up to a set of null measure. For every $X,X^*\in \R^{2d}$, we have
\begin{multline*} \label{}
\sum_{j=1}^k\frac{|V_j|}{|V|} \left\| S (\cdot,V,X,X^*) - S(\cdot,V_j,X,X^*) \right\|^2_{\underline{L}^2(V_j)} \\
 \leq C \sum_{j=1}^k \frac{|V_j|}{|V|}\left( J(V_j,X,X^*) -  J(V,X,X^*) \right).
\end{multline*}
\end{lemma}
\begin{proof}
Denote $T := S(\cdot,V,X,X^*)$. Applying \eqref{e.Jquadresponse} on the subdomain $V_j$ for each $j \in \{1,\ldots,k\}$, we get
\begin{multline*}  
\frac 1 C \int_{V_j} \Ll| T - S(\cdot,V_j,X,X^*) \Rr|^2 \\
\le 
|V_j| \, J(V_j,X,X^*) - \int_{V_j} \left( -\frac12  T \cdot \Abf  T  -X\cdot\Abf   T + X^*\cdot   T  \right).
\end{multline*}
Summing over $j$ and recalling \eqref{e.compact.J} yields the result.
\end{proof}

We next show that the spatial averages of $S$ can be controlled by an expression involving the additivity defects of $J$ on all smaller length scales. 
We denote
\begin{align*}
\ov S_n(X,X^*) & := \E \Ll[ \fint_{\cud_n} S(\cdot, \cud_n, X, X^*) \Rr] . 
\end{align*}
\begin{lemma}  
\label{l.spatavg.S}
There exist $\alpha(d) < \infty$ and $C(d,\Lambda) < \infty$  such that, for every $n \in \N$ and $X, X^* \in B_1(\R^{2d})$, we have
\begin{equation}  
\label{e.spatavS}
\E \Ll[ \Ll| \fint_{\cud_n}S(\cdot, \cud_n, X, X^*) - \ov S_n(X,X^*) \Rr|^2  \Rr] \\
\le C  3^{-\alpha n}\Ll( 1+ \sum_{k=0}^{n-1} 3^{\alpha k}\tau_k \Rr).
\end{equation}
\end{lemma}
\begin{proof}
For any $X'\in B_1$, $m \le n$ and $z\in \mcl Z_m$, the first variation \eqref{e.Jfirstvar} gives
\begin{multline*}
X'\cdot \fint_{z+\cud_m} \left(S(\cdot,\cud_n,X,X^*)-S(\cdot,z+\cud_m,X,X^*)\right)\\=\fint_{z+\cud_m}S(\cdot,z+\cud_m,0,X')
 \cdot \Abf \left(S(\cdot,\cud_n,X,X^*)-S(\cdot,z+\cud_m,X,X^*)\right).
\end{multline*}
Averaging over $z\in \mcl Z_m$ and using the Cauchy-Schwarz inequality yields
\begin{multline}
|\mcl Z_m|^{-1}\cdot\left|X'\cdot \sum_{z\in \mcl Z_m} \fint_{z+\cud_m}\left( S(\cdot,\cud_n,X,X^*)-S(\cdot,z+\cud_m,X,X^*)\right)\right| \\
\le  \left( |\mcl Z_m|^{-1}\sum_{z\in \mcl Z_m} \fint_{z+\cud_m} \Ll| \Abf S(\cdot,z+\cud_m,0,X')\Rr|^2 \right)^\frac12\\
\cdot \left( |\mcl Z_m|^{-1 }\sum_{z\in \mcl Z_m} \fint_{z+\cud_m}| S(\cdot,\cud_n,X,X^*)-S(\cdot,z+\cud_m,X,X^*)|^2 \right)^\frac12.
\end{multline}
The first term on the right side is bounded by a constant $C(d,\Lambda) < \infty$. We use Lemma~\ref{l.quadresponse} to bound the second term and obtain
\begin{multline*}
\Ll||\mcl Z_m|^{-1} \sum_{z\in \mcl Z_m} \fint_{z+\cud_m} | S(\cdot,\cud_n,X,X^*)-S(\cdot,z+\cud_m,X,X^*)|^2\Rr|  \\
   \le  \frac{C}{|\mcl Z_m|}  \sum_{z\in \mcl Z_m} \Ll(J(z+\cud_m,X,X^*)-J(\cud_n,X,X^*)\Rr).
\end{multline*}
Now, we can estimate the variance of $\fint_{\cud_n}S(\cdot,\cud_n,X,X^*)$ using those at scale $m$ :
\begin{multline*}
\var\Ll[  \fint_{\cud_n} S(\cdot,\cud_n,X,X^*)     \Rr]  \le 2 \var\Ll[  3^{-(d+2)(n-m)} \sum_{z\in \mcl Z_m} \fint_{z+\cud_m}S(\cdot,z+\cud_m,X,X^*)  \Rr] \\
+ C \E\Ll[  J(\cud_m,X,X^*)-J(\cud_n,X,X^*)  \Rr].
\end{multline*}
For $m,n \in \N$, $m \le n$,  we can decompose $\mcl Z_m$ into a  union of $2^{d+1}$ ``checkerboard'' subsets $\mcl Z^{(1)}, \ldots, \mcl Z^{(2^{1+d})}$ to ensure that for each $i \in \{1,\ldots, 2^{1+d}\}$, 
$$
(z,z')\in \mcl Z^{(i)} \implies \dist(z+\cud_m,z'+\cud_m) \ge 1.
$$ 
For example, to any $i\in \{1,\ldots,2^{d+1}\}$ we can associate $(i_0,\ldots,i_d)\in\{0,1\}^{d+1}$, and then set
\begin{equation}\label{e.defZi}
Z^{(i)}:=\Ll((i_0 3^{2m},i_1 3^m,\ldots,i_d 3^m)+2\Ll(  \left(3^{2m}\Z\right)\times\left( 3^m\Z^d\right)  \Rr)\Rr)\cap \cud_n.
\end{equation}
Thus, we obtain the following bound
\begin{multline*}
\var\Ll[ \sum_{z\in \mcl Z_m} \fint_{z+\cud_m} S(\cdot,z+\cud_m,X,X^*) \Rr]
\\
 \le C(d)\sum_{i=1}^{2^{d+1}} \var\Ll[  \sum_{z\in \mcl Z^{(i)}}  \fint_{z+\cud_m} S(\cdot,z+\cud_m,X,X^*) \Rr],
\end{multline*}
and by independence at distance larger than one and stationarity :
\begin{equation*}
\var\Ll[ \sum_{z\in \mcl Z_m} \fint_{z+\cud_m} S(\cdot,z+\cud_m,X,X^*)\Rr]  \le C 3^{(d+2)(n-m)} 
\var\Ll[    \fint_{\cud_m} S(\cdot,\cud_m,X,X^*) \Rr].
\end{equation*}
We can now estimate the variance of the spatial average of $S$ at scale $n$ by the variance at smaller scales :
\begin{multline}\label{e.estim_varS}
\var\Ll[  \fint_{\cud_n} S(\cdot,\cud_n,X,X^*)     \Rr]  \le  C 3^{-(d+2)(n-m)}  \var\Ll[  \fint_{\cud_m} S(\cdot,\cud_m,X,X^*)\Rr] \\
+ C \E\Ll[  J(\cud_m,X,X^*)-J(\cud_n,X,X^*)  \Rr].
\end{multline}
Selecting $\ell$ to be the smallest integer such that $C3^{-(d+2)\ell}\le \frac13$, we get
\begin{multline*}
	\var\Ll[  \fint_{\cud_{m+\ell}} S(\cdot,\cud_{m+\ell},X,X^*)     \Rr] \le  \frac13  \var\Ll[  \fint_{\cud_m}S(\cdot,\cud_m,X,X^*) \Rr]\\
	 + C \E\Ll[  J(\cud_{m},X,X^*)-J(\cud_{m+\ell},X,X^*)  \Rr].
\end{multline*}	
We introduce 
\[
u_n := \var\Ll[  \fint_{\cud_{n}} S(\cdot,\cud_{n},X,X^*) \Rr],
\]
and $v_n:=u_{n\ell}$.
We have
\begin{equation*}
v_n  \le  \frac13 v_{n-1} + C \sum_{k=(n-1)\ell}^{n\ell-1} \tau_k,
\end{equation*}
and, by induction,
\begin{equation*}
v_n \le 3^{-n}u_0 + C \sum_{i=1}^{n} 3^{-i} \sum_{k=(n-i)\ell}^{(n-i+1)\ell-1}\tau_k.
\end{equation*}
Defining $\alpha:=\frac{1}{\ell}$ (recall that $\ell$ only depends on $d$), we get
\begin{equation*}
v_n \le C\left( 3^{-n}   +  \sum_{k=0}^{n\ell-1} 3^{-\alpha (n\ell-k)}\tau_k  \right),
\end{equation*}
Thus, if $n$ is a multiple of $\ell$, we have 
\begin{equation*}
u_n \le C\left( 3^{-\alpha n}   +  \sum_{k=0}^{n-1} 3^{-\alpha (n-k)}\tau_k  \right),
\end{equation*}
and for $n=\ell n'+m$, with $0\le m < \ell$, another application of \eqref{e.estim_varS} gives the same estimate, so finally
\begin{equation*}
\var\Ll[   \fint_{\cud_n}S(\cdot,\cud_n,X,X^*)     \Rr] \le C\left( 3^{-\alpha n} +\sum_{k=0}^{n-1} 3^{-\alpha (n-k)} \tau_k \right),
\end{equation*}
which is \eqref{e.spatavS}.
\end{proof}

We can now sum the scales and deduce that $S(\cdot,\cud_n,X,X^*)$ is close to a constant in a weak sense, provided that a weighted norm of $(\tau_k)_{k < n}$ is small.
\begin{lemma}[Weak control of $ S $]
	\label{l.weak.S}
There exist $\alpha(d) < \infty$ and $C(d,\Lambda) < \infty$ such that, for every $n \in \N$ and $X, X^* \in B_1(\R^{2d})$,
	\begin{equation*}  
	\E\Ll[\|S(\cdot,\cud_n,X, X^*) - \ov S_n(X,X^*)\|_{\un {\hat H}^{-1}_\pa(\cud_n)}^2\Rr] \le C 3^{(2-\alpha) n} \left( 1+  \sum_{k=0}^{n-1} 3^{\alpha k}\tau_k \right).
	\end{equation*}
\end{lemma}
\begin{proof}
	We decompose the proof into three steps.
	
	\smallskip
	
	\emph{Step 1.} To begin with, we show that there exists a constant $C(d,\Lambda) < \infty$ such that, for every $m, n \in \N$, $m \le n$ and $X, X^* \in B^1(\R^{2d})$,
	\begin{equation}  
	\label{e.Sm.Sn}
	\Ll|\ov S_n(X,X^*) - \ov S_m(X,X^*) \Rr|^2 \le C \sum_{k = m}^{n-1} \tau_k.
	\end{equation}
	Indeed, recalling the definition of $\mcl Z_m$ in \eqref{e.def.mclZ} (which depends implicitly on $n$), we use Jensen's inequality, Lemma~\ref{l.quadresponse} and stationarity to get
	\begin{align*}  
	\lefteqn{
		\Ll| \E\Ll[\fint_{\cud_n} S(\cdot, \cud_n,X,X^*) - |\mcl Z_m|^{-1} \sum_{z \in \mcl Z_m} \fint_{z + \cud_m} S(\cdot, z + \cud_m,X,X^*) \Rr] \Rr|^2
	} \qquad & \\
	& \le \E \Ll[ |\mcl Z_m|^{-1} \sum_{z \in \mcl Z_m} \fint_{z + \cud_m} \Ll|S(\cdot, \cud_n, X, X^*) - S(\cdot, z + \cud_m,X,X^*) \Rr|^2 \Rr]  \\
	& \le C \Ll(J(\cud_n, X, X^*) - J(\cud_m,X, X^*)\Rr) \\
	& \le C \sum_{k = m}^{n-1} \tau_k,
	\end{align*}
	and this implies \eqref{e.Sm.Sn}.

	\smallskip
	
	\emph{Step 2.}
	In this step, we show that there exists a constant $C(d,\Lambda) < \infty$ such that, for every $m \in \N$, $m \le n$ ,
	\begin{multline}  
	\label{e.weakS.tobound}
	\E\Ll[|\mcl Z_m|^{-1}\sum_{z \in \mcl Z_m} \Ll| \Ll( S\Ll(\cdot,\cud_n, X, X^*\Rr) \Rr)_{z + \cud_m} - \ov S_n(X,X^*)  \Rr|^2\Rr] \\
	\le C \Ll( 3^{-\alpha m} + \sum_{k=0}^m 3^{\alpha(k-m)}\tau_k + \sum_{k=m}^{n-1} \tau_k \Rr) .
	\end{multline}
	By Lemma~\ref{l.quadresponse}, we have
	\begin{multline}  
	\label{e.quad.resp.S}
	\sum_{z \in \mcl Z_m} \|S\Ll(\cdot,\cud_n, X, X^* \Rr) - S\Ll(\cdot,z + \cud_m, X, X^*\Rr) \|_{\un L^2(z + \cud_m)}^2
	\\
	\le
	C \sum_{z \in \mcl Z_m} \Ll( J \Ll( \cud_n, X, X^*\Rr) - J \Ll( z + \cud_m, X, X^*\Rr)  \Rr) .
	\end{multline}
	Taking expectations, using stationarity and Jensen's inequality, we deduce that
	\begin{multline}  
	\label{e.one.triangle}
	\E \Ll[\sum_{z \in \mcl Z_m} \Ll| \Ll( S\Ll(\cdot,\cud_n, X, X^*\Rr) \Rr)_{z + \cud_m} - \Ll( S\Ll(\cdot,z + \cud_m, X, X^*\Rr) \Rr)_{z + \cud_m}  \Rr|^2 \Rr] 
	\\
	\le C  {|\mcl Z_m|} \sum_{k = m}^{n-1} \tau_k.
	\end{multline}
	Moreover, by stationarity and Lemma~\ref{l.spatavg.S}, we have
	\begin{multline}  
	\label{e.two.triangle}
	\E \Ll[|\mcl Z_m|^{-1} \sum_{z \in \mcl Z_m} \Ll| \Ll( S\Ll(\cdot,z + \cud_m, X, X^*\Rr) \Rr)_{z + \cud_m}  - \ov S_m(X, X^*)\Rr|^2 \Rr]  
	\\
	\le C \Ll( 3^{-\alpha m} + \sum_{k=0}^m 3^{\alpha(k-m)}\tau_k \Rr).
	\end{multline}
	Since
	\begin{align*}  
	\lefteqn{\Ll| \Ll( S\Ll(\cdot,\cud_n, X, X^*\Rr) \Rr)_{z + \cud_m} - \ov S_n(X,X^*)  \Rr|^2} \qquad &  \\
	& \le 
	3 \Ll| \Ll( S\Ll(\cdot,\cud_n, X, X^*\Rr) \Rr)_{z + \cud_m} -\Ll( S\Ll(\cdot,z+\cud_m, X, X^*\Rr) \Rr)_{z + \cud_m}   \Rr|^2 
	\\
	& \quad + 3  \Ll| \Ll( S\Ll(\cdot,z + \cud_m, X, X^*\Rr) \Rr)_{z + \cud_m}  - \ov S_m(X, X^*)\Rr|^2 
	\\
	& \quad + 3 \Ll| \ov S_n(X,X^*) - \ov S_m(X,X^*) \Rr|^2,
	\end{align*}
	we obtain \eqref{e.weakS.tobound} by combining \eqref{e.one.triangle}, \eqref{e.two.triangle} and \eqref{e.Sm.Sn}.
	
	\smallskip
	
	\emph{Step 3.} 
	We now combine Proposition~\ref{p.msp} with the result of the previous step to obtain that
	\begin{equation}  
	\label{e.weakS.almost1}
	\|S(\cdot,\cud_n,X, X^*) - \ov S_n(X,X^*)\|_{\un {H}^{-1}_\pa(\cud_n)}^2  \le C \Ll( 1 + \Ll( \sum_{m=0}^{n-1} 3^{m} Z_m^\frac 1 2 \Rr)^2  \Rr) ,
	\end{equation}
	where $Z_m$ is a random variable satisfying
	\begin{equation}  
	\label{e.weakS.almost2}
	\E[Z_m] \le C \Ll(   3^{-\alpha m} + \sum_{k=0}^m 3^{\alpha(k-m)}\tau_k + \sum_{k=m}^{n-1} \tau_k \Rr) .
	\end{equation}
	By H\"older's inequality, we have
	\begin{equation}
	\label{e.discrete.holder}
	\left( \sum_{m=0}^{n-1} 3^{m} Z_m^{\frac12} \right)^2 
	\leq \left( \sum_{m=0}^{n-1} 3^{m} \right) \left( \sum_{m=0}^{n-1} 3^{m}  Z_m\right) 
	\leq C 3^{n} \sum_{m=0}^{n-1} 3^{m} Z_m.
	\end{equation}
	Taking expectations and using \eqref{e.weakS.almost2}, we get
	\begin{align*}
	\E\left[ \left( \sum_{m=0}^{n-1} 3^{m} Z_m^{\frac12} \right)^2 \right]
	& \leq C  3^{n} \sum_{m=0}^{n-1} 3^{m}\left( 3^{-\alpha m} + \sum_{k=0}^m 3^{\alpha (k-m)}\tau_k + \sum_{k=m}^{n-1} \tau_k\right).
	\end{align*}
	For the last two terms, we reverse the order of the sums to find
	\begin{align*}  
	\sum_{m = 0}^{n-1} 3^{m} \sum_{k = 0}^m 3^{\alpha(k-m)} \tau_k = \sum_{k = 0}^{n-1} 3^{\alpha k}\tau_k \sum_{m = k}^{n-1} 3^{(1-\alpha m)} \leq C 3^{(1-\alpha)n} \sum_{k = 0}^{n-1}  3^{\alpha k} \tau_k,
	\end{align*}
	and
	\begin{equation}
	\label{e.resum2}
	\sum_{m=0}^{n-1} 3^{m} \sum_{k=m}^{n-1} \tau_k = \sum_{k=0}^{n-1} \sum_{m=0}^k 3^{m} \tau_k \leq C \sum_{k=0}^{n-1} 3^{k}\tau_k. 
	\end{equation}
	The second sum is bounded by the first, thus combining the above displays yields
	\begin{equation*}
	\E\left[ \left( \sum_{m=0}^{n-1} 3^m Z_m^{\frac12} \right)^2 \right] 
	\leq C 3^{(2-\alpha) n} \left( 1+  \sum_{k=0}^{n-1} 3^{\alpha k}\tau_k \right),
	\end{equation*}
	and this completes the proof.
\end{proof}

We next complete the proof of Proposition~\ref{p.Jsmall} and then of Proposition~\ref{p.control.expec}.
\begin{proof}[Proof of Proposition~\ref{p.Jsmall}]
According to Lemma~\ref{l.weak.S}, Proposition~\ref{p.caccio} and \eqref{e.spatav.Abfh}, we have
\begin{equation*}  
\E \Ll[\|S(\cdot,\cud_n,X,\Abfh_n X)\|_{\un L^2(\cud_{n-1})}^2\Rr] \le C 3^{-\alpha n } \Ll(1 + \sum_{k = 0}^{n-1} 3^{\alpha k} \tau_k \Rr) .
\end{equation*}
By Lemma~\ref{l.quadresponse}, we deduce
\begin{equation*}  
\E \Ll[\|S(\cdot,\cud_{n-1},X,\Abfh_n X)\|_{\un L^2(\cud_{n-1})}^2\Rr] \le C 3^{-\alpha n } \Ll(1 + \sum_{k = 0}^{n-1} 3^{\alpha k } \tau_k \Rr) .
\end{equation*}
Recall that $(z + \cud_{n-1})_{z \in \mcl Z_{n-1}}$ is a partition of $\cud_n$, up to a set of null measure. Moreover, by stationarity, the previous display implies that for every $z \in \mcl Z_{n-1}$,
\begin{equation*}  
\E \Ll[\|S(\cdot,z + \cud_{n-1},X,\Abfh_n X)\|_{\un L^2(z + \cud_{n-1})}^2\Rr] \le C 3^{-\alpha n } \Ll(1 + \sum_{k = 0}^{n-1} 3^{\alpha k } \tau_k \Rr) .
\end{equation*}
Applying Lemma~\ref{l.quadresponse} once more and summing over $z \in \mcl Z_{n-1}$, we obtain the result.
\end{proof}

\begin{proof}[Proof of Proposition~\ref{p.control.expec}]

We denote by $\mcl B$ the set of canonical basis elements of $\R^{2d}$, and observe that there exists a constant $C(d) < \infty$ such that if $X \mapsto B(X)$ is a nonnegative quadratic form over $\R^{2d}$, then 
\begin{equation}  
\label{e.sup.easy}
\sup_{X \in B_1} B(X) \le C \sum_{X \in \mcl B} B(X). 
\end{equation}
Indeed, a quadratic form is associated to a nonnegative symmetric matrix with largest eigenvalue bounded by its trace; this trace is equal to the right side above.

\smallskip

By the definition of $\tau_n$, see \eqref{e.def.taun}, and Lemma~\ref{l.Jsplitting}, we have
\begin{multline*}  
\tau_n \le \sup_{X \in 	B_1} \Ll( \E\Ll[\mu(\cud_n,X)\Rr] - \E\Ll[\mu(\cud_{n+1}, X) \Rr]\Rr) \\+  \sup_{X^* \in B_1} \Ll( \E\Ll[\mu^*(\cud_n,X^*)\Rr] - \E\Ll[\mu^*(\cud_{n+1}, X^*)\Rr] \Rr).
\end{multline*}
Since $X \mapsto \E\Ll[\mu(V,X)\Rr]$ and $X^* \mapsto \E\Ll[\mu(V, X^*)\Rr]$ are nonnegative quadratic forms, and since this property is stable under linear changes of coordinates, it follows from \eqref{e.sup.easy} that
\begin{multline*}  
\tau_n \le C \sum_{X \in \mcl B} \Ll( \E[\mu(\cud_n,X)] - \E[\mu(\cud_{n+1},X)] \Rr) \\ + C \sum_{X \in \mcl B} \Ll( \E[\mu^*(\cud_n,\Abfh_n X)] - \E[\mu^*(\cud_{n+1},\Abfh_n X)] \Rr)  ,
\end{multline*}
and thus by Lemma~\ref{l.Jsplitting},
\begin{equation*}  
\tau_n \le C \sum_{X \in \mcl B} \Ll( \E[J(\cud_n,X,\Abfh_n X)] - \E[J(\cud_{n+1},X,\Abfh_{n} X)]  \Rr) .
\end{equation*}
By \eqref{e.quad.resp.Abfh}, we have
\begin{equation*}  
\E \left[J(\cud_{n+1}, X, \Abfh_{n+1} X) \right] \le \E\left[J(\cud_{n+1}, X,\Abfh_n X)\right],
\end{equation*}
and therefore
\begin{equation}  
\label{e.taun.Dn}
\tau_n \le C \sum_{X \in \mcl B} \Ll( \E[J(\cud_n,X,\Abfh_n X)] - \E[J(\cud_{n+1},X,\Abfh_{n+1} X)]  \Rr) .
\end{equation}
This motivates the definition of 
\begin{equation*}  
D_n := \sum_{X \in \mcl B} \E\Ll[ J(\cud_n, X, \Abfh_n X) \Rr].
\end{equation*}
Proposition~\ref{p.Jsmall} asserts that
\begin{equation*}  
D_n \le C 3^{-\alpha n } \Ll(1 + \sum_{k = 0}^{n-1} 3^{\alpha k } \tau_k \Rr) .
\end{equation*}
Setting 
\begin{equation*}  
\td D_n := 3^{-\frac \alpha 2 n} \sum_{k = 0}^n 3^{\frac {\alpha } 2 k } D_k,
\end{equation*}
we deduce that
\begin{align}  
\notag
\td D_n & \le C 3^{-\frac \alpha 2 n} \sum_{m = 0}^n 3^{-\frac \alpha 2 m} \Ll( 1 + \sum_{k = 0}^m 3^{\alpha k} \tau_k \Rr)  \\
\notag
& \le C 3^{-\frac \alpha 2 n} + C 3^{-\frac \alpha 2 n} \sum_{k = 0}^n \sum_{m = k}^n 3^{-\frac \alpha 2 m} \, 3^{\alpha k} \tau_k \\
\label{e.dn.bound}
& \le C 3^{-\frac \alpha 2 n} \Ll( 1 + \sum_{k = 0}^n 3^{\frac \alpha 2 k } \tau_k \Rr) .
\end{align}
Since $D_0 \le C$, we also have
\begin{equation*}  
\td D_n - \td D_{n+1} \ge 3^{-\frac \alpha 2 n} \sum_{k = 0}^n 3^{\frac \alpha 2 k} \Ll(D_k - D_{k+1}\Rr) - C 3^{-\frac \alpha 2 n}.
\end{equation*}
Combining this with \eqref{e.taun.Dn} yields
\begin{equation*}  
\td D_n - \td D_{n+1} \ge C^{-1}  \, 3^{-\frac \alpha 2 n} \sum_{k = 0}^n 3^{\frac \alpha 2 } \tau_k - C 3^{-\frac \alpha 2 n}.
\end{equation*}
From this and \eqref{e.dn.bound}, we obtain that there exists an exponent $\be(d,\Lambda) \in (0,\frac \alpha 2)$ such that 
\begin{equation*}  
\td D_{n+1}  \le 3^{-\be} \td D_n + C 3^{-\frac \alpha 2 n},
\end{equation*}
introducing $v_n:=3^{\be n}\td D_n$ and multiplying the previous identity by $3^{(\be+1) n}$ gives
\begin{equation*}
	v_{n+1} \le v_n + C 3^{(\be-\frac{\alpha}{2})n}.
\end{equation*}
Summing this inequality over $n$ yields $v_n \le v_0 + \frac{C}{1-\left( 3^{ \beta-\frac \alpha 2}  \right)}$, i.e.\ $v$ is bounded, that is,
	\begin{equation*}
		\td D_n \le C 3^{-\be n}.
	\end{equation*}
By \eqref{e.taun.Dn}, we also obtain
\begin{equation*}
\tau_n \le C 3^{-\be n}.
\end{equation*}
By the definition of $\Abfh_n$, we have
\begin{equation*}  
\Ll| \Abfh_n - \Abfh_{n+1} \Rr|  \le C \tau_n,
\end{equation*}
so that, setting
\begin{equation}
\label{e.def.Abfh.bis}
\Abfh := \lim_{n \to \infty} \Abfh_n,
\end{equation}
we get
\begin{equation*}  
\Ll| \Abfh_n - \Abfh \Rr|
\le 
\sum_{m =n }^\infty \Ll| \Abfh_m - \Abfh_{m+1} \Rr| 
\leq 
C \sum_{m =n }^\infty \tau_m
\leq
C 3^{-\be n}.
\end{equation*}
Combining the last displays with \eqref{e.quad.resp.Abfh} yields
\begin{equation*}  
\sup_{X \in B_1} \E \Ll[ J\Ll(\cud_n,X,\Abfh X\Rr) \Rr] \le C 3^{-\be n}.
\end{equation*}
By an application of Lemma~\ref{l.minimalset}, we can verify that the matrix $\Abfh$ defined in \eqref{e.def.Abfh.bis} coincides with that defined in \eqref{e.defAbfh}. The proof is therefore complete.
\end{proof}

\subsection{Control of the fluctuations of \texorpdfstring{$J$}{J}}
\label{ss.fluc}
In this subsection, we prove Theorem~\ref{t.subadd}. In view of Lemma~\ref{l.minimalset}, the main point is to obtain a control on the fluctuations of $J(\cud_n, X, \Abfh X)$, which we obtain using subadditivity.
\begin{proof}[Proof of Theorem~\ref{t.subadd}]
\emph{Step 1.} In this first step, we show that there exists an exponent $\be(d,\Lambda) > 0$ and a constant $C(d,\Lambda) > 0$ such that, for every $X \in B_1(\R^{2d})$ and $m,n \in \N$, $m \le n$, we have
\begin{equation}
\label{e.logL.bound}
3^{-(2+d)(n-m)} \log \E \Ll[ \exp \Ll( C^{-1}  3^{(2+d)(n-m)} J(\cud_n, X, \Abfh X) \Rr)  \Rr] \le C 3^{-\be m}.
\end{equation}
For $m,n \in \N$, $m \le n$, recall that the cube $\cud_n$ is partitioned into  into a union of $2^{d+1}$ ``checkerboard'' subsets, see \eqref{e.defZi}, to ensure that for each $i \in \{1,\ldots, 2^{1+d}\}$,
\begin{equation*}  
z, z' \in \mcl Z^{(i)} \quad \implies \quad \dist \Ll(z+\cud_n, z' + \cud_n\Rr) \ge 1.
\end{equation*} 
In particular, for each fixed $i \in \{1,\ldots, 2^{1+d}\}$, the random variables $(z+\cud_n)_{z \in \mcl Z^{(i)}_m}$ are independent. By subadditivity, for each $X \in B_1(\R^{2d})$ and $t > 0$, we have
\begin{align*}  
\log \E \Ll[ \exp \Ll( t 3^{(2+d)(n-m)} J(\cud_n, X, \Abfh X) \Rr)  \Rr] \le \log \E \Ll[ \exp \Ll( t \sum_{z \in \mcl Z_m} J(z+\cud_m, X, \Abfh X) \Rr)  \Rr],
\end{align*}
and by H\"older's inequality and independence, the latter is bounded by 
\begin{multline*}
\le 2^{-(1+d)} \sum_{i = 1}^{2^{1+d}} \log \E \Ll[ \exp \Ll( t 2^{1+d}\sum_{z \in \mcl Z_m^{(i)}} J(z+\cud_m, X, \Abfh X) \Rr)  \Rr] \\
\le 2^{-(1+d)} \sum_{z \in \mcl Z_m} \log \E \Ll[ \exp \Ll( t 2^{1+d} J(z+\cud_m, X, \Abfh X) \Rr)  \Rr] .
\end{multline*}
By stationarity, the summands above do not depend on $z \in \mcl Z_m$. Since $$J(\cud_m, X, \Abfh X) \le C(d,\Lambda),$$ we can choose $t(d,\Lambda) > 0$ sufficiently small and use the elementary inequalities
\begin{equation*} 
\left\{ 
\begin{aligned}
& \exp(s) \leq 1+2s && \mbox{for all} \ 0\leq s \leq 1, \\
& \log(1+s) \leq s && \mbox{for all} \ s\geq 0
\end{aligned} 
\right.
\end{equation*}
to obtain that
\begin{equation*}  
\log \E \Ll[ \exp \Ll( C^{-1} 3^{(2+d)(n-m)} J(\cud_n, X, \Abfh X) \Rr)  \Rr] \le C 3^{(2+d)(n-m)} \E \Ll[ J(\cud_m,X,\Abfh X) \Rr] .
\end{equation*}
Inequality \eqref{e.logL.bound} then follows by an application of Proposition~\ref{p.control.expec}.

\smallskip 

\emph{Step 2.} Set
\begin{equation*}  
\rho_n := \sup_{X \in B_1} J(\cud_n,X, \Abfh X).
\end{equation*}
In this step, we show that there exists an exponent $\be(d,\Lambda) > 0$ and, for every $s \in (0,2+d)$, a constant $C(s,d,\Lambda) < \infty$ such that, for every $n \in \N$,
\begin{equation}  
\label{e.rhon.O}
\rho_n \le C 3^{-\be(2+d-s)n} + \O_1 \Ll(C  3^{-sn} \Rr) .
\end{equation}
By \eqref{e.sup.easy} and H\"older's inequality, the relation \eqref{e.logL.bound} can be improved to
\begin{equation}
\label{e.logL.sup}
3^{-(2+d)(n-m)} \log \E \Ll[ \exp \Ll( C^{-1}  3^{(2+d)(n-m)} \rho_n \Rr) \Rr] \le C 3^{-\be m}.
\end{equation}
By Chebyshev's inequality, for every $t \ge 0$, 
\begin{align*}  
\P \Ll[ \rho_n \ge t \Rr] &  \le \exp \Ll( - C^{-1}  3^{(2+d)(n-m)} t \Rr) \E \Ll[ \exp \Ll( C^{-1}  3^{(2+d)(n-m)} \rho_n \Rr) \Rr] \\
& \le \exp \Ll(- C^{-1}  3^{(2+d)(n-m)} t + C 3^{ (2+d)(n-m)-\be m} \Rr) .
\end{align*}
Replacing $t$ by $C 3^{-\be m} + t$ gives
\begin{equation*}  
\P \Ll[ \rho_n \ge C 3^{-\be m} + t \Rr]  \le \exp \Ll(- C^{-1}  3^{(2+d)(n-m)} t\Rr) .
\end{equation*}
Choosing 
\begin{equation*}
m := \Ll\lfloor \frac{2+d-s}{2+d} n \Rr\rfloor
\end{equation*}
yields
\begin{equation*}  
\P \Ll[ \rho_n \ge C 3^{-\be\frac{2+d-s}{2+d}n} + t \Rr]  \le \exp \Ll(- C^{-1}  3^{s n} t\Rr) .
\end{equation*}
By~\eqref{e.chebyconverse}, this is \eqref{e.rhon.O}, up to a redefinition of $\be(d,\Lambda) > 0$.

\smallskip

\emph{Step 3.}
We now combine Lemma~\ref{l.minimalset},~\eqref{e.logL.sup} and the elementary inequality
\begin{equation} 
\label{e.nicelementary}
\forall a,b > 0, \qquad (a+b)^{\frac12} \leq a^{\frac12} + \frac12 a^{-\frac12} b 
\end{equation}
to get
\begin{multline}  
\label{e.whats.your.end.game}
\sup_{X,X^*\in B_1} 
\left|  
J(\cud_n,X,X^*) - 
\left( 
\frac12X
\cdot {\Abfh}X
+ \frac12 X^*
\cdot {\Abfh}^{-1} X^*
- X\cdot X^*
 \right) 
 \right|
\\
\le 3^{-\frac{\be}2 (2+d-s)n} + \O_1 \Ll(C  3^{-\Ll(s - \frac{\be}2 (2+d-s)\Rr)n} \Rr) .
\end{multline}
For every $s' \in (0,2+d)$, if we set
\begin{equation*}  
s := \frac{2s' + \be(2+d)}{2+\be} \in (0,2+d),
\end{equation*}
then the right side of \eqref{e.whats.your.end.game} can be rewritten as
\begin{equation*}  
3^{-\frac{\be}{2+\be} (2+d-s')n} + \O_1 \Ll(C  3^{-s'n} \Rr).
\end{equation*}
We thus obtained \eqref{e.subadderror}, up to a redefinition of $\be(d,\Lambda) > 0$.
\end{proof}

\begin{proposition}
\label{p.ahomdef}
There exist $C(d,\Lambda)<\infty$ and a matrix $\ahom \in \R^{d\times d}$ satisfying
\begin{equation} 
\label{e.ahombounds.bigJ}
\forall \xi\in\Rd, \quad 
\xi\cdot \ahom \xi \geq \frac1C\left| \xi \right|^2
\quad \mbox{and} \quad 
\left| \ahom \xi \right| \leq C\left| \xi \right|,
\end{equation}
such that, for every $p,q \in \Rd$, we have the equivalence
\begin{equation}
\label{e.mappingAbar}
\frac 12 \begin{pmatrix} p \\ q  \end{pmatrix} \cdot \Abfh \begin{pmatrix} p \\ q  \end{pmatrix} 
- p\cdot q = 0 
\iff
q = \ahom p
.
\end{equation}
\end{proposition}
\begin{proof}

\emph{Step 1.} We show that, for every $p,q\in\Rd$, 
\begin{equation} 
\label{e.Awantsrep}
\overline A(p,q) := \frac12 \begin{pmatrix}  p \\ q \end{pmatrix} \cdot \Abfh \begin{pmatrix}  p \\ q \end{pmatrix} 
\geq p \cdot q. 
\end{equation}
By Lemma~\ref{l.lin.Fitz}, we have for every $S=(\nabla u,\g)\in \C_0(I\times U)$ and $p,q \in \Rd$ that
\begin{equation*}
\fint A(p+\nabla u, q+\g,\cdot)   \ge    \fint (p+\nabla u)\cdot (q + \g) = p\cdot q.
\end{equation*}
By the definition of $\mu$ in \eqref{e.def.mu}, we deduce that
\begin{equation*}
	\mu(I\times U,p,q)\ge p\cdot q,
\end{equation*}
and thus~\eqref{e.Awantsrep} follows from~\eqref{e.defAbfh}.

\smallskip

\emph{Step 2.} We show that, for every $q^*,p^*\in\Rd$, 
\begin{equation} 
\label{e.Ainvwantsrep}
\frac12 \begin{pmatrix}  q^* \\ p^* \end{pmatrix} \cdot \Abfh^{-1} \begin{pmatrix}  q^* \\ p^* \end{pmatrix} 
\geq q^* \cdot p^*. 
\end{equation}
Fix $p^*\in\Rd$. 
For every $u\in \ell_{p^*} + H^1_{\pa,\sqcup}(I\times U)$ and $\g \in L^2(I\times U;\Rd)$ satisfying $-\nabla \cdot \g = -\partial_t u$, we have $(\nabla u,\g) \in \mcl{C}(I\times U)$ as well as 
\begin{equation*} \label{}
\fint_{I\times U}  \nabla u =  p^*
\end{equation*}
and
\begin{equation*} \label{}
\fint_{I\times U}( p^*-\nabla u ) \cdot \g = \frac{1}{|I|} \fint_{U} (u - \ell_p)^2 \geq 0. 
\end{equation*}
Therefore, for every $q^*\in\Rd$, 
\begin{align*} \label{}
\mu^*(I\times U,q^*,p^*) 
& 
\geq 
\fint_{I\times U} 
\left( -A(\nabla u, \g,\cdot) + q^*\cdot\nabla u + p^* \cdot \g \right)
\\ & 
\geq \fint_{I\times U} 
\left( -A(\nabla u, \g,\cdot) + q^*\cdot p^* + \nabla u \cdot \g \right)
\\ & 
= q^*\cdot p^* - \fint_{I\times U} 
\left( A(\nabla u, \g,\cdot) - \nabla u \cdot \g \right).
\end{align*}
By the solvability of the Cauchy-Dirichlet problem (Proposition~\ref{p.parabolic.min.app}), for every $p^*\in\Rd$,
\begin{multline*} \label{}
0 = \inf
\bigg\{ \int_{I\times U} 
\left( A(\nabla u, \g,\cdot) - \nabla u\cdot \g \right) 
\\
\,:\, 
u\in \ell_{p^*} + H^1_{\pa,\sqcup}(I\times U), \ \g \in L^2(I\times U;\Rd),\ 
-\nabla \cdot \g = -\partial_t u
\bigg\}.
\end{multline*}
Combining the above yields 
\begin{equation*} \label{}
\mu^*(I\times U,q^*,p^*) 
\geq q^*\cdot p^*. 
\end{equation*}
According to Theorem~\ref{t.subadd}, we have the $\P$-a.s.~limit
\begin{equation*} \label{}
\lim_{n\to \infty} \mu^*(\cud_n,X^*) = \frac12 X^* \cdot  \Abfh^{-1} X^*. 
\end{equation*}
We therefore obtain~\eqref{e.Ainvwantsrep}.

\smallskip

\emph{Step 3.} We argue that, for every $p\in\Rd$, 
\begin{equation} 
\label{e.touchbottom}
\inf_{q\in\Rd} \left( \overline A(p,q) - p\cdot q \right) = 0.
\end{equation}
We have already shown in~\eqref{e.Awantsrep} that the infimum on the left is nonnegative. The infimum is attained, by the quadratic growth of $q\mapsto \overline{A}(p,q)$. To see that it is equal to zero, we fix $p \in \Rd$ and select $q \in \Rd$ achieving the infimum. Then 
\begin{equation*}  
 \Abfh \begin{pmatrix} p \\ q  \end{pmatrix} = \begin{pmatrix} \ast \\ p  \end{pmatrix}. 
\end{equation*}
Let  $q^* \in \Rd$ denote the ``$\ast$'' in the previous line, so that 
\begin{equation}  
\label{e.define.q*}
\Abfh \begin{pmatrix} p \\ q  \end{pmatrix} = \begin{pmatrix} q^* \\ p  \end{pmatrix}.
\end{equation}
Then using~\eqref{e.define.q*}, we find that 
\begin{align*}  
\frac 1 2  \begin{pmatrix} q^* \\ p  \end{pmatrix} \Abfh^{-1}  \begin{pmatrix} q^* \\ p  \end{pmatrix} & = \sup_{p',q' \in \Rd} \Ll(  \begin{pmatrix} q^* \\ p  \end{pmatrix} \cdot  \begin{pmatrix} p' \\ q'  \end{pmatrix} - \frac 12 \begin{pmatrix} p' \\ q'  \end{pmatrix} \cdot \Abfh \begin{pmatrix} p' \\ q'  \end{pmatrix}  \Rr)  \\
& = \begin{pmatrix} q^* \\ p  \end{pmatrix} \cdot  \begin{pmatrix} p \\ q  \end{pmatrix} 
- \frac 12 \begin{pmatrix} p \\ q  \end{pmatrix} \cdot \Abfh \begin{pmatrix} p \\ q  \end{pmatrix} 
\end{align*}
By the previous inequality and~\eqref{e.Ainvwantsrep}, we discover that 
\begin{equation*} \label{}
p\cdot q^* \leq \begin{pmatrix} q^* \\ p  \end{pmatrix} \cdot  \begin{pmatrix} p \\ q  \end{pmatrix} 
- \frac 12 \begin{pmatrix} p \\ q  \end{pmatrix} \cdot \Abfh \begin{pmatrix} p \\ q  \end{pmatrix} 
= p\cdot q^* + p\cdot q - \frac 12 \begin{pmatrix} p \\ q  \end{pmatrix} \cdot \Abfh \begin{pmatrix} p \\ q  \end{pmatrix}.
\end{equation*}
Rearranging, this yields~$\bar A (p,q) \leq p\cdot q$, which in view of~\eqref{e.Awantsrep} allows us to deduce that $\bar A (p,q) = p\cdot q$ and completes the proof of~\eqref{e.touchbottom}. 

\smallskip

\emph{Step 4.} We define~$\ahom$ to be the matrix associated to the linear mapping taking~$p$ to the~$q$ achieving the infimum in~\eqref{e.touchbottom}. That the infimum is achieved at a unique minimum point is a consequence of the uniform convexity of~$q\mapsto \overline{A}(p,q)$. That this mapping is linear is due to the fact that~$q\mapsto \overline{A}(p,q)$ is quadratic. The bounds~\eqref{e.ahombounds.bigJ} are a consequence of~\eqref{e.abigbarbounds}. This completes the proof of the proposition. 
\end{proof}

\section{Quantitative homogenization of the Cauchy-Dirichlet problem}
\label{s.CDP}

In this section, we demonstrate the passage from the convergence of~$J$ to the homogenization of the parabolic operator. In particular, we complete the proof of Theorem~\ref{t.CDP} on the quantitative homogenization of the Cauchy-Dirichlet problem. The argument is completely deterministic in the sense that the only probabilistic ingredient is the appeal to Theorem~\ref{t.subadd}. The argument proceeds in four steps: (i) we show that convergence of~$J$ implies convergence of~$S$ and~$\Abf S$ in~$H^{-1}$; (ii) we use Remark~\ref{r.deconstruction} to show that there are ``finite-volume correctors'' which can be found hiding in~$S$ and~$\Abf S$ and we obtain estimates on them; (iii) we use the finite-volume correctors and a quantitative version of the standard two-scale expansion argument to pass from estimates on the correctors to estimates on the homogenization error for a general Cauchy-Dirichlet problem.

\subsection{Convergence of $J$ maximizers}
In this subsection, we use the multiscale Poincar\'e inequality (Proposition~\ref{p.msp}) to obtain information about the weak convergence of~$S(\cud_n,X,X^*)$ as~$n\to \infty$. It is useful to define the quantity 
\begin{equation*} \label{}
\mathcal{E}(V) := \sup_{X,X^*\in B_1} \left| J(V,X,X^*) - \overline{J}(X,X^*) \right|,
\end{equation*}
which keeps track of the convergence of~$J$. We also denote, given $X,X^* \in \R^{2d}$, 
\begin{equation*} \label{}
\overline S(X,X^*) := \Abfh^{-1} X^* - X. 
\end{equation*}
Note that $\overline S = \nabla_{X^*} \overline{J}$ and therefore, by~\eqref{e.JderX*} and the fact that $J$ and $\overline{J}$ are quadratic, we have, for every $X,X^*\in \R^{2d}$, 
\begin{align} 
\label{e.spatavgS}
\left| \overline S(X,X^*) - \fint_{V} S(\cdot,V,X,X^*)  \right| 
& 
= \left|\nabla_{X^*}  \overline J(X,X^*) - \nabla_{X^*}  J(V,X,X^*)  \right| 
\\ & \notag
\leq  C \left( |X|+|X^*| \right) \mathcal{E}(V) .
\end{align}
Similarly, 
\begin{align} 
\label{e.spatavgAS}
\left| \Abfh \overline S(X,X^*) - \fint_{V} \Abf S(\cdot,V,X,X^*)  \right| 
& 
= \left| \nabla_{X}  \overline J(X,X^*) - \nabla_{X}  J(V,X,X^*)  \right| 
\\ & \notag
\leq  C \left( |X|+|X^*| \right) \mathcal{E}(V).
\end{align}
That is, we can control the spatial averages of $S(\cdot,V,X,X^*)$ and $\Abf S(\cdot,V,X,X^*)$ in terms of the random variable~$\mathcal{E}(V)$. The combination of this observation and Proposition~\ref{p.msp} yields the following result. 

\begin{proposition}[{Weak convergence of $(S,\Abf S)$}]
\label{p.controlofS}
There exists $C(d,\Lambda)<\infty$ such that, for every $X,X^*\in B_1$ and $n\in\N$,
\begin{multline} 
\label{e.HminusoneS}
 3^{-n} \left\| S(\cdot,\cud_n,X,X^*) - \overline{S}(X,X^*) \right\|_{\un {\hat H}^{-1}_\pa(\cud_n)}
\\
\leq
C3^{-n} + C \sum_{m = 0}^{n-1} 3^{m-n} \Ll(  \left|\mcl Z_m\right|^{-1} \sum_{z\in \mcl Z_m} \left( \mathcal{E}\left( z+\cud_m \right) \right) \Rr)^{\frac 12}
\end{multline}
and
\begin{multline} 
\label{e.HminusoneAS}
3^{-n} \left\| \Abf S(\cdot,\cud_n,X,X^*) - \Abfh \overline{S}(X,X^*) \right\|_{\un {\hat H}^{-1}_\pa(\cud_n)}
\\
\leq
C3^{-n} + C \sum_{m = 0}^{n-1} 3^{m-n} \Ll(  \left|\mcl Z_m\right|^{-1} \sum_{z\in \mcl Z_m} \left( \mathcal{E}\left( z+\cud_m \right) \right) \Rr)^{\frac 12}.
\end{multline}
\end{proposition}
\begin{proof}
We fix $X,X^*\in B_1$ and, since it plays no role in the argument, we drop explicit display of the dependence on $(X,X^*)$. According to Proposition~\ref{p.msp}, 
\begin{multline*}
\left\| S(\cdot,\cud_n) - \overline{S} \right\|_{\un {\hat H}^{-1}_\pa(\cud_n)} 
\leq
C \left\| S(\cdot,\cud_n) - \overline{S} \right \|_{\un L^2(\cud_n)} 
\\
+ C \sum_{m = 0}^{n-1} 3^{m} \Ll( \left|\mcl Z_m\right|^{-1} \sum_{z \in \mcl Z_m} \left| \left (S(\cdot,\cud_n) \right)_{z + \cud_m} - \overline{S} \right|^2 \Rr)^{\frac 12}. 
\end{multline*}
To estimate the first term on the right side, we just observe that 
\begin{equation*} \label{}
\left\| S(\cdot,\cud_n) - \overline{S} \right \|_{\un L^2(\cud_n)} 
\leq \left\| S(\cdot,\cud_n) \right \|_{\un L^2(\cud_n)} + \left|  \overline{S}  \right|
\leq C. 
\end{equation*}
We next estimate the second term. By the triangle inequality,~\eqref{e.spatavgS} and Lemma~\ref{l.quadresponse},
\begin{align*} \label{}
\lefteqn{
\sum_{z\in \mcl Z_m} \left| \left (S(\cdot,\cud_n) \right)_{z + \cud_m} - \overline{S} \right|^2
} \qquad &
\\ & 
\leq 
2\sum_{z\in \mcl Z_m}
\left( \left| \left(S\left(\cdot,{z+\cud_m} \right) \right)_{z + \cud_m} - \overline{S} \right|^2 
+
\left\| S\left(\cdot,\cud_n\right) - S\left(\cdot,z+\cud_m \right)\right\|^2_{\underline{L}^2(z+\cud_m)}
\right)
\\ & 
\leq C \sum_{z\in \mcl Z_m} \mathcal{E}\left( z+\cud_m \right) + C \sum_{z\in \mcl Z_m} \left( J(z+\cud_m) - J(\cud_n) \right) 
\\ & 
\leq C \sum_{z\in \mcl Z_m} \left( \mathcal{E}\left( z+\cud_m \right) + \mathcal{E}\left(\cud_n\right) \right). 
\end{align*}
Thus
\begin{align*}
\lefteqn{
\sum_{m = 0}^{n-1} 3^{m} \Ll( \left|\mcl Z_m\right|^{-1} \sum_{z \in \mcl Z_m} \left| \left (S(\cdot,\cud_n) \right)_{z + \cud_m} - \overline{S} \right|^2 \Rr)^{\frac 12}
} \qquad & 
\\ &
\leq
C \sum_{m = 0}^{n-1} 3^{m} \Ll(  C\mathcal{E}\left(\cud_n\right)  + \left|\mcl Z_m\right|^{-1} \sum_{z\in \mcl Z_m} \left( \mathcal{E}\left( z+\cud_m \right) \right) \Rr)^{\frac 12}
\\ & 
\leq C \sum_{m = 0}^{n-1} 3^{m} \Ll(  \left|\mcl Z_m\right|^{-1} \sum_{z\in \mcl Z_m} \left( \mathcal{E}\left( z+\cud_m \right) \right) \Rr)^{\frac 12}.
\end{align*}
Combining the above yields~\eqref{e.HminusoneS}. The estimate~\eqref{e.HminusoneAS} is obtained similarly, we just need to use~\eqref{e.spatavgAS} instead of~\eqref{e.spatavgS}. 
\end{proof}

We next give an estimate of the random variable appearing on the right side of~\eqref{e.HminusoneS} and~\eqref{e.HminusoneAS}, which is a straightforward consequence of Theorem~\ref{t.subadd}. This is the only place in this section where Theorem~\ref{t.subadd} or any other stochastic ingredient is used. 

\begin{proposition}
\label{p.controlofEm}
There exists $\beta(d,\Lambda)$ and, for every $s\in (0,2+d)$, a constant $C(s,d,\Lambda)<\infty$ such that, for every $n\in\N$, 
\begin{equation}
\label{e.estimatemonster}
\sum_{m = 0}^{n-1} 3^{m-n} \Ll(  \left|\mcl Z_m\right|^{-1} \sum_{z\in \mcl Z_m} \left( \mathcal{E}\left( z+\cud_m \right) \right) \Rr)^{\frac 12}
\leq C3^{-n\beta(2+d-s)} + \O_1\left(C3^{-ns} \right).
\end{equation}
\end{proposition}

\begin{proof}
Fix $s':= \frac13(2s+2+d)$ and $s'':= \frac13(s+2(2+d))$ so that $s<s'<s''< 2+d$ with equally sized gaps between these numbers. By Theorem~\ref{t.subadd} and~\eqref{e.Osums}, we have  
\begin{equation*} \label{}
 \left|\mcl Z_m\right|^{-1} \sum_{z\in \mcl Z_m} \left( \mathcal{E}\left( z+\cud_m \right) \right)
\leq
C3^{-m\beta(2+d-s'')} + \O_1\left(C3^{-ms''} \right).
\end{equation*}
Using the elementary inequality~\eqref{e.nicelementary}, we deduce that 
\begin{equation*} \label{}
\left( 
\left|\mcl Z_m\right|^{-1} \sum_{z\in \mcl Z_m} \left( \mathcal{E}\left( z+\cud_m \right) \right)
\right)^{\frac12}
\leq 
C3^{-m\beta(2+d-s'')/2} + \O_1\left(C3^{m\beta(2+d-s'')/2-ms''} \right).
\end{equation*}
Redefining $\beta$ to be smaller if necessary, we get
\begin{equation*} \label{}
\left( 
\left|\mcl Z_m\right|^{-1} \sum_{z\in \mcl Z_m} \left( \mathcal{E}\left( z+\cud_m \right) \right)
\right)^{\frac12}
\leq 
C3^{-m\beta(2+d-s')} + \O_1\left(C3^{-ms'} \right).
\end{equation*}
As the left side of the previous line is bounded by~$C$, we can apply~\cite[Lemma A.3]{AKMBook} to obtain
\begin{equation*} \label{}
\left( 
\left|\mcl Z_m\right|^{-1} \sum_{z\in \mcl Z_m} \left( \mathcal{E}\left( z+\cud_m \right) \right)
\right)^{\frac12}
\leq 
C3^{-m\beta(2+d-s')} + \O_{d+2}\left(C3^{-ms'/(d+2)} \right).
\end{equation*}
Since $s'/(d+2)\leq1-c$, we may apply~\eqref{e.Osums} again to obtain 
\begin{equation*} \label{}
\sum_{m = 0}^{n-1} 3^{m-n} \Ll(  \left|\mcl Z_m\right|^{-1} \sum_{z\in \mcl Z_m} \left( \mathcal{E}\left( z+\cud_m \right) \right) \Rr)^{\frac 12}
\leq C3^{-n\beta(2+d-s')} + \O_{d+2}\left(C3^{-ns'/(d+2)} \right).
\end{equation*}
We conclude by observing that, for any nonnegative random variable $X$, 
\begin{equation} 
\label{e.Ximpy}
X \leq  C3^{-n\beta(2+d-s')} + \O_{d+2}\left(C3^{-ns'/(d+2)}  \right) 
\implies
X \leq C3^{-n\beta(2+d-s)} + \O_1\left(C3^{-ns} \right),
\end{equation}
where $\beta(d,\Lambda)$ in the second statement may be smaller than in the first. To see this, we compute
\begin{align*} \label{}
X 
\leq X+3^{-n\beta(2+d-s)(1+d)}
&
\leq 3^{n\beta(2+d-s)(1+d)} \left(X+3^{-n\beta(2+d-s)} \right)^{2+d}
\\ & 
\leq C 3^{n\beta(2+d-s)(1+d)} \left(X^{2+d} +3^{-n\beta(2+d-s')(2+d)} \right)
\\ & 
\leq C 3^{n\beta(2+d-s)(1+d)} \left(C3^{-n\beta(2+d-s)(2+d)} + \O_{1}\left(C3^{-ns'} \right) \right)
\\ & 
\leq C3^{-n\beta(2+d-s)} + \O_1\left(C3^{-ns} \right),
\end{align*}
provided that $\beta$ is small enough that $\beta(2+d-s)(2+d) \leq s'-s$. It suffices to require $\beta \leq 1/3(2+d)$. This completes the proof of~\eqref{e.Ximpy} and of the proposition.
\end{proof}

\subsection{Construction of finite-volume correctors}
We next give the construction of the (finite-volume) correctors. The usage of the term ``corrector'' in stochastic homogenization is typically reserved for a function with stationary, mean-zero gradient which is the difference of a solution of the equation in the full space and an affine function. For our purposes, it is more convenient to work with a finite-volume approximation of the corrector which will be defined on a large cylinder~$\cud_n$, because this is what comes most easily and naturally out of the estimates we have already proved above. These correctors will be obtained in a simple way from~$S(\cdot,X,X^*)$ and~$\Abf S(\cdot,X,X^*)$ and Remark~\ref{r.deconstruction}; the estimates we need for them will be easy consequences of~\eqref{e.HminusoneS} and~\eqref{e.HminusoneAS}. The fact that these correctors are not stationary functions defined in the whole space does not create any complication in the proof of Theorem~\ref{t.CDP}. 

\smallskip

The corrector with slope~$e \in \Rd$ on the cylinder~$\cud_n$ with~$n\in\N$ will be denoted by~$\phi_{e,n}$. 
We define it from the maximizers of $J(\cud_m,X,0)$, studied in the previous section. We first must make an appropriate choice of $X$, depending on $e$. This is a linear algebra exercise using Proposition~\ref{p.ahomdef}. We set
\begin{equation*} \label{}
X_e:= 
- \begin{pmatrix} e \\ \ahom e \end{pmatrix}
\end{equation*}
and observe from~\eqref{e.mappingAbar} that we have 
\begin{equation} 
\label{e.freedom}
\Abfh X_e = 
- \begin{pmatrix} \ahom e \\ e \end{pmatrix}.
\end{equation}
To check the previous line, we note (see Proposition~\ref{p.ahomdef}) that the map
\begin{equation*} \label{}
q \mapsto \frac12 \begin{pmatrix} e \\ q \end{pmatrix} \cdot \Abfh \begin{pmatrix} e \\ q \end{pmatrix} - e\cdot q
\quad \mbox{attains its minimum at} \ q=\ahom e
\end{equation*}
and the map
\begin{equation*} \label{}
p \mapsto \frac12 \begin{pmatrix} p \\ \ahom e \end{pmatrix} \cdot \Abfh \begin{pmatrix} p \\ \ahom e \end{pmatrix} - p\cdot \ahom e
\quad \mbox{attains its minimum at} \ p= e.
\end{equation*}
Differentiating in $p$ and $q$, respectively, gives~\eqref{e.freedom}.

\smallskip

We next take~$u_{e,n}\in H^1_\pa(\cud_{n+1})$ to be the element $u\in \mathcal{A}(\cud_{n+1})$ in the representation of~$S(\cdot,\cud_{n+1},X_e,0)$ given in Lemma~\ref{l.identif.S}, with additive constant chosen so that~$\left(u_{e,n}\right)_{\cud_n} = 0$. Equivalently, in view of Remark~\ref{r.deconstruction}, we can define~$u_{e,n}$ to be the function on $\cud_{n+1}$ with mean zero on~$\cud_n$ with gradient given by
\begin{equation} 
\label{e.deconformugrad}
\nabla u_{e,n} = \frac12 \left( \pi_1 S\left( \cdot,\cud_{n+1},X_e,0\right) + \pi_2 \Abf S\left( \cdot,\cud_{n+1},X_e,0 \right) \right),
\end{equation}
where $\pi_1$ and $\pi_2$ denote the projections $\R^{2d}\to \Rd$ onto the first and second~$d$ variables, respectively (that is, $\pi_1(x,y) = x$ and $\pi_2(x,y)=y$ for $x,y\in\Rd$). Note that, by Remark~\ref{r.deconstruction}, we also have the formula
\begin{equation}
\label{e.deconformuflux}
\a \nabla u_{e,n} =  \frac12 \left( \pi_2 S\left( \cdot,\cud_{n+1},X_e,0\right) + \pi_1 \Abf S\left( \cdot,\cud_{n+1},X_e,0 \right) \right).
\end{equation}
By Proposition~\ref{p.controlofS},~\eqref{e.freedom},~\eqref{e.deconformugrad} and~\eqref{e.deconformuflux}, we have that 
\begin{multline} 
\label{e.weakucontrol}
3^{-n} \left\| \nabla u_{e,n} - e \right\|_{\un {\hat H}^{-1}_\pa(\cud_{n+1})} 
+ 
3^{-n} \left\| \a \nabla u_{e,n} - \ahom e \right\|_{\un {\hat H}^{-1}_\pa(\cud_{n+1})} 
\\
\leq
C3^{-n} + C \sum_{m = 0}^{n} 3^{m-n} \Ll(  \left|\mcl Z_m\right|^{-1} \sum_{z\in \mcl Z_m} \left( \mathcal{E}\left( z+\cud_m \right) \right) \Rr)^{\frac 12}.
\end{multline}
Since~$u_{e,n}\in \mathcal{A}(\cud_{n+1})$, we have that~$u_{e,n}$ is a solution of 
\begin{equation} 
\label{e.eqnforuen}
\partial_t u_{e,n} - \nabla \cdot \left( \a \nabla u_{e,n}\right) = 0 \quad \mbox{in} \ \cud_{n+1}.
\end{equation}
The approximate first-order corrector $\phi_{e,n}$ is defined by subtracting the affine function $x\mapsto e\cdot x$ from $u_{e,n}$:
\begin{equation*} \label{}
\phi_{e,n}(x) := u_{e,n}(x) - e\cdot x. 
\end{equation*}
Summarizing, we therefore have that $\phi_{e,n}$ is a solution of 
\begin{equation} 
\label{e.eq.phien}
\partial_t \phi_{e,n} - \nabla \cdot \left( \a \left(e+\nabla \phi_{e,n}\right) \right) = 0 \quad \mbox{in} \ \cud_{n+1},
\end{equation}
and satisfies the estimates
\begin{multline} 
\label{e.weakphicontrol}
3^{-n} \left( \left\| \nabla \phi_{e,n} \right\|_{\un {\hat H}^{-1}_\pa(\cud_{n+1})} 
+ 
 \left\| \a \left(e+ \nabla \phi_{e,n} \right)- \ahom e \right\|_{\un {\hat H}^{-1}_\pa(\cud_{n+1})} 
 \right)
\\
\leq
C3^{-n} + C \sum_{m = 0}^{n} 3^{m-n} \Ll(  \left|\mcl Z_m\right|^{-1} \sum_{z\in \mcl Z_m} \left( \mathcal{E}\left( z+\cud_m \right) \right) \Rr)^{\frac 12}.
\end{multline}
By the previous two displays, Proposition~\ref{p.integrate} and $\left( \phi_{e,n} \right)_{\cud_n} = 0$, we also have
\begin{equation} 
\label{e.L2phicontrol}
3^{-n} \left\|  \phi_{e,n} \right\|_{\un L^2(\cud_n)} 
\leq
C3^{-n} + C \sum_{m = 0}^{n} 3^{m-n} \Ll(  \left|\mcl Z_m\right|^{-1} \sum_{z\in \mcl Z_m} \left( \mathcal{E}\left( z+\cud_m \right) \right) \Rr)^{\frac 12}.
\end{equation}

\subsection{{The proof of Theorem~\ref{t.CDP}}}

The main step in the proof of Theorem~\ref{t.CDP} is the following proposition. It is a deterministic estimate of the homogenization error in terms of the error in the convergence of the correctors defined in the previous subsection. Since we have already estimated the latter in~\eqref{e.estimatemonster},~\eqref{e.weakphicontrol} and~\eqref{e.L2phicontrol}, this is sufficient to imply the theorem. It is convenient to denote, for every $m\in\N$,
\begin{multline*} \label{}
\mathcal{E}'(m):= 3^{-m} \sum_{k=1}^d \bigg( \left\| \phi_{e_k,m}  \right\|_{\un L^2(\cud_m)} 
+  \left\| \nabla \phi_{e_k,m}  \right\|_{\un {\hat H}^{-1}_\pa(\cud_m)} 
\\
+  \left\| \a \left(e_k+ \nabla \phi_{e_k,m} \right)- \ahom e_k \right\|_{\un {\hat H}^{-1}_\pa(\cud_m)} 
\bigg).
\end{multline*}
We also set, for each $\ep>0$, 
\begin{equation} 
\label{e.rescaleeps}
\a^\ep(t,x):= \a\left(\frac t{\ep^2}, \frac x\ep \right)
\quad \mbox{and} \quad
\phi^\ep_{e,n}(t,x):= \ep \phi_{e,n} \left( \frac{t}{\ep^2}, \frac x\ep \right).
\end{equation}

\begin{proposition}
\label{p.blackbox}
Fix a bounded interval $I:=(I_-,0)\subseteq \left( -\frac14,0 \right)$, a bounded Lipschitz domain $U \subseteq \cu_0$, a small parameter~$\ep \in \left(0,\tfrac12\right]$, an exponent~$\delta>0$ and a initial-boundary condition~$f\in W^{1,2+\delta}_{\pa} \left(I \times U\right)$. Let 
$$u^\ep,u \in f + H^1_{\pa, \sqcup}\left( I \times U\right)$$
respectively denote the solutions of
\begin{equation} 
\label{e.CDP}
\left\{
\begin{aligned}
& \partial_t u^\ep - \nabla \cdot \left( \a^\ep \nabla u^\ep  \right) = 0 & \mbox{in} & \ I \times U,
\\ 
& u^\ep = f & \mbox{on} & \ \partial_\sqcup \left( I \times U \right),
\end{aligned}
\right.
\end{equation}
and
\begin{equation}
\label{e.CDPhom}
\left\{
\begin{aligned}
& \partial_t u - \nabla \cdot \left( \ahom \nabla u  \right) = 0 & \mbox{in} & \ I \times U,
\\ 
& u = f & \mbox{on}  & \ \partial_\sqcup \left( I \times U \right).
\end{aligned}\right.
\end{equation}
Let $n\in\N$ be such that $3^{-n} \leq \ep < 3^{-(n+1)}$. Then there exist $\beta(\delta,d,\Lambda)>0$ and $C(I,U,\delta,d,\Lambda)<\infty$ such that, for every~$r\in (0,1)$, we have the estimate
\begin{multline} 
\label{e.blackbox}
\left\|  \nabla u^\ep - \nabla u \right\|_{{\hat H}^{-1}_\pa(I\times U)}
+ \left\| \a^\ep \nabla u^\ep - \ahom \nabla u \right\|_{{\hat H}^{-1}_\pa(I\times U)}
+ \left\| u^\ep - u \right\|_{{L}^2(I\times U)}
\\
\leq
C \left\| f \right\|_{W^{1,2+\delta}_\pa(I\times U)}  
\left(
r^\beta + \frac1{r^{3+(2+d)/2}} \mathcal{E}'(n) 
\right).
\end{multline}
\end{proposition}
\begin{proof}
With $n$ fixed as in the statement of the proposition, we let $\phi_{e} = \phi_{e,n}$ denote, for each~$e\in\Rd$, the (finite-volume) corrector defined in the previous subsection (we will not display its dependence on~$n$). We also use the notation $\phi_{e}^\ep=\phi_{e,n}^\ep$ as in~\eqref{e.rescaleeps}. 

\smallskip

We will argue that~$u^\ep$ is close to its modified two-scale expansion suitably cut off near the boundary. The latter is defined by
\begin{align} 
\label{e.twoscale.def}
w^\ep(t,x)
&
:= 
u(t,x) + \ep \zeta_r (t,x) \sum_{k=1}^d \partial_{x_k} u(t,x) \phi_{e_k} \left( \frac t{\ep^2},\frac x \ep \right)
\\ & \notag
= u(t,x) +  \zeta_r (t,x) \sum_{k=1}^d \partial_{x_k} u(t,x) \phi_{e_k}^\ep (t,x),
\end{align}
where $r\in (0,1)$ is the free parameter (representing a mesoscopic scale) given in the proposition, and we denote
\begin{equation*} \label{}
U_r:= \left\{ x\in U \,:\, \dist (x,\partial U) > r \right\}
\quad \mbox{and} \quad 
I_r:= \left( I_- +r^2,I_+\right),
\end{equation*}
where the cutoff function $\zeta_r$ is selected so that
\begin{equation} 
\label{e.cutoffzeta.parab}
\left\{ 
\begin{aligned}
& 0\leq \zeta_r \leq 1,  \ \  
\zeta_r = 1  \ \mbox{in} \ I_{2r} \times U_{2 r}, \\  & 
\zeta_r \equiv 0  \  \mbox{in} \ (I\times U) \setminus \left( I_r \times U_{r} \right),
\\ & 
\forall k,l\in\N, \quad  \left|  \nabla^k \partial^l_t \zeta_r \right| \leq C_{k+2l} r^{-(k+2l)}.
\end{aligned}
\right.
\end{equation}
Note that the constant $C_m$ here depends on~$(I,U,d)$ in addition to~$m\in\N$.

\smallskip

\emph{Step 0.} We record some standard estimates from the deterministic regularity theory for uniformly parabolic equations that are needed below. The global Meyers estimate (see Proposition~\ref{p.globalmeyers}) gives us~$\delta_0(U,d,\Lambda)>0$ such that $\delta \leq \delta_0$ implies that 
\begin{equation} 
\label{e.globalmeyers.app}
\left\| \nabla u^\ep \right\|_{{L}^{2+\delta}(I\times U)} 
+ 
\left\| \nabla u \right\|_{{L}^{2+\delta}(I\times U)} 
\leq C  \left\| f \right\|_{W^{1,2+\delta}_\pa(I\times U)} .
\end{equation}
We henceforth assume without loss of generality that $\delta\leq \delta_0$ so that~\eqref{e.globalmeyers.app} holds. We also need pointwise derivative estimates for constant-coefficient parabolic equations. These can be found for instance in~\cite[Section 2.3.3.c]{Evans} (note that estimates for the operator $\partial_t - \nabla \cdot \ahom\nabla$ are implied by estimates for the heat equation, by a simple affine change of variables), and they yield, for every $m,l\in\N$, 
\begin{align} 
\label{e.nablam.u.pointswise}
\left\| \partial_t^l \nabla^m u \right\|_{L^\infty(I_r \times U_{r})}
&
\leq C_{m+2l} r^{-m-2l} r^{-(2+d)/2} \left\| u \right\|_{ L^2(I\times U)}
\\ & \notag
\leq C_{m+2l} r^{-m-2l}r^{1-(2+d)/2} \left\| \nabla u \right\|_{ L^2(I\times U)}. 
\end{align}
Here~$C_k$ depends only on~$(d,\Lambda)$ in addition to~$k\in\N$. 

\smallskip

The main step in the proof is to obtain an estimate on~$\left\| 
u^\ep - w^\ep 
\right\|_{H^1_\pa(I\times U)}$, which is stated below in~\eqref{e.bigH1bound}.

\emph{Step 1.} We plug $w^\ep$ into the heterogeneous equation and estimate the error. The  claim is that we can write $\left( \partial_t - \nabla \cdot \a^\ep \nabla  \right)  w^\ep$ in the form 
\begin{equation*}
\left( \partial_t - \nabla \cdot \a^\ep \nabla  \right)  w^\ep
= \partial_t F + G
\end{equation*}
where $F\in H^1_{\pa,\sqcup}(I\times U)$ and $G \in L^2(I;H^{-1}(U))$ satisfy the estimates
\begin{equation}
\label{e.bigHminusone}
\left\| F \right\|_{L^2(I;H^1(U))} + \left\| G \right\|_{L^2(I;H^{-1}(U))} 
\leq 
C\left(r^{\frac{\delta}{4+2\delta}}  +  r^{-3-(2+d)/2}\mathcal{E}'(n)  \right) \left\| f \right\|_{W^{1,2+\delta}_\pa(I\times U)}. 
\end{equation}
We begin by computing
\begin{equation*} \label{}
\left\{ 
\begin{aligned}
& \nabla w^\ep =  \zeta_r \sum_{k=1}^d \left( e_k + \nabla \phi_{e_k}^\ep \right) \partial_{x_k}u 
+ \sum_{k=1}^d \phi_{e_k}^\ep \nabla \left( \zeta_r \partial_{x_k} u\right) 
+ (1-\zeta_r) \nabla u, 
\\ & 
\partial_t w^\ep = \partial_t u + \zeta_r  \sum_{k=1}^d \partial_t \phi_{e_k}^\ep \partial_{x_k} u + \sum_{k=1}^d \phi_{e_k}^\ep \partial_t \left( \zeta_r \partial_{x_k} u  \right). 
\end{aligned}
\right.
\end{equation*}
According to~\eqref{e.eqnforuen}, the map $\hat{u}^\ep_{e}(t,x) := e\cdot x + \phi_e^\ep(t,x)$ is a solution of the equation
\begin{equation*} \label{}
\partial_t \hat{u}^\ep_{e} - \nabla \cdot\left( \a^\ep \nabla \hat{u}^\ep_e \right) = 0 
\quad \mbox{in} \ I\times U. 
\end{equation*}
Therefore we find that 
\begin{align*}
\partial_t w^\ep - \nabla \cdot \left( \a^\ep\nabla w^\ep \right)
& =
\partial_t u + \sum_{k=1}^d \phi_{e_k}^\ep \partial_t \left( \zeta_r \partial_{x_k} u  \right)
- \sum_{k=1}^d \nabla \left( \zeta_r \partial_{x_k}u \right) \cdot \a^\ep \left( e_k+\nabla \phi^\ep_{e_k} \right)
\\ & \qquad
- \nabla \cdot \left( \a^\ep \left(  \sum_{k=1}^d \phi_{e_k}^\ep \nabla \left( \zeta_r \partial_{x_k} u\right) + (1-\zeta_r)\nabla u \right) \right) .
\end{align*}
Since $u$ satisfies the homogenized equation, we have furthermore that 
\begin{equation*} \label{}
\partial_t u = \nabla \cdot \ahom \nabla u = \sum_{k=1}^{d} \nabla \left( \zeta_r \partial_{x_k}u \right) \cdot \ahom e_k + \nabla \cdot \left( ( 1-\zeta_r) \ahom \nabla u \right),
\end{equation*}
and this gives us the identity 
\begin{align*}
\partial_t w^\ep - \nabla \cdot \left( \a^\ep\nabla w^\ep \right)
& =
 \sum_{k=1}^d \phi_{e_k}^\ep \partial_t \left( \zeta_r \partial_{x_k} u  \right)
- \sum_{k=1}^d \nabla \left( \zeta_r \partial_{x_k}u \right) \cdot \left( \a^\ep \left( e_k+\nabla \phi^\ep_{e_k} \right) - \ahom e_k \right)
\\ & \quad
- \nabla \cdot \left( \a^\ep \sum_{k=1}^d \phi_{e_k}^\ep \nabla \left( \zeta_r \partial_{x_k} u\right) \right)
- \nabla \cdot \left( \left( \a^\ep - \ahom \right) (1-\zeta_r)\nabla u \right).
\end{align*}
According to Lemma~\ref{l.identif.H-1par}, we can find~$v\in L^2(I;H^1_0(U))$ and $v^* \in L^2(I;H^{-1}(U))$ such that 
\begin{equation}
\label{e.sumdecomp}
f^* := - \sum_{k=1}^d \nabla \left( \zeta_r \partial_{x_k}u \right) \cdot \left( \a^\ep \left( e_k+\nabla \phi^\ep_{e_k} \right) - \ahom e_k \right)
=
\partial_t v + v^*
\end{equation}
and 
\begin{equation*}
\left\| v \right\|_{L^2(I;H^1(U))} + \left\| v^*\right\|_{L^2(I;H^{-1}(U))} 
\leq 
C \left\| f^* \right\|_{{\hat H}^{-1}_\pa(I\times U)}.
\end{equation*}
The lemma allows us to take~$v$ and $v^*$ to vanish in a neighborhood of the parabolic boundary of~$I\times U$. Since the left side of~\eqref{e.sumdecomp} belongs to $L^2(I;H^{-1}(U))$, we have also that $v \in H^1_{\pa,\sqcup}(I\times U)$. 
Therefore we obtain that 
\begin{equation*}
\partial_t w^\ep - \nabla \cdot \left( \a^\ep\nabla w^\ep \right)
= \partial_t F + G
\end{equation*}
where 
\begin{equation*}
F:= v
\end{equation*}
and 
\begin{align*}
G & 
:= 
v^* + 
\sum_{k=1}^d \phi_{e_k}^\ep \partial_t \left( \zeta_r \partial_{x_k} u  \right)
- \nabla \cdot \left( \a^\ep \sum_{k=1}^d \phi_{e_k}^\ep \nabla \left( \zeta_r \partial_{x_k} u\right) \right)
- \nabla \cdot \left( \left( \a^\ep - \ahom \right) (1-\zeta_r)\nabla u \right).
\end{align*}
It is clear that 
\begin{align*}
\lefteqn{
\left\| F \right\|_{L^2(I;H^1(U))} + \left\| G \right\|_{L^2(I;H^{-1}(U))} 
}  & 
\\ & 
\leq
C\sum_{k=1}^d \left\|  \phi_{e_k}^\ep \partial_t \left( \zeta_r \partial_{x_k} u  \right) \right\|_{L^2(I\times U)}
+ C\sum_{k=1}^d \left\|  \nabla \left( \zeta_r \partial_{x_k}u \right) \cdot \left( \a^\ep \left( e_k+\nabla \phi^\ep_{e_k} \right) - \ahom e_k \right) \right\|_{{\hat H}^{-1}_\pa(I\times U)}
\\ & \quad
+ C\left\| \a^\ep \sum_{k=1}^d \phi_{e_k}^\ep \nabla \left( \zeta_r \partial_{x_k} u\right)  \right\|_{L^2(I\times U)}
+C\left\|  \left( \a^\ep - \ahom \right) (1-\zeta_r)\nabla u \right\|_{L^2(I\times U)}
\\ & 
= T_1 + T_2 + T_3 + T_4. 
\end{align*}
We will now show that each of the four terms~$T_i$ can be estimated by the right side of~\eqref{e.bigHminusone}, using the definition of~$\mathcal{E}'(n)$ and the bounds~\eqref{e.cutoffzeta.parab},~\eqref{e.globalmeyers.app} and~\eqref{e.nablam.u.pointswise}. For $T_1$, we use~\eqref{e.cutoffzeta.parab} and~\eqref{e.nablam.u.pointswise} to find that, for each $e\in\partial B_1$, 
\begin{align*} \label{}
\left\|  \phi_{e_k}^\ep \partial_t \left( \zeta_r \partial_{x_k} u  \right) \right\|_{L^2(I\times U)}
&
\leq C \left\| \partial_t \left( \zeta_r \partial_{x_k} u  \right) \right\|_{L^\infty(I\times U)} \left\| \phi_{e}^\ep \right\|_{L^2(I\times U)}
\\ &
\leq Cr^{-3-(2+d)/2} \mathcal{E}'(n) \left\| f \right\|_{H^{1}_\pa(I\times U)}. 
\end{align*}
For $T_2$, we have 
\begin{align*} \label{}
\lefteqn{
\left\|  \nabla \left( \zeta_r \partial_{x_k}u \right) \cdot \left( \a^\ep \left( e+\nabla \phi^\ep_{e} \right) - \ahom e \right) \right\|_{{\hat H}^{-1}_\pa(I\times U)}
} \quad & 
\\ & 
\leq
C \left\| 
\nabla \left( \zeta_r \partial_{x_k}u \right)
\right\|_{W^{1,\infty}(I\times U)} 
\left\| \a^\ep \left( e+\nabla \phi^\ep_{e} \right) - \ahom e  \right\|_{{\hat H}^{-1}_\pa(I\times U)}
\\ & 
\leq Cr^{-3-(2+d)/2}  \mathcal{E}'(n) \left\| f \right\|_{H^{1}_\pa(I\times U)}.
\end{align*}
For $T_3$, we use~\eqref{e.cutoffzeta.parab} and~\eqref{e.nablam.u.pointswise} again to get
\begin{align*} \label{}
\left\| \a^\ep  \phi_{e}^\ep \nabla \left( \zeta_r \partial_{x_k} u\right)  \right\|_{L^2(I\times U)}
&
\leq C \left\| \nabla \left( \zeta_r \partial_{x_k} u\right)  \right\|_{L^\infty(I\times U)} \left\|\phi_{e}^\ep  \right\|_{L^2(I\times U)}
\\ &
\leq C r^{-2-(2+d)/2} \mathcal{E}'(n) \left\| f \right\|_{H^{1}_\pa(I\times U)}. 
\end{align*}
Finally, for $T_4$, we use~\eqref{e.cutoffzeta.parab},~\eqref{e.globalmeyers.app} and H\"older's inequality to get
\begin{align*}
\left\|  \left( \a^\ep - \ahom \right) (1-\zeta_r)\nabla u \right\|_{L^2(I\times U)}
&
\leq C\left| \left\{x\in I\times U \,:\, \zeta_r(x) \neq 1  \right\}\right|^{\frac{\delta}{4+2\delta}} \left\| \nabla u  \right\|_{L^{2+\delta}(I\times U)} 
\\ & 
\leq 
C r^{\frac\delta{4+2\delta}}
\left\| f \right\|_{W^{1,2+\delta}_\pa(I\times U)}.
\end{align*}
This completes the proof of~\eqref{e.bigHminusone}.

\smallskip

\emph{Step 2.} We deduce that
\begin{equation} 
\label{e.bigH1bound}
\left\| 
u^\ep - w^\ep 
\right\|_{L^2(I ; H^1(U))} 
\leq C \left(r^{\frac{\delta}{4+2\delta}} +  r^{-3-(2+d)/2} \mathcal{E}'(n)  \right) \left\| f \right\|_{W^{1,2+\delta}_\pa(I\times U)}.
\end{equation}
This is an immediate consequence of the estimate~\eqref{e.bigHminusone} proved in the previous step, the fact that~$u^\ep - w^\ep, F \in H^1_{\pa,\sqcup}(I\times U)$ and the estimate~\eqref{e.generalsolutionestimate} proved in the appendix. 

\smallskip

At this point, we have succeeded in comparing~$u^\ep$ to~$w^\ep$. What is left is to compare $w^\ep$ to $u$ by showing that the second term on the right side of~\eqref{e.twoscale.def} is small. This is relatively straightforward to obtain from~\eqref{e.weakphicontrol} and~\eqref{e.L2phicontrol}. 

\smallskip

\emph{Step 3.} We show that 
\begin{multline} 
\label{e.gradsandfluxes}
\left\| u - w^\ep  \right\|_{L^2(I\times U)} 
+ \left\| \nabla u - \nabla w^\ep \right\|_{{\hat H}^{-1}_\pa(I\times U)} 
+ \left\| \ahom \nabla u - \a^\ep \nabla w^\ep \right\|_{{\hat H}^{-1}_\pa(I\times U)} 
\\
\leq 
C  \left(r^{\frac{\delta}{4+2\delta}} +  r^{-3-(2+d)/2} \mathcal{E}'(n)  \right) \left\| f \right\|_{W^{1,2+\delta}_\pa(I\times U)}.
\end{multline}
We use the formula 
\begin{multline*} \label{}
\nabla w^\ep(t,x) - \nabla u(t,x)
\\ 
=
\ep  \sum_{k=1}^d \nabla \left( \zeta_r \partial_{x_k} u \right)(t,x) \phi_{e_k} \left( \tfrac t{\ep^2},\tfrac x \ep \right)
+ \zeta_r(t,x) \sum_{k=1}^d \partial_{x_k} u(t,x) \nabla \phi_{e_k}  \left( \tfrac t{\ep^2},\tfrac x \ep \right) 
\end{multline*}
to get
\begin{align*} \label{}
\lefteqn{
\left\| \nabla u - \nabla w^\ep \right\|_{{\hat H}^{-1}_\pa(I\times U)} 
} \ & 
\\ & 
 \leq \left\| \nabla \left( \zeta_r \nabla u \right) \right\|_{L^\infty(I\times U)} \ep  \sum_{k=1}^d \left\| \phi_{e_k} \right\|_{\underline{L}^2\left(Q_{\ep^{-1}}\right)}
+ C\left\| \zeta_r \nabla u \right\|_{W^{1,\infty}(I\times U)}  \left\| \nabla \phi_{e_k}\left( \tfrac\cdot{\ep^2},\tfrac \cdot{\ep}\right) \right\|_{{\hat H}^{-1}_\pa(I\times U)} 
\\ & 
\leq Cr^{-2} \left\| f \right\|_{H^1_\pa(I\times U)} \mathcal{E}'(n). 
\end{align*}
For the fluxes, we find it convenient to use coordinates. We have 
\begin{align*}
\left( \a^\ep \nabla w^\ep \right)_i(t,x)
&
= \sum_{j,k=1}^d 
\zeta_r(t,x)  \a^\ep_{ij}(t,x) \partial_{x_k} u(t,x) \left( \delta_{jk} +  \partial_{x_j} \phi_{e_k}\left( \tfrac t{\ep^2}, \tfrac x\ep \right)  \right)
\\ & \quad
+\ep \sum_{j,k=1}^d
\a^\ep_{ij} (t,x) \partial_{x_j} \left( \zeta_r \partial_{x_k} u \right) (t,x)  \phi_{e_k}\left( \tfrac t{\ep^2}, \tfrac x\ep \right).
\end{align*}
Thus
\begin{align*}
\lefteqn{
\left( \a^\ep \nabla w^\ep \right)_i(t,x) - \left( \ahom\nabla u\right)_i(t,x)
} \quad & 
\\ & 
=  \sum_{j,k=1}^d 
\zeta_r(t,x)  \partial_{x_k} u(t,x) \left(  \a^\ep_{ij}(t,x) \left( \delta_{jk} +  \partial_{x_j} \phi_{e_k}\left( \tfrac t{\ep^2}, \tfrac x\ep \right)  \right) - \ahom_{ik} \right) 
\\ & \quad
+ \sum_{j,k=1}^d 
\left( 1 - \zeta_r(t,x) \right) \ahom_{ik} \partial_{x_k} u(t,x)
+\ep \sum_{j,k=1}^d
\a^\ep_{ij} (t,x) \partial_{x_j} \left( \zeta_r \partial_{x_k} u \right) (t,x)  \phi_{e_k}\left( \tfrac t{\ep^2}, \tfrac x\ep \right).
\end{align*}
We can easily estimate the last two terms on the right side using~\eqref{e.cutoffzeta.parab},~\eqref{e.globalmeyers.app},~\eqref{e.nablam.u.pointswise} and the H\"older inequality. We have
\begin{align*}
\sum_{j,k=1}^d 
\left\| \left( 1 - \zeta_r  \right) \ahom_{ik} \partial_{x_k} u \right\|_{L^2(I\times U)}
&
\leq C \left\| \nabla u \right\|_{L^2\left( (I\times U) \setminus (I_r\times U_r) \right)}
\\ & 
\leq C r^{\frac{\delta}{4+2\delta}} \left\| \nabla u \right\|_{L^{2+\delta}\left(I\times U\right)}
\leq
C r^{\frac{\delta}{4+2\delta}}  \left\| f \right\|_{W^{1,2+\delta}_\pa(I\times U)} 
\end{align*}
and
\begin{align*}
\ep \sum_{j,k=1}^d
\left\| \a^\ep_{ij}  \partial_{x_j} \left( \zeta_r \partial_{x_k} u \right)  \phi_{e_k}\left( \tfrac \cdot{\ep^2}, \tfrac \cdot\ep \right) \right\|_{L^2(I\times U)} 
&
\leq C \left\| \nabla \left( \zeta_r \nabla u \right) \right\|_{L^\infty(I\times U)} \ep \sum_{k=1}^d \left\| \phi_{e_k} \right\|_{\underline{L}^2\left(Q_{1/2}\right)}
\\ & 
\leq Cr^{-2-(2+d)/2} \left\| f \right\|_{H^1_\pa(I\times U)} \mathcal{E}'(n). 
\end{align*}
For the first term, we have 
\begin{align*}
\lefteqn{
\sum_{j,k=1}^d 
\left\| 
\zeta_r   \partial_{x_k} u  \left(  \a^\ep_{ij}  \left( \delta_{jk} +  \partial_{x_j} \phi_{e_k}\left( \tfrac \cdot{\ep^2}, \tfrac \cdot \ep \right)  \right) - \ahom_{ik} \right) 
\right\|_{{\hat H}^{-1}_\pa(I\times U)} 
} \quad &
\\ & 
\leq 
C \sum_{j,k=1}^d 
\left\| \zeta_r \partial_{x_k} u \right\|_{W^{1,\infty}(I\times U)} 
\left\| \a^\ep_{ij}  \left(\delta_{jk} +  \partial_{x_j} \phi_{e_k}\left( \tfrac \cdot{\ep^2}, \tfrac \cdot \ep \right)  \right) - \ahom_{ik} \right\|_{{\hat H}^{-1}_\pa(I\times U)} 
\\ & 
\leq Cr^{-3} \left\| f \right\|_{H^1_\pa(I\times U)}  \sum_{k=1}^d \left\| \a^\ep\left( e_k + \nabla \phi_{e_k} \right) -\ahom e_k \right\|_{{\hat H}^{-1}_\pa\left(Q_{1/2}\right)} 
\\ &
\leq Cr^{-3-(2+d)/2}\left\| f \right\|_{H^1_\pa(I\times U)}  \mathcal{E}'(n). 
\end{align*}
Combining the previous four displays, we obtain
\begin{equation*} \label{}
\left\| \a^\ep \nabla w^\ep - \ahom \nabla u \right\|_{{\hat H}^{-1}_\pa(I\times U)} 
\leq C  \left(r^{\frac{\delta}{4+2\delta}} +  r^{-3-(2+d)/2} \mathcal{E}'(n)  \right) \left\| f \right\|_{W^{1,2+\delta}_\pa(I\times U)}.
\end{equation*}
Finally, for the estimate of $w^\ep - u$, we have 
\begin{align*} \label{}
\left\| w^\ep - u \right\|_{L^2(I\times U)} 
& \leq C \left\| \nabla u \right\|_{L^\infty((I\times U) \setminus (I_r\times U_r))} \ep \sum_{k=1}^d \left\| \phi_{e_k} \left( \tfrac \cdot{\ep^2},\tfrac \cdot \ep \right) \right\|_{L^2(I\times U)} 
\\ & 
\leq C r^{-1-(2+d)/2}\left\| f \right\|_{H^1_\pa(I\times U)} \mathcal{E}'(n).
\end{align*}
This completes the proof of~\eqref{e.gradsandfluxes}.

\smallskip

\emph{Step 4.} We summarize and conclude the argument. According to~\eqref{e.bigH1bound},~\eqref{e.gradsandfluxes} and the triangle inequality, we have 
\begin{align*}
\left\| \nabla u^\ep - \nabla u \right\|_{{\hat H}^{-1}_\pa(I\times U)} 
& 
\leq \left\| \nabla u^\ep - \nabla w^\ep \right\|_{{\hat H}^{-1}_\pa(I\times U)} 
+ \left\| \nabla w^\ep - \nabla u \right\|_{{\hat H}^{-1}_\pa(I\times U)} 
\\ & 
\leq C \left\| u^\ep - w^\ep \right\|_{L^2(I; H^1( U))} + \left\| \nabla w^\ep - \nabla u \right\|_{{\hat H}^{-1}_\pa(I\times U)} 
\\ & 
\leq C \left(r^{\frac{\delta}{4+2\delta}} +  r^{-3-(2+d)/2} \mathcal{E}'(n)  \right) \left\| f \right\|_{W^{1,2+\delta}_\pa(I\times U)}.
\end{align*}
Similarly, for the fluxes we have 
\begin{align*} \label{}
\left\| \a^\ep \nabla u^\ep - \ahom \nabla u \right\|_{{\hat H}^{-1}_\pa(I\times U)} 
&
\leq \left\| \a^\ep \nabla u^\ep - \a^\ep \nabla w^\ep \right\|_{{\hat H}^{-1}_\pa(I\times U)} 
+ \left\| \a^\ep \nabla w^\ep - \ahom \nabla u \right\|_{{\hat H}^{-1}_\pa(I\times U)} 
\\ & 
\leq C  \left\| u^\ep - w^\ep \right\|_{L^2(I; H^1( U))} + \left\| \a^\ep \nabla w^\ep -\ahom \nabla u \right\|_{{\hat H}^{-1}_\pa(I\times U)} 
\\ & 
\leq C \left(r^{\frac{\delta}{4+2\delta}} +  r^{-3-(2+d)/2} \mathcal{E}'(n)  \right) \left\| f \right\|_{W^{1,2+\delta}_\pa(I\times U)},
\end{align*}
and, for the homogenization error, we have 
\begin{align*}
\left\| u^\ep - u \right\|_{L^2(I\times U)} 
& 
\leq \left\| u^\ep - w^\ep \right\|_{L^2(I\times U)} 
+ \left\| w^\ep - u \right\|_{L^2(I\times U)} 
\\ & 
\leq C \left(r^{\frac{\delta}{4+2\delta}} +  r^{-3-(2+d)/2} \mathcal{E}'(n)  \right) \left\| f \right\|_{W^{1,2+\delta}_\pa(I\times U)}.
\end{align*}
This completes the proof of the proposition. 
\end{proof}

To complete the proof of Theorem~\ref{t.CDP}, we just need to estimate the random variables on the right side of~\eqref{e.blackbox} using Proposition~\ref{p.controlofEm} and the estimates~\eqref{e.weakphicontrol} and~\eqref{e.L2phicontrol} for the correctors. 

\begin{proof}[{Proof of Theorem~\ref{t.CDP}}]
Fix $s \in \left( 0, 2+d\right)$ and put $s':=\frac12 s + \frac12(2+d)$ and $s'':=\frac12 s' + \frac12(2+d)$. Thus $s < s' < s'' < 2+d$ and the gaps are at least of size $\frac14(2+d-s)$.
Observe that~\eqref{e.weakphicontrol} and~\eqref{e.L2phicontrol} imply that 
\begin{equation*} \label{}
\mathcal{E}'(n) \leq C3^{-n} + C \sum_{m = 0}^{n-1} 3^{m-n} \Ll(  \left|\mcl Z_m\right|^{-1} \sum_{z\in \mcl Z_m} \left( \mathcal{E}\left( z+\cud_m \right) \right) \Rr)^{\frac 12}.
\end{equation*}
Thus, by Proposition~\ref{p.controlofEm},
\begin{equation*} \label{}
\mathcal{E}'(n)\leq C3^{-n\beta(2+d-s'')} + \O_1\left(C3^{-ns''} \right).
\end{equation*}
Thus
\begin{equation*} \label{}
3^{ns'} \left( \mathcal{E}'(n) -  C3^{-n\beta(2+d-s'')} \right)_+ \leq \O_1\left(C3^{-n(s''-s')} \right).
\end{equation*}
By~\eqref{e.Osums}, 
\begin{equation*} \label{}
\X:= \sum_{n\in\N} 3^{ns'} \left( \mathcal{E}'(n) -  C3^{-n\beta(2+d-s'')} \right)_+ \leq \O_1\left(C \right).
\end{equation*}
Proposition~\ref{p.blackbox} yields therefore that, for every $\ep \in \left( 0,\tfrac12 \right]$ and $r\in (0,1)$,
\begin{multline*}
\left\|  \nabla u^\ep - \nabla u \right\|_{{\hat H}^{-1}_\pa(I\times U)}
+ \left\| \a^\ep \nabla u^\ep - \ahom \nabla u \right\|_{{\hat H}^{-1}_\pa(I\times U)}
+ \left\| u^\ep - u \right\|_{\underline{L}^2(I\times U)}
\\
\leq
C \left\| f \right\|_{W^{1,p}_\pa(I\times U)} 
\left(
r^\beta + \frac1{r^{3+(2+d)/2}} \left( \ep^{\beta(2+d-s'')} + \ep^{s'} \X'\right) 
\right).
\end{multline*}
We now select $r\in (0,1)$ as small as possible (it must be no larger than a positive power of $\ep$) such that $r^{-3-(2+d)/2} \ep^{s'} \leq \ep^s$ and $r^{-3-(2+d)/2} \leq \ep^{-\beta(2+d-s'')/2}$. We can take for example
\begin{equation*} \label{}
r:= \ep^{\beta(2+d-s'')/(3+(2+d)/2)} \vee \ep^{(s'-s)/(3+(2+d)/2)}. 
\end{equation*}
Recalling that $2+d-s'' \geq \frac14 (2+d-s)$ and $s'-s\geq \frac14 (2+d-s)$, we obtain the theorem. 
\end{proof}

\section{Regularity theory}
\label{s.regularity}

In this section, we sketch the proof of Theorem~\ref{t.regularity}, following along the lines of the argument given in the proof of~\cite[Theorem 3.6]{AKMBook} in the elliptic case. We do not give full details, since this would involve an almost verbatim repetition of the proof of the latter.

\smallskip

We begin by reformulating Theorem~\ref{t.CDP} in a slightly different way in terms of \emph{caloric approximation} which is more convenient for its application in this section. The next statement can be compared to its elliptic analogue in~\cite[Proposition 3.2]{AKMBook}.

\begin{proposition}
[{Caloric approximation}]
\label{p.caloricapproximation}
Fix $s\in (0,2+d)$. There exist an exponent~$\alpha(d,\Lambda)>0$, a constant~$C(s,d,\Lambda)<\infty$,  and a random variable $\X_s:\Omega \to [1,\infty]$ satisfying 
the estimate
\begin{equation}
\label{e.Xcaloricapproximation}
\X_s = \O_s\left( C \right),
\end{equation}
such that the following holds: for every $R\geq \X_s$ and weak solution $u\in H^1_\pa(Q_{R})$ of
\begin{equation} 
\label{e.wksrt.parab}
\partial_t u -\nabla \cdot \left( \a\nabla u \right) = 0 \quad \mbox{in} \ Q_{R},
\end{equation}
there exists a solution $\overline{u} \in H^1_\pa(Q_{R/2})$ of the equation 
\begin{equation*} \label{}
\partial_t \bar{u} -\nabla \cdot \left( \ahom \nabla \bar u \right) = 0 \quad \mbox{in} \ Q_{R/2}
\end{equation*}
such that 
\begin{equation}
\label{e.caloricapproximation}
\left\| u - \overline{u}  \right\|_{\underline{L}^2(Q_{R/2})} \leq CR^{-\alpha(2+d-s)} \left\|  u - \left( u \right)_{Q_{R}} \right\|_{\underline{L}^2(Q_{R})}.
\end{equation}
\end{proposition}
\begin{proof}
This is a simple application of Theorem~\ref{t.CDP} combined with the parabolic Meyers estimate. 
The argument is almost the same as in the elliptic case presented in~\cite[Proposition 3.2]{AKMBook}, we just need to replace the elliptic interior Meyers estimate  with its parabolic analogue proved in Proposition~\ref{p.interiormeyers} below. The latter gives us~$\delta(d,\Lambda)>0$ and $C(d,\Lambda)>0$ such that, for every $u\in H^1_\pa (Q_{R})$ satisfying~\eqref{e.wksrt.parab}, we have that $\nabla u \in L^{2+\delta}(Q_{R/2})$ and the estimate 
\begin{equation}
\label{e.meyersrt.parab}
\left\| \nabla u \right\|_{\underline{L}^{2+\delta}(Q_{R/2})} \leq \frac{C}{R} \left\|  u - \left( u \right)_{Q_{R}} \right\|_{\underline{L}^2(Q_{R})}. 
\end{equation}
Following the proof of~\cite[Proposition 3.2]{AKMBook}, using~\eqref{e.meyersrt.parab}, substituting Theorem~\ref{t.CDP} in place of~\cite[Theorem 2.16]{AKMBook}  and making obvious changes to the notation, we obtain the proposition.
\end{proof} 

We next state a parabolic counterpart of~\cite[Lemma 3.5]{AKMBook}.

\begin{lemma}
\label{l.u:decayestimate.parab}
Fix $\alpha \in [0,1]$,  $K\geq 1$ and $X\geq 1$. 
Let $R\geq 2X$ and $u\in L^2(Q_{R})$ have the property that, for every $r\in \left[ X, R \right]$, there exists $w_r \in H^1_{\pa}(Q_{r/2})$ which is a solution of 
\begin{equation*} \label{}
\partial_t w_r -\nabla \cdot \left( \ahom \nabla w_r \right) = 0 \quad \mbox{in} \ Q_{r/2}
\end{equation*}
and satisfies 
\begin{equation} 
\label{e.wharmapproxit02.parab}
\left\| u - w_r \right\|_{\underline{L}^2(Q_{r/2})} \leq K r^{-\alpha} \left\| u - \left( u \right)_{Q_{r}} \right\|_{\underline{L}^2(Q_{r})}. 
\end{equation}
Then, for every $k \in \N$, there exists $\theta(\alpha,k,d,\Lambda) \in (0,\frac12)$  and $C(\alpha,k,d,\Lambda)<\infty$ such that, for every $r\in \left[ X, R\right]$, 
\begin{multline}
\label{e.u:decayestimate.parab}
\inf_{p \in \Ahom_k(Q_\infty) } \left\| u - p \right\|_{\underline{L}^2(Q_{\theta r})}
\\
\leq
\frac14 \theta^{k+1 - \alpha/2} \inf_{p \in \Ahom_k(Q_\infty) } \left\| u - p  \right\|_{\underline{L}^2(Q_{r})}   + C K r^{-\alpha} \left\| u - (u)_{Q_r} \right\|_{\underline{L}^2(Q_{r})} .
\end{multline}
\end{lemma}
\begin{proof} 
The proof is essentially the same as that of~\cite[Lemma 3.5]{AKMBook}. We just have to substitute balls for parabolic cylinders and use Proposition~\ref{p.caloricapproximation} in place of its elliptic version. These changes cause no additional complexity in the proof. 
\end{proof}

With Lemma~\ref{l.u:decayestimate.parab} in hand, 
the proof of Theorem~\ref{t.regularity} is now completed in the same way as the one of~\cite[Theorem 3.6]{AKMBook}, by following the argument almost verbatim and making only obvious modifications. We refer to~\cite{AKMBook} for the details.

\appendix

\section{Variational structure of uniformly parabolic equations}
\label{s.app}

The aim of this appendix is to show that the solution of the parabolic equation~\eqref{e.parab.rhs} can be obtained as the minimizer of a uniformly convex functional. We will prove this result in the more general context of uniformly monotone operators, since this causes no modification to the proof. Although our statement differs in detail, it is close to the main result of \cite{ghoussoub-tzou} (see also the monograph \cite{ghoussoub-book}). The proof we give is also relatively close to that of \cite{ghoussoub-tzou}; we hope that the reader will appreciate the short and self-contained presentation in this appendix. The fact that a parabolic equation can be cast as the first variation of a uniformly convex integral functional was first discovered in~\cite{BE1,BE2}.


\smallskip

Let $I :=(0,T) \subset \R$ and $U \subset \Rd$ be a bounded Lipschitz domain. For a given right-hand side $w^*$ and boundary condition (both of which will be made precise below), we study the solvability of the parabolic equation 
\begin{equation}  
\label{e.parab.app}
\partial_t u - \nabla \cdot (\a(\nabla u,\cdot)) = w^* \qquad \text{in } I \times U,
\end{equation}
where the dot ``$\, \cdot\, $'' represents the time-space variable in $I\times U \subset \R^{1+d}$, and $\a \in L^\infty_{\mathrm{loc}}(\Rd \times \R^{1+d};\Rd)$ is Lipschitz and uniformly monotone in its first argument. That is, we assume that there exists a constant $\lambda  < \infty$ such that, for every $p_1, p_2 \in \Rd$ and $z \in \R^{1+d}$,
\begin{equation}
\label{e.unifmonotone}
\Ll\{
\begin{aligned}  
& |\a(p_1,z) - \a(p_2,z)| \le \lambda |p_1 - p_2|, \\
& \Ll( \a(p_1,z) - \a(p_2,z) \Rr) \cdot (p_1 - p_2) \ge \lambda^{-1} \left|p_1 - p_2\right|^2.
\end{aligned}
\Rr.
\end{equation}
As a first step, we introduce a variational representation of the mapping $p \mapsto \a(p,z)$, for each $z \in \R^{1+d}$. This idea is often attributed to Fitzpatrick \cite{Fitz}, although it actually  appeared in the work of Krylov~\cite{Krylov} several years earlier.

\smallskip

By \cite[Theorem~2.9]{AM}, there exists $A \in L^\infty_{\mathrm{loc}}(\Rd\times \Rd\times \R^{1+d})$ satisfying the following properties, for $\Lambda := 2\lambda + 1$ and for each $z \in \R^{1+d}$:
\begin{itemize}
\item the mapping 
\begin{equation}  
\label{e.A.unif.convex}
(p,q) \mapsto A(p,q,z) - \frac 1 {2\Lambda}(|p|^2 + |q|^2) \quad \text{is convex};
\end{equation}
\item the mapping
\begin{equation}  
\label{e.A.C11}
(p,q) \mapsto A(p,q,z) - \frac \Lambda {2}(|p|^2 + |q|^2) \quad \text{is concave};
\end{equation}
\item for every $p,q \in \Rd$, we have
\begin{equation}  
\label{e.A.ineq}
A(p,q,z) \ge p\cdot q,
\end{equation}
and
\begin{equation} 
\label{e.A.eqiff}
A(p,q,z) = p\cdot q 
\iff
q = \a(p,z).
\end{equation}
\end{itemize}
In the particular case when $p \mapsto \a(p,z)$ is linear, we can define the mapping $(p,q) \mapsto A(p,q,z)$ according to \eqref{e.def.F}, see Lemma~\ref{l.lin.Fitz}. Another familiar example is when $\a(\cdot,z)$ is the gradient of a uniformly convex Lagrangian $L(p,z)$, that is, $\a(\cdot,z) = \nabla_p L(\cdot,z)$ where $p \mapsto L(p,z)$ is uniformly convex. In this case, we can take
\begin{equation*} \label{}
A(p,q,z):= L(p,z)+L^*(q,z),
\end{equation*}
where $L^*$ is the Legendre-Fenchel transform of~$L$. We remark that the choice of $A$ is in general not unique. 

\smallskip

We define the function space
\begin{equation*}
Z(I\times U) := \Ll\{ (u,u^*) \ : \  u \in L^2(I;H^1(U)) \ \text{ and } \ (u^*- \partial_t u) \in L^2(I;H^{-1}(U))\Rr\},
\end{equation*}
with norm
\begin{equation*}  
\|(u,u^*)\|_{Z(I\times U)} := \|u\|_{L^2(I;H^1(U))} + \| u^* - \partial_t u\|_{L^2(I;H^{-1}(U))}.
\end{equation*}
The function space $H^1_\pa(I\times U)$ is defined in \eqref{e.def.X}-\eqref{e.def.Xnorm}. We denote by $H^1_{\pa,\sqcup}(I\times U)$ the closure in $H^1_\pa(I\times U)$ of the set of smooth functions with compact support in $(0,T] \times U$. %
For every $(u,u^*) \in Z(I\times U)$, we set
\begin{equation}
\label{e.def.mclJ.2}
\mcl J[u,u^*] := \inf \Ll\{ \int_{I \times U} \Ll(A(\nabla u,\g,\cdot) - \nabla u \cdot \g\Rr) \ : \ -\nabla \cdot \g = u^* - \partial_t u \Rr\}.
\end{equation}
In the infimum above, we understand that $\g \in L^2(I\times U;\Rd)$, and the last condition is interpreted as
\begin{equation}  
\label{e.interp.div}
\forall \phi \in L^2(I;H^1_0(U)), \quad \int_{I\times U} \nabla \phi \cdot \g = \int_{I \times U} \phi \Ll( u^* - \partial_t u \Rr) .
\end{equation}
Note that the set of candidates for $\g$ is not empty; indeed, denoting by $\Delta_U^{-1}$ the solution operator for the Laplacian in $U$ with a null Dirichlet boundary condition, we verify that 
\begin{equation*}  
\g = \nabla \Delta_U^{-1}(u^* - \partial_t u)
\end{equation*}
is a suitable candidate, by the assumption of $u^* - \partial_t u \in L^2(I;H^{-1}(U))$. 

\smallskip

The goal of this appendix is to prove the following proposition.
\begin{proposition}
\label{p.parabolic.min.app}
For each $(w,w^*) \in Z(I\times U)$, the mapping
\begin{equation}  
\label{e.parab.min.app}
\Ll\{
\begin{array}{ccc}
w + H^1_{\pa,\sqcup}(I\times U)  & \to & \R \\
u & \mapsto & \mcl J[u,w^*]
\end{array}
\Rr.
\end{equation}
is uniformly convex. Moreover, its minimum is zero, and the associated minimizer is the unique $u \in w + H^1_{\pa,\sqcup}(I\times U)$ solution of \eqref{e.parab.app}, in the sense that
\begin{equation*}  
\forall \phi \in L^2(I;H^1_0(U)), \quad \int_{I \times U} \nabla \phi \cdot \a( \nabla u,\cdot) = \int_{I \times U} \phi \Ll( w^* - \partial_t u \Rr).
\end{equation*}
\end{proposition}
\begin{remark}
\label{r.Dirichlet}
By the inclusion 
\begin{equation}  
\label{e.fs.inclusions}
H^1_\pa(I\times U) \times L^2(I;H^{-1}(U)) \subset Z(I\times U) ,
\end{equation}
Proposition~\ref{p.parabolic.min.app} ensures in particular the solvability of the parabolic equation \eqref{e.parab.app} for every right-hand side $w^* \in L^2(I;H^{-1}(U))$ and every boundary condition $w \in H^1_\pa(I\times U)$; the solution thus obtained then belongs to $H^1_\pa(I\times U)$.  

\smallskip

More generally, for every $w^*$ of the form
\begin{equation}
\label{e.wstarform}
w^* = \partial_t f + v, \quad f \in L^2(I;H^1_0(U)), \ v \in L^2(I;H^{-1}(U)),
\end{equation}
we have that $(f,w^*) \in Z(I\times U)$ and hence 
Proposition~\ref{p.parabolic.min.app} yields the existence of a unique solution~$u \in f+H^1_{\pa,\sqcup}(I\times U)$ of~\eqref{e.parab.app} which satisfies the estimate
\begin{equation}
\label{e.generalsolutionestimate}
\left\| u - f \right\|_{H^1_\pa(I\times U) } 
\leq
C \left( \left\| f \right\|_{L^2(I;H^1(U))} + \left\| w^* -\partial_t f \right\|_{L^2(I;H^{-1}(U))} \right). 
\end{equation}
In other words, we have identified a mapping
\begin{equation}  
\label{e.sol.map.expansion}
\partial_t f + v \longmapsto f + P(f,v),
\end{equation}
where $f \in L^2(I;H^1_0(U))$, $v \in L^2(I;H^{-1}(U))$, and $P$ is a bounded linear operator from $L^2(I;H^1_0(U)) \times L^2(I;H^{-1}(U))$ to $H^1_{\pa,\sqcup}(I\times U)$. If we moreover restrict our attention, say, to the set of functions $f$ which vanish in a neighborhood of $\{0\}\times U$, then this mapping provides with a notion of solution of \eqref{e.parab.app} with null Dirichlet boundary condition on the parabolic boundary of $I\times U$. This additional regularity assumption on the behavior of $f$ near the initial time can of course be weakened as desired.

\smallskip

Note that every $w^*$ of the form~\eqref{e.wstarform} belongs to $\hat{H}^{-1}_\pa(I\times U)$, but the latter space is strictly larger than the set of such~$w^*$. This may at first glance appear at odds with Lemma~\ref{l.identif.H-1par}, however that lemma required that $u^*$ belong to $L^2(I\times U)$. This hypothesis rules out certain singular distributions which belong to~$\hat{H}^{-1}_\pa(I\times U)$ but cannot be written in the form~\eqref{e.wstarform}. 
\end{remark}

\begin{remark}  
\label{r.no.reflexive}
One may wonder if, in analogy with the elliptic setting, one can identify a reflexive subspace $E$ of the space of distributions such that the standard heat operator $(\partial_t - \Delta)$ maps $E$ to its dual $E^*$ surjectively. This is however not possible, as we now explain briefly. Observe first that by Proposition~\ref{p.parabolic.min.app}, the heat operator is a bijective mapping from $H^1_{\pa,\sqcup}(I\times U)$ to $L^2(I;H^{-1}(U))$, and that $L^2(I;H^{-1}(U))$ is strictly smaller than the dual of $H^1_{\pa,\sqcup}(I\times U)$. Indeed, the dual of $H^1_{\pa,\sqcup}(I\times U)$ contains all elements of the form $\partial_t v$, for $v \in L^2(I;H^{1}_0(U))$.  Hence, the space~$E$ should be strictly between the spaces $H^1_{\pa,\sqcup}(I\times U)$ and $L^2(I;H^1_0(U))$. Using the decomposition of the solution operator in \eqref{e.sol.map.expansion}, one can then verify that such a space $E$ does not exist.
\end{remark}

Before turning to the proof of Proposition~\ref{p.parabolic.min.app}, we first recall the following continuity result for elements of a space intermediate between $H^1_{\pa,\sqcup}(I\times U)$ and $H^1_\pa(I\times U)$ where the null boundary condition is only imposed in the space direction. We refer to \cite[Section III.1.4]{temam} for a proof.
\begin{lemma}
\label{l.continuity}
Let $u \in L^2(I;H^1_0(U))$ be such that $\partial_t u \in L^2(I;H^{-1}(U))$. There exists $\td u \in C\Ll(\bar I; L^2(U)\Rr)$ such that, for almost every $t \in I$, we have $u(t,\cdot) = \td u(t,\cdot)$.
\end{lemma}
From now on, whenever a function $u$ satisfies the conditions of Lemma~\ref{l.continuity}, we identify it with its continuous representative.

\begin{proof}[Proof of Proposition~\ref{p.parabolic.min.app}] We decompose the proof into four steps.

\smallskip

\emph{Step 1.} We show that the mapping in \eqref{e.parab.min.app} is uniformly convex. We will in fact prove the stronger statement that the mapping
\begin{equation}  
\label{e.inner.map}
(u,\g)  \mapsto  \int_{I \times U} \Ll(A(\nabla u,\g,\cdot) - \nabla u \cdot \g\Rr) ,
\end{equation}
defined over all pairs $(u,\g)$ in the set
\begin{equation}  
\label{e.set.inner.map}
\Ll\{(u,\g) \in \Ll(w + H^1_{\pa,\sqcup}(I\times U)\Rr) \times L^2(I\times U;\Rd) \ \ \text{ and } \ -\nabla \cdot \g = w^* - \partial_t u \Rr\},
\end{equation}
is uniformly convex. We first show that the mapping
\begin{equation*}  
(u,\g) \mapsto -\int_{I \times U} \nabla u \cdot \g
\end{equation*}
is convex over the set defined in \eqref{e.set.inner.map}. By \eqref{e.interp.div} (with $u^*$ replaced by $w^*$), we have
\begin{align}  
\notag
-\int_{I \times U} \nabla u \cdot \g & = -\int_{I \times U} \nabla w \cdot \g + \int_{I \times U} (w-u)(w^* - \partial_t u) \\
\label{e.null.lagr}
& = -\int_{I \times U} \nabla w \cdot \g + \int_{I \times U} (w-u) (w^* - \partial_t w) + \frac 1 2 \|(u-w)(T,\cdot)\|_{L^2(U)}^2.
\end{align}
This expression is clearly convex in the pair $(u,\g)$. We now complete this step by showing that the mapping
\begin{equation*}  
(u,\g) \mapsto \int_{I \times U} A(\nabla u,\g,\cdot)
\end{equation*}
is uniformly convex over the set defined in \eqref{e.set.inner.map}. By \eqref{e.A.unif.convex}, for every $(u,\g)$ in the set defined in \eqref{e.set.inner.map} and 
\begin{equation}  
\label{e.tangent.space}
(v,\h) \in H^1_{\pa,\sqcup}(I\times U) \times L^2(I\times U;\Rd) \ \ \text{s.t.} \ \ \nabla \cdot \h = \partial_t v,
\end{equation}
we have
\begin{equation*}  
\frac 1 2 A(\nabla (u+v),\g + \h,\cdot) + \frac 1 2 A(\nabla (u-v),\g - \h,\cdot) - A(\nabla u,\g,\cdot) \ge \frac 1 {2\Lambda} \Ll( |\nabla v|^2 + |\h|^2 \Rr) .
\end{equation*}
Moreover, by \eqref{e.tangent.space}, 
\begin{align*}  
\|\partial_t v\|_{L^2(I;H^{-1}(U))} & = \sup \Ll\{\int_{I\times U} \phi \, \partial_t v \ : \ \phi \in L^2(I;H^1_0(U)), \ \|\nabla \phi\|_{L^2(I\times U)} \le 1 \Rr\}\\
& = \sup \Ll\{ \int_{I \times U} \nabla \phi \cdot \h \ : \ \phi \in L^2(I;H^1_0(U)), \ \|\nabla \phi\|_{L^2(I\times U)} \le 1 \Rr\}\\
& \le \|\h\|_{L^2(I\times U)}.
\end{align*}
We have thus shown that
\begin{multline*}  
\int_{I \times U} \Ll(\frac 1 2 A(\nabla (u+v),\g+\h,\cdot) + \frac 1 2 A(\nabla (u-v),\g-\h,\cdot) - A(\nabla u,\g,\cdot)\Rr)\\
\ge
\frac{1}{4\Lambda} \Ll( \|\nabla v\|_{L^2(I\times U)}^2 + \|\partial_t v\|_{L^2(I;H^{-1}(U))}^2 + \|\h\|_{L^2(I\times U)}^2 \Rr) ,
\end{multline*}
so the proof of uniform convexity is complete.

\smallskip

\emph{Step 2.} By the result of the previous step, there exists a unique pair $(u_0,\g_0)$ in the set defined by \eqref{e.set.inner.map} which minimizes the functional in \eqref{e.inner.map}. In order to complete the proof, it suffices to show that 
\begin{equation}  
\label{e.null.min}
\int_{I \times U} \Ll(A(\nabla u_0,\g_0,\cdot) - \nabla u_0 \cdot \g_0\Rr) = 0.
\end{equation}
Indeed, by \eqref{e.A.ineq}, the identity \eqref{e.null.min} implies that 
\begin{equation*}  
\g_0 = \a(\nabla u_0,\cdot) \quad \text{a.e. in } I\times U,
\end{equation*}
and moreover, by \eqref{e.set.inner.map},
\begin{equation*}  
\nabla \cdot \g_0 = w^* - \partial_t u_0,
\end{equation*}
so that $u_0$ indeed solves
\begin{equation*}  
\partial_t u_0 - \nabla \cdot \Ll( \a(\nabla u_0,\cdot) \Rr) = w^*
\end{equation*}
in the weak sense. Our goal is therefore to show \eqref{e.null.min}. The fact that the left side of \eqref{e.null.min} is non-negative is immediate from \eqref{e.A.ineq}. There remains to show that this quantity is non-positive, that is,
\begin{equation}
\label{e.null.inf.explicit}
\inf_{u \in H^1_{\pa,\sqcup}(I\times U)} \mcl J[w+u,w^*] \le 0.
\end{equation}
In order to do so, we consider the perturbed convex minimization problem defined for every $u^* \in L^2(I;H^{-1}(U))$ by
\begin{equation*}  
G(u^*) := \inf_{u \in H^1_{\pa,\sqcup}(I\times U)} \Ll( \mcl J[w+u,w^*+u^*] + \int_{I \times U} u \, u^* \Rr).
\end{equation*}
Note that \eqref{e.null.inf.explicit} is equivalent to the statement that $G(0) \le 0$. By the computation in \eqref{e.null.lagr}, for every $u^* \in L^2(I;H^{-1}(U))$ and 
\begin{equation}  
\label{e.set.perturbed}
(u,\g) \in H^1_{\pa,\sqcup}(I\times U)  \times L^2(I\times U;\Rd)  \ \ \text{s.t.} \ -\nabla \cdot \g = w^* + u^* - \partial_t (w+u),
\end{equation}
we have
\begin{multline}  
\label{e.expand.perturbed}
 \int_{I \times U} \Ll(A(\nabla (w+u),\g,\cdot) - \nabla (w+u) \cdot \g\Rr) + \int_{I \times U} u \, u^* \\
= \int_{I \times U} \Ll(A(\nabla(w+ u),\g,\cdot) - \nabla w \cdot \g -u(w^*-\partial_t w)\Rr) + \frac 1 2  \|u(T,\cdot)\|_{L^2(U)}^2,
\end{multline}
and hence the function $G$ is convex over $L^2(I;H^{-1}(U))$. Moreover, one can check that it is also locally bounded above, which implies that $G$ is lower semi-continuous, by convexity (see e.g.\ \cite[Lemma~I.2.1 and Corollary~I.2.2]{ET}). Denoting by $G^*$ the convex dual of $G$, defined for every $v \in L^2(I;H^1_0(U))$ by
\begin{equation*}  
G^*(v) := \sup_{u^* \in L^2(I;H^{-1}(U))} \Ll(  - G(u^*) + \int_{I \times U} v \, u^* \Rr) ,
\end{equation*}
and by $G^{**}$ its bidual, we deduce that $G = G^{**}$ (see \cite[Proposition~I.4.1]{ET}), and in particular,
\begin{equation*}  
G(0) = G^{**}(0) = \sup_{v \in L^2(I;H^{1}_0(U))} \Ll( -G^*(v) \Rr) .
\end{equation*}
The statement \eqref{e.null.inf.explicit} is therefore equivalent to
\begin{equation}
\label{e.nonneg.G*}
\forall v \in L^2(I;H^1_0(U)), \ \ G^*(v) \ge 0.
\end{equation}
Ths proof of this fact occupies the next two steps.

\smallskip

\emph{Step 3.} For each $v \in L^2(I;H^1_0(U))$, we have $G^*(v) \in \R \cup \{+\infty\}$. In this step, we show that
\begin{equation}  
\label{e.nonneg.G*.1}
G^*(v) < +\infty \quad \implies \quad \partial_t v \in L^2(I;H^{-1}(U)).
\end{equation}
We note that
\begin{multline}  
\label{e.expr.G*}
G^*(v) = \sup \bigg\{ \int_{I\times U} \Ll( (v-u) \, u^* - A(\nabla (w+u),\g,\cdot) + \nabla (w+u) \cdot \g \Rr)  \ : \\ u^* \in L^2(I;H^{-1}(U)), \ (u ,\g) \ \text{satisfy} \ \eqref{e.set.perturbed} \bigg\} .
\end{multline}
Specifying to $u^* = \partial_t u$ and to a fixed $\g \in L^2(I\times U;\Rd)$ satisfying $-\nabla \cdot \g = w^* - \partial_t w$ (which can be constructed as the gradient of the solution of a Dirichlet problem) yields the lower bound
\begin{multline*}  
G^*(v) \ge \sup \bigg\{ \int_{I\times U} \Ll(v \, \partial_t u - A(\nabla (w+ u),\g,\cdot) + \nabla (w+u) \cdot \g \Rr) - \frac 1 2 \|u(T,\cdot)\|_{L^2(U)}^2 \ : 
\\
 u \in H^1_{\pa,\sqcup}(I\times U)\bigg\} .
\end{multline*}
The assumption of $G^*(v) < \infty$ thus implies that
\begin{equation*}  
\sup \Ll\{ \int_{I\times U} v \, \partial_t u \ : \ u \in H^1_{\pa,\sqcup}(I\times U), \ \|\nabla u\|_{L^2(I\times U)} \le 1, \ \|u(T,\cdot)\|_{L^2(U)} \le 1 \Rr\}  < \infty.
\end{equation*}
Denoting the supremum above by $C < \infty$, we infer that for every smooth test function $u$ with compact support in $I \times U$, 
\begin{equation*}  
\Ll|\int_{I\times U} u \, \partial_t v \Rr|\le C \|\nabla u\|_{L^2(I\times U)}.
\end{equation*}
By density, we deduce that $\partial_t v$ can be identified with an element of the dual of $L^2(I;H^1_0(U))$. Since this dual space is $L^2(I;H^{-1}(U))$, the proof of \eqref{e.nonneg.G*.1} is complete.

\smallskip

\emph{Step 4.} In this step, we show that 
\begin{equation}
\label{e.nonneg.G^*.2}
v \in L^2(I;H^1_0(U))  \ \text{ and } \ \partial_t v \in L^2(I;H^{-1}(U)) \quad \implies \quad G^*(v) \ge 0.
\end{equation}
Together with \eqref{e.nonneg.G*.1}, this would complete the proof of \eqref{e.nonneg.G*} and therefore of the proposition. 

\smallskip

The fact that $G^*(v) \ge 0$ would follow immediately from \eqref{e.expr.G*} if we could choose $u = v$ and then ensure the equality of the last two terms under the integral. The difficulty we face is that the function $u$ is allowed to range in $H^1_{\pa,\sqcup}(I\times U)$, while the function $v$ does not belong to this space in general, due to the boundary condition at the initial time. We therefore wish to argue that this constraint on $u$ can be relaxed. 

\smallskip

Replacing $u^*$ by $u^* + \partial_t u$ in the supremum in \eqref{e.expr.G*}, we can rewrite $G^*(v)$ as
\begin{equation}  
\label{e.expr.G*.2}
G^*(v) = \sup \Ll\{ \int_{I\times U} \Ll( (v-u) \,( u^* + \partial_t u) - A(\nabla (w+u),\g,\cdot) + \nabla (w+u) \cdot \g \Rr)  \Rr\},
\end{equation}
where the supremum is taken over every $u^* \in L^2(I;H^{-1}(U))$, $u \in H^1_{\pa,\sqcup}(I\times U)$ and $\g \in L^2(I\times U;\Rd)$ satisfying
\begin{equation}  
\label{e.set.perturbed.2}
-\nabla \cdot \g = w^* + u^* - \partial_t w.
\end{equation}
Integrating by parts, we can rewrite the term involving $\partial_t u$ on the right side of \eqref{e.expr.G*.2} as 
\begin{equation*}  
\int_{I\times U} (v-u) \,\partial_t u = -\int_{I\times U} u \, \partial_t v + \int_{U} u(T,\cdot)  v(T,\cdot) - \frac 1 2 \|u(T,\cdot)\|_{L^2(U)}^2.
\end{equation*}
The functional under the supremum in \eqref{e.expr.G*.2} can thus be decomposed into the sum of
\begin{equation}  
\label{e.G.inside.bit}
I_1(u,u^*,\g) := \int_{I\times U} \Ll( (v-u)u^* - u \partial_t v  - A(\nabla (w+u),\g,\cdot) + \nabla (w+u) \cdot \g \Rr) 
\end{equation}
and
\begin{equation}
\label{e.G.boundary.bit}
I_2(u(T,\cdot)) := \int_{U} u(T,\cdot)  v(T,\cdot) - \frac 1 2 \|u(T,\cdot)\|_{L^2(U)}^2.
\end{equation}
Moreover, for each given $u^* \in L^2(I;H^{-1}(U))$ and $\g \in L^2(I\times U;\Rd)$, the mapping $u \mapsto I_1(u,u^*,\g)$ is continuous for the topology of $L^2(I;H^1(U))$. For any given $b \in H^1_0(U)$, and $\td u \in L^2(I;H^1_0(U))$, one can find elements of the space
\begin{equation*}  
\{u \in H^1_{\pa,\sqcup}(I\times U) \ : \ u(T,\cdot) = b\}
\end{equation*}
which approximate $\td u$ with arbitrary precision, for the topology of $L^2(I;H^1(U))$. Hence, for each given $u^* \in L^2(I;H^{-1}(U))$ and $\g \in L^2(I\times U;\Rd)$, we have
\begin{multline*}  
\sup \Ll\{I_1(u,u^*,\g) + I_2(u(T,\cdot)) \ : \ u \in H^1_{\pa,\sqcup}(I\times U)\Rr\} \\
\ge \sup \Ll\{ I_1(u,u^*,\g) + I_2(b) \ : \ u \in L^2(I;H^1_0(U)) \ \text{ and } b \in H^1_0(U)\Rr\}.
\end{multline*}
Moreover, the mapping $b \mapsto I_2(b)$ is continuous for the topology of $L^2(U)$, and thus we have in fact
\begin{multline*}  
\sup \Ll\{I_1(u,u^*,\g) + I_2(u(T,\cdot)) \ : \ u \in H^1_{\pa,\sqcup}(I\times U)\Rr\} \\
= \sup \Ll\{ I_1(u,u^*,\g) + I_2(b) \ : \ u \in L^2(I;H^1_0(U)) \ \text{ and } b \in L^2(U)\Rr\}.
\end{multline*}
Selecting $u = v$ and $b = v(T,\cdot)$, we have thus shown that
\begin{equation*}  
G^*(v) \ge \sup \Ll\{ \frac 1 2 \|v(T,\cdot)\|_{L^2(U)}^2 + \int_{I\times U} \Ll( -v \, \partial_t v - A(\nabla (w+v),\g,\cdot) + \nabla (w+v) \cdot \g \Rr)   \Rr\} ,
\end{equation*}
where the supremum is taken over every $u^* \in L^2(I;H^{-1}(U))$ and $\g \in L^2(I\times U;\Rd)$ satisfying \eqref{e.set.perturbed.2}. Note that
\begin{equation*}  
\frac 1 2 \|v(T,\cdot)\|_{L^2(U)}^2 - \int_{I\times U}  v \, \partial_t v = \frac 1 2 \|v(0,\cdot)\|_{L^2(U)}^2 \ge 0.
\end{equation*}
Selecting $u^*$ such that 
\begin{equation*}  
-\nabla \cdot \Ll( \a(\nabla (w+v),\cdot) \Rr) = w^* + u^* - \partial_t w,
\end{equation*}
and then 
\begin{equation*}  
\g = \a(\nabla (w+v),\cdot)
\end{equation*}
ensures that the constraint \eqref{e.set.perturbed.2} is satisfied, and by \eqref{e.A.eqiff}, that
\begin{equation*}  
\int_{I\times U} \Ll( A(\nabla (w+v),\g,\cdot) - \nabla (w+v) \cdot \g  \Rr) =0.
\end{equation*}
The proof of \eqref{e.nonneg.G^*.2} is therefore complete.
\end{proof}

\section{Meyers-type estimates}
\label{s.meyers}

In this appendix, we present local and global versions of the Meyers improvement of integrability estimate for gradients of solutions of linear, uniformly parabolic equations with measurable coefficients. 

\smallskip

The interior Meyers estimate in the parabolic case was first proved in~\cite{GS}. We follow their argument to obtain Proposition~\ref{p.interiormeyers}, below, which is included for completeness and since the same ideas are needed to prove the global version in Proposition~\ref{p.globalmeyers}. The statement of the latter will certainly not come as a surprise to experts, but we do not believe it has appeared before. Global versions of the Meyers estimate in the parabolic setting have been previously considered in~\cite{Parv}, but the statement of  Proposition~\ref{p.globalmeyers} is stronger than the results of~\cite{Parv} since we do not require any additional regularity of the boundary condition in time---a modest technical improvement, but it gives a more natural statement and one which is useful for the application in this paper. 

\smallskip

In what follows, we use the same notation for parabolic cylinders as in Section~\ref{s.regularity}, see~\eqref{e.alternativQr}. That is, for $(t,x)\in \R\times\Rd$, we denote
\begin{equation*} \label{}
\tilde I_r:= (-r^2,0], \ \quad \ 
Q_r(t,x):= (t,x) + \tilde I_r \times B_r, \quad \mbox{and} \quad Q_r := Q_r(0,0).
\end{equation*}
We fix a coefficient field $\a=\a(t,x)$ satisfying~\eqref{e.unifellip-1} for every $(t,x) \in \R \times \Rd$, and consider the linear parabolic equation
\begin{equation} 
\label{e.appendixBpde}
\partial_t u -\nabla \cdot \left( \a(t,x) \nabla u \right) = u^*. 
\end{equation}
We remark that the argument we present only makes mild use of linearity and can be adapted to give similar estimates for solutions of nonlinear parabolic equations like the ones considered in Appendix~\ref{s.app}. 

\smallskip

We first present the interior Meyers estimate. Recall that the space~$W^{1,p}_\pa$ is defined in~\eqref{e.def.W1p.par} and~\eqref{e.def.W1p.par.norm}.

\begin{proposition}[{Interior Meyers estimate~\cite[Theorem 2.1]{GS}}]
\label{p.interiormeyers}
Fix $r>0$, $p\geq 2$ and suppose that $u\in H^1_\pa(Q_{2r})$ and $u^*\in L^p\left(I_{2r};W^{-1,p}(B_{2r})\right)$ satisfy equation~\eqref{e.appendixBpde} in $Q_{2r}$. There exist an exponent $\delta(d,\Lambda)>0$ and a constant $C(d,\Lambda)<\infty$ such that $u\in W^{1,p\wedge (2+\delta)}_\pa(Q_r)$ and we have the estimate
\begin{equation} 
\label{e.interiorMeyers}
\left\| \nabla u \right\|_{\underline{L}^{p\wedge (2+\delta)}\left(Q_r\right)}
\leq 
C \left(
\left\| \nabla u \right\|_{\underline{L}^{2}\left(Q_{2r} \right)}
+ 
\left\| u^* \right\|_{\underline{L}^{p\wedge (2+\delta)}\left(I_{2r};\underline{W}^{-1,p\wedge (2+\delta)}(B_{2r})\right)}
\right).
\end{equation}
\end{proposition}

We next give a global statement of the Meyers estimate with respect to a Cauchy-Dirichlet initial-boundary condition. 

\begin{proposition}[Global Meyers estimate]
\label{p.globalmeyers}
Fix $p\geq 2$. Let $U\in\Rd$ be a bounded Lipschitz domain, $I \subseteq \R$ a bounded interval and set $V:=I\times U$. Fix $f\in W^{1,p}_\pa(V)$, $u^* \in L^p\left(I;W^{-1,p}(V) \right)$ and suppose that
\begin{equation*} \label{}
u\in f + H^1_{\pa,\sqcup}(V)
\end{equation*}
is the unique solution of the Cauchy-Dirichlet problem
\begin{equation*} \label{}
\left\{
\begin{aligned}
& \partial_t u - \nabla \cdot \left( \a \nabla u  \right) = u^* & \mbox{in} & \ V, 
\\ 
& u = f & \mbox{on} & \ \partial_\sqcup V.
\end{aligned}\right.
\end{equation*}
There exist $\delta(V,d,\Lambda)>0$ and a constant $C(V,d,\Lambda)<\infty$ such that $
u\in W^{1,p\wedge(2+\delta)}_\pa(V)$ 
and we have the estimate
\begin{equation} \label{}
\left\| u \right\|_{W^{1,p\wedge(2+\delta)}_\pa(V)}
\leq
C \left( \left\| f \right\|_{W^{1,p\wedge(2+\delta)}_\pa(V)}
+ 
\left\| u^* \right\|_{L^{p\wedge(2+\delta)}\left(I;W^{-1,p\wedge(2+\delta)}(V) \right)}
\right).
\end{equation}
\end{proposition}

The Meyers estimates are consequences of the Caccioppoli inequality, the most basic regularity estimate for divergence-form equations.

\begin{lemma}[parabolic Caccioppoli inequality]
\label{l.cacciopp.this}
Suppose that $u\in H^1_\pa(Q_{2r})$ and $u^* \in L^2(I_{2r};H^{-1}(B_{2r}))$ satisfy
\begin{equation*} \label{}
\partial_t u - \nabla \cdot \left(\a \nabla u \right) = u^* \quad \mbox{in} \ Q_{2r}. 
\end{equation*}
Then there exists $C(d,\Lambda)<\infty$ such that
\begin{equation} 
\label{e.cacciopp.this}
 \left\| \nabla u \right\|_{L^2(Q_r)}  
\leq Cr^{-1}\left\|  u \right\|_{L^2(Q_{2r})}  +  C \left\| u^* \right\|_{L^2 \left( I_{2r};H^{-1}(B_{2r}) \right) }
\end{equation}
and
\begin{equation} 
\label{e.Cacciopp.thetime}
\sup_{s\in I_{r}} \left\|  u(s,\cdot) \right\|_{L^2(B_r)}  
\leq C\left\|  \nabla u \right\|_{L^2(Q_{2r})} + C \left\| u^* \right\|_{L^2\left(I_{2r};H^{-1}(B_{2r})\right)}.
\end{equation}
\end{lemma}
\begin{proof}
We take $\eta_r \in C^\infty_c(Q_{2r})$ to be a test function satisfying 
\begin{equation*} \label{}
0\leq \eta \leq 1, \quad
\eta \equiv 1 \ \mbox{on} \ Q_r,  \quad
\left| \partial_t \eta \right| + \left| \nabla \eta \right|^2 \leq Cr^{-2}. 
\end{equation*}
We test the weak formulation 
\begin{equation*} \label{}
\forall \phi \in L^2(I_{2r};H^{1}_0(B_{2r})), \quad \int_{Q_{2r}} \phi \left( u^*-\partial_tu \right) 
= \int_{Q_{2r}}\nabla \phi\cdot \a\nabla u 
\end{equation*}
with the function $\phi := \eta^2_r u \in L^2(I_{2r};H^{1}_0(B_{2r}))$.
We estimate the right side from below by 
\begin{align*} \label{}
\int_{Q_{2r}}\nabla \phi\cdot \a \nabla u 
& 
\geq \frac{1}{\Lambda} \int_{Q_{2r}} \eta_r^2 \left| \nabla u \right|^2 -  C \int_{Q_{2r}} \eta_r \left| \nabla \eta_r \right| |u| \left| \nabla u \right|
\\ &
\geq \frac{1}{2\Lambda} \int_{Q_{2r}} \eta_r^2 \left| \nabla u \right|^2 - C\int_{Q_{2r}} \left| \nabla \eta_r \right|^2 |u|^2 
\\ & 
\geq  \frac{1}{2\Lambda} \int_{Q_{2r}} \eta_r^2 \left| \nabla u \right|^2 - Cr^{-2} \int_{Q_{2r}}  |u|^2 
\end{align*}
and the left side from above by 
\begin{align*}
\int_{Q_{2r}} \eta^2_r u \left( u^*-\partial_tu \right) 
&
\leq - \int_{Q_{2r}} \partial_t \left( \frac12 \eta^2_r u^2 \right) + \int_{Q_{2r}} \eta_r \left| \partial_t\eta_r \right| u^2 
\\ & \quad
+ \int_{-4r^2}^0 \left\| (\eta^2_r u)(t,\cdot) \right\|_{H^1(B_{2r})} \left\| u^*(t,\cdot) \right\|_{H^{-1}(B_{2r})} \,dt 
\\ & 
\leq - \frac12 \int_{B_{2r}} \eta_r^2(0,x) u^2(0,x)\,dx + Cr^{-2} \int_{Q_{2r}} u^2 
\\ & \quad
+C  \left\| \eta^2_r u \right\|_{L^2\left(I_{2r};H^1(B_{2r})\right)} \left\| u^* \right\|_{L^2\left(I_{2r};H^{-1}(B_{2r})\right)}.
\end{align*}
Using that 
\begin{equation*} \label{}
\left\| \eta^2_r u \right\|_{L^2\left(I_{2r};H^1(B_{2r})\right)}
\leq Cr^{-1} \left\| u \right\|_{L^2(I_{2r}\times B_{2r})} + C \left\| \eta_r \nabla u \right\|_{L^2(I_{2r}\times B_{2r})},
\end{equation*}
we get 
\begin{multline*} \label{}
C \left\| \eta^2_r u \right\|_{L^2\left(I_{2r};H^1(B_{2r})\right)} \left\| u^* \right\|_{L^2\left(I_{2r};H^{-1}(B_{2r})\right)}
 \\
 \leq r^{-2} \left\| u \right\|_{L^2(I_{2r}\times B_{2r})}^2 + \frac1{4\Lambda} \left\| \eta_r \nabla u \right\|_{L^2(I_{2r}\times B_{2r})}^2
 + C  \left\| u^* \right\|_{L^2\left(I_{2r};H^{-1}(B_{2r})\right)}^2.
\end{multline*}
Combining the above, we get that 
\begin{equation*} \label{}
\frac12 \int_{B_{2r}} \eta_r^2(0,x) u^2(0,x)\,dx +  \frac1{4\Lambda} \int_{Q_{2r}} \eta_r^2 \left| \nabla u \right|^2 
\leq
Cr^{-2} \int_{Q_{2r}}  |u|^2 +  C \left\| u^* \right\|_{L^2\left(I_{2r};H^{-1}(B_{2r})\right)}^2.
\end{equation*}
This yields~\eqref{e.cacciopp.this}.

\smallskip

By repeating the above computation, using instead the test function
$\phi := \eta^2_r u \indc_{\{ t < s\}}$ for fixed $s\in I_{2r}$, and estimating the right side of the weak formulation from below differently, namely
\begin{align*} \label{}
\int_{Q_{2r}}\nabla \phi\cdot \a \nabla u
& 
\geq -C \left\| \eta_r \nabla u \right\|_{L^2(Q_{2r})}^2 -C \left\| \nabla \eta_r \nabla u \right\|_{L^2(Q_{2r})}  \left\| u  \eta_r \right\|_{L^2(Q_{2r})} 
\\ & 
\geq -C  \left\| \nabla u \right\|_{L^2(Q_{2r})}^2 - \frac1{16}r^{-2}  \int_{Q_{2r}} \eta_r^2 u^2
\\ & 
\geq -C  \left\|  \nabla u \right\|_{L^2(Q_{2r})}^2 - \frac14 \sup_{t\in I_{2r}} \int_{B_{2r}} \eta_r^2(t,x) u^2(t,x)\,dx,
\end{align*}
we get the bound 
\begin{multline*} \label{}
\frac12 \int_{B_{2r}} \eta_r^2(s,x) u^2(s,x)\,dx 
\\
\leq C\left\|  \nabla u \right\|_{L^2(Q_{2r})}^2 +  \frac14 \sup_{t\in I_{2r}}  \int_{B_{2r}} \eta_r^2(t,x) u^2(t,x)\,dx + C \left\| u^* \right\|_{L^2\left(I_{2r};H^{-1}(B_{2r})\right)}^2.
\end{multline*}
Taking the supremum over $s\in I_{2r}$ and rearranging, we get~\eqref{e.Cacciopp.thetime}. 
\end{proof}

In the following statement, what is important is that $q<2$. It is convenient to use the Sobolev exponent~$q:=2_*$, although the choice $q=1$ in $d=2$ causes technical problems so in that case we just take $q \in \left(\frac54,\tfrac74\right)$.

\begin{lemma}[Reverse H\"older inequality]
\label{l.reverse.Holder}
Suppose that $u\in H^1_\pa(Q_{4r})$ and $u^* \in L^2\left(I_{4r};H^{-1}(B_{4r})\right)$ satisfy
\begin{equation*} \label{}
\partial_t u - \nabla \cdot \left( \a(x) \nabla u \right) = u^* \quad \mbox{in} \ Q_{4r}. 
\end{equation*}
Denote $q:=2_* = 2d/(2+d)$ if~$d>2$ or let~$q$ be any element of $\left(\frac54,\tfrac74\right)$ if~$d=2$. 
Then there exists $C(d,\Lambda)<\infty$ such that, for every $\alpha>0$,
\begin{equation} 
\label{e.reverse.Holder}
\left\| \nabla u \right\|_{\underline{L}^2(Q_r)}^2
\leq \frac{C}{\alpha} \left\| \nabla u \right\|_{\underline{L}^q(Q_{4r})}^2 + \alpha \left\|  \nabla u \right\|_{\underline{L}^2(Q_{4r})}^{2}
+ C \left\| u^* \right\|_{\underline{L}^2\left(I_{4r};\underline{H}^{-1}(B_{4r})\right)}^2,
\end{equation}
\end{lemma}
\begin{proof}
By subtracting a constant, we may suppose that $\left( u \right)_{Q_{2r}}=0$. Let $\xi\in C^\infty_c(B_r)$ with $\int_{B_r} \xi = 1$ and $\left| \nabla \xi \right| \leq Cr^{-1}$. Denote
\begin{equation*} \label{}
v(t,x) := u(t,x) - w(t), \quad w(t):= \int_{B_r} \xi(y)u(t,y)\,dy.
\end{equation*}
Then $v$ satisfies
\begin{equation*} \label{}
\partial_t v - \nabla \cdot\left( \a\nabla v \right) = u^* - \partial_t w.
\end{equation*}
Applying~\eqref{e.Cacciopp.thetime} to $v$, we find that
\begin{align*}
\int_{Q_{2r}} \left| v  \right|^2 
& 
\leq  \left(\sup_{s\in I_{2r}} \int_{B_{2r}} \left| v(s,x)  \right|^2 \,dx \right)^{\frac12} \int_{I_{2r}} \left( \int_{B_{2r}} \left| v(t,x)  \right|^2 \,dx\right)^{\frac12} \,dt
\\ & 
\leq C \left( \left\|  \nabla u \right\|_{L^2(Q_{4r})} + \left\| u^* - \partial_tw\right\|_{L^2\left(I_{4r};H^{-1}(B_{4r})\right)} \right) 
\int_{I_{2r}} \left( \int_{B_{2r}} \left| v(t,x)  \right|^2 \,dx\right)^{\frac12} \,dt.
\end{align*}
Denote by $q'$ the H\"older conjugate exponent to $q$ and notice that $q' = 2^*$ in $d>2$ and~$q'<\infty$ in $d=2$. Using the H\"older and Sobolev inequalities, we find that  
\begin{align*}
\lefteqn{ 
\int_{I_{2r}} \left( \int_{B_{2r}} \left| v(t,x)  \right|^2 \,dx\right)^{\frac12} \,dt
} \qquad & 
\\ & 
\leq \int_{I_{2r}} \left( \int_{B_{2r}} \left| v(t,x)\right|^q \,dx \right)^{\frac{1}{2q}} \left(\int_{B_{2r} } \left| v(t,x)\right|^{q'}\,dx \right)^{\frac{1}{2q'}}\,dt
\\ & 
\leq C r^{1 + d\left( \frac14-\frac1{2q}\right) } 
\int_{I_{2r}} 
\left( \int_{B_{2r}} \left| \nabla v(t,x)\right|^q \,dx \right)^{\frac{1}{2q}}  \left(\int_{B_{2r} } \left| \nabla v(t,x)\right|^{2}\,dx \right)^{\frac{1}{4}}\,dt
\\ &
\leq C r^{1 + d\left( \frac14-\frac1{2q}\right) }  
\left\| \nabla v \right\|_{L^q(Q_{2r})}^{\frac12} \left( \int_{I_{2r}} \left( \int_{B_{2r}} \left| \nabla v(t,x) \right|^2 \,dx \right)^{\frac{(2q)'}{4}} \,dt  \right)^{\frac{1}{(2q)'}}.
\end{align*}
As $(2q)'/4 \leq \frac12 < 1$, we can use H\"older's inequality in time and then~\eqref{e.cacciopp.this} and Lemma~\ref{l.Poincare.psi} to get 
\begin{align*}
\lefteqn{ 
\left( \int_{I_{2r}} \left( \int_{B_{2r}} \left| \nabla v(t,x) \right|^2 \,dx \right)^{\frac{(2q)'}{4}} \,dt  \right)^{\frac{2}{(2q)'}}
} \qquad & 
\\ &
\leq 
Cr^{\frac{2}{(2q)'} -\frac12} \left(    \int_{I_{2r}}  \int_{B_{2r}}  \left| \nabla v (t,x) \right|^2  \,dx \,dt  \right)^{\frac12}
\\ &
\leq Cr^{\frac{2}{(2q)'} -\frac12} \left( \left\|  \nabla u \right\|_{L^2(Q_{4r})} + \left\| u^* - \partial_tw\right\|_{L^2\left(I_{4r};H^{-1}(B_{4r})\right)} \right).
\end{align*}
Let $\kappa:=d\left(\frac14-\frac1{2q}\right) + \frac{1}{(2q)'} +\frac34$. Combining the above, we get
\begin{equation*} \label{}
\left\| v \right\|_{L^2(Q_{2r})}^2 
\leq Cr^{\kappa} 
 \left\| \nabla v \right\|_{L^q(Q_{2r})}^{\frac12} \left( \left\|  \nabla u \right\|_{L^2(Q_{4r})} + \left\| u^* - \partial_tw\right\|_{L^2\left(I_{4r};H^{-1}(B_{4r})\right)} \right)^{\frac32}.
\end{equation*}
Combining~\eqref{e.cacciopp.this} and the previous inequality, we obtain
\begin{multline*} \label{}
\left\| \nabla v \right\|_{L^2(Q_r)}^2
\leq Cr^{\kappa-2} 
  \left\| \nabla v \right\|_{L^q(Q_{2r})}^{\frac12} 
 \left( \left\|  \nabla v \right\|_{L^2(Q_{4r})} + \left\| u^* - \partial_tw\right\|_{L^2\left(I_{4r};H^{-1}(B_{4r})\right)} \right)^{\frac32}
\\ 
+ C \left\| u^* - \partial_tw\right\|_{L^2\left(I_{4r};H^{-1}(B_{4r})\right)}^2.
\end{multline*}
Normalizing the norms, we find that this is the same as 
\begin{multline*} 
\label{}
\left\| \nabla v \right\|_{\underline{L}^2(Q_r)}^2
\leq
  \left\| \nabla v \right\|_{\underline{L}^q(Q_{2r})}^{\frac12} 
 \left( \left\|  \nabla v \right\|_{\underline{L}^2(Q_{4r})} + \left\| u^* - \partial_tw\right\|_{\underline{L}^2\left(I_{4r};\underline{H}^{-1}(B_{4r})\right)} \right)^{\frac32}
\\
+
C\left\| u^* - \partial_tw\right\|_{\underline{L}^2\left(I_{4r};\underline{H}^{-1}(B_{4r})\right)}^2.
\end{multline*}
Applying Young's inequality, we obtain, for every $\alpha >0$, 
\begin{equation} 
\label{e.paused}
\left\| \nabla v \right\|_{\underline{L}^2(Q_r)}^2
\leq \frac{C}{\alpha} \left\| \nabla v \right\|_{\underline{L}^q(Q_{4r})}^2 + \alpha \left\|  \nabla v \right\|_{\underline{L}^2(Q_{4r})}^{2}
+ C \left\| u^* - \partial_tw\right\|_{\underline{L}^2\left(I_{4r};\underline{H}^{-1}(B_{4r})\right)}^2.
\end{equation}
It is not difficult to show, by using the equation and the definition of $w$, that 
\begin{equation*} \label{}
\left\|  \partial_tw\right\|_{\underline{L}^2\left(I_{4r};\underline{H}^{-1}(B_{4r})\right)}
\leq
C \left( 
\left\| \nabla u \right\|_{\underline{L}^1(Q_{4r})}
+
\left\| u^* \right\|_{\underline{L}^2\left(I_{4r};\underline{H}^{-1}(B_{4r})\right)} \right).
\end{equation*}
Combining the previous two displays yields 
\begin{equation*} \label{}
\left\| \nabla v \right\|_{\underline{L}^2(Q_r)}^2
\leq \frac{C}{\alpha} \left\| \nabla v \right\|_{\underline{L}^q(Q_{4r})}^2 + \alpha \left\|  \nabla v \right\|_{\underline{L}^2(Q_{4r})}^{2}
+ C \left\| u^* \right\|_{\underline{L}^2\left(I_{4r};\underline{H}^{-1}(B_{4r})\right)}^2.
\end{equation*}
Since $\nabla v = \nabla u$, this completes the argument. 
\end{proof}

To complete the proof of the interior Meyers estimate, we need the following version of Gehring's lemma for parabolic cylinders which states that a reverse H\"older inequality implies an improvement of integrability. This result is standard and so we do not give the proof here. See for instance~\cite[Proposition 5.1]{GM}, where the statement is given in cubes rather than parabolic cylinders (which makes no difference in its proof).

\begin{lemma}[Gehring-type lemma]
\label{l.parab.Gehring}
Assume that $R>0$, $q>1$, $F \in L^1(Q_{4R})$, $G\in L^{q}(Q_{4R})$, $m\in (0,1)$ and $A\in [1,\infty)$. Suppose that, for every $(t,x) \in Q_{R}$ and $r\in \left(0,\tfrac 12R \right]$,
\begin{equation*} \label{}
\left\| F \right\|_{L^1(Q_r(t,x))} \leq A \left(  \left\| F^m \right\|_{L^1(Q_{4r}(t,x))}^{\frac1m} + \left\| G \right\|_{L^1(Q_{4r}(t,x))}   \right) + \ep \left\| F \right\|_{L^1(Q_{4r} (t,x))}.
\end{equation*}
Then there exists $\ep_0(d,m)\in \left(0,\tfrac 12 \right]$ such that $\ep \leq \ep_0$ implies the existence of an exponent~$\delta(\ep,A,m,q,d)\in\left(0,\tfrac12\right]$ and~$C(\ep,A,m,d) <\infty$ such that $F\in L^{1+\delta}(Q_R)$ and 
\begin{equation*} \label{}
\left\| F \right\|_{L^{1+\delta}(Q_R)} \leq C \left( \left\| F \right\|_{L^{1}(Q_{4R}) }  + \left\| G \right\|_{L^{1+\delta} (Q_{4R}) } \right). 
\end{equation*}
\end{lemma}

The statement of Proposition~\ref{p.interiormeyers} can now be obtained as a consequence of Lemmas~\ref{l.reverse.Holder} and~\ref{l.parab.Gehring} and a routine covering argument. Indeed, any element of $W^{-1,p}(U)$ can be represented as the divergence of an element of $L^p(U;\Rd)$ by the Riesz representation theorem (see~\cite[Theorem 3.9]{AF}).
This allows us to obtain the estimate~\eqref{e.interiorMeyers}. The statement that $\partial_t u$ belongs to $L^{p\wedge (2+\delta)}\left(I_r;W^{-1,p\wedge (2+\delta)}(B_r)\right)$, with an appropriate estimate, follows from~\eqref{e.interiorMeyers} and the equation~\eqref{e.appendixBpde}. 

\smallskip

We next give a sketch of the proof of Proposition~\ref{p.globalmeyers}, which requires us to first revisit the proof of the Caccioppoli inequality to obtain a global version.

\begin{lemma}[global Caccioppoli inequality]
\label{l.cacciopp.this.global}
Let $U\subseteq\Rd$ be a bounded Lipschitz domain and denote $V := I_2 \times U$. 
Suppose that $f\in H^1_\pa(V)$, $u\in f+ H^1_{\pa,\sqcup}(V)$ and $u^* \in L^2(I_2;H^{-1}(U))$ satisfy
\begin{equation*} \label{}
\partial_t u - \nabla \cdot \left( \a(x) \nabla u  \right) = u^* \quad \mbox{in} \ V.
\end{equation*}
Then there exists $C(V,d,\Lambda)<\infty$ such that, for every $r\in (0,1)$ and $(t,x) \in I_1\times U$,
\begin{multline} 
\label{e.cacciopp.this.global}
 \left\| \nabla (u-f) \right\|_{L^2\left(Q_r(t,x) \cap V\right)}  
\leq 
Cr^{-1}\left\|  u - f \right\|_{L^2\left(Q_{2r}(t,x) \cap V\right)} 
\\
+ C\left\| \nabla f \right\|_{L^2\left(Q_{2r}(t,x) \cap V\right)}  
+  C \left\| u^* \right\|_{L^2 \left( (t+I_{2r})\cap I_2;H^{-1}(B_{2r}(x)\cap U) \right) }
\end{multline}
and
\begin{multline} 
\label{e.Cacciopp.thetime.global}
\sup_{s\in (t+ I_{2r})\cap I_2} \left\|  (u-f)(s,\cdot) \right\|_{L^2(B_r(x) \cap U)}  
\leq 
C\left\|  \nabla (u-f) \right\|_{L^2\left(Q_{2r}(t,x) \cap V\right)} 
\\
+ C\left\| \nabla f \right\|_{L^2\left(Q_{2r}(t,x) \cap V\right)}  
+ C \left\| u^* \right\|_{L^2\left((t+I_{2r})\cap I_2;H^{-1}(B_{2r}(x)\cap U)\right)}.
\end{multline}
\end{lemma}
\begin{proof}
By replacing~$u$ by $\tilde{u} : = u-f$ and $u^*$ by $\tilde{u}^*:= u^* - \left( \partial_t   - \nabla \cdot \a \nabla \right)f$, we may assume without loss of generality that $f=0$. The lemma is then obtained by repeating the argument of Lemma~\ref{l.cacciopp.this} and making obvious adjustments to the notation. 
\end{proof}

Folllowing the proof of Lemma~\ref{l.reverse.Holder}, we obtain a global version of the reverse H\"older inequality. 

\begin{lemma}[Reverse H\"older inequality]
\label{l.reverse.Holder.global}
Let $U\subseteq\Rd$ be a bounded Lipschitz domain and denote $V := I_2 \times U$. 
Suppose that $f\in H^1_\pa(V)$, $u\in f+ H^1_{\pa,\sqcup}(V)$ and $u^* \in L^2(I_2;H^{-1}(U))$ satisfy
\begin{equation*} \label{}
\partial_t u - \nabla \cdot \left( \a(x) \nabla u  \right) = u^*  \quad \mbox{in} \ V.
\end{equation*}
Denote $q:=2_* = 2d/(2+d)$ if~$d>2$ or let~$q$ be any element of $\left(\frac54,\tfrac74\right)$ if~$d=2$. 
Then there exists $C(V,d,\Lambda)<\infty$ such that, for every $r\in (0,1)$, $(t,x) \in  I_{1} \times U$ and~$\alpha>0$,
\begin{multline*} \label{}
\left\| \nabla (u-f) \right\|_{\underline{L}^2(Q_r(t,x)\cap V)}^2
\\
\leq \frac{C}{\alpha} \left\| \nabla (u-f) \right\|_{\underline{L}^q(Q_{4r}(t,x)\cap V)}^2 
+ \alpha \left\|  \nabla (u-f) \right\|_{\underline{L}^2(Q_{4r}(t,x) \cap V)}^{2}
\\
+ \alpha \left\|  \nabla f \right\|_{\underline{L}^2(Q_{4r}(t,x) \cap V)}^{2}
+ C \left\| u^* \right\|_{\underline{L}^2\left(t+I_{4r}\cap I_2;\underline{H}^{-1}(B_{4r}(x)\cap U)\right)}^2.
\end{multline*}
\end{lemma}
\begin{proof}
The argument is omitted, since it is an easy adaptation of the proof of Lemma~\ref{l.reverse.Holder}. 
\end{proof}

Proposition~\ref{p.globalmeyers} is now a straightforward consequence of Lemmas~\ref{l.parab.Gehring} and~\ref{l.reverse.Holder.global}.

\subsection*{Acknowledgments} 
SA was partially supported by the NSF Grant DMS-1700329.
JCM was partially supported by the ANR grant LSD (ANR-15-CE40-0020-03).

\small
\bibliographystyle{abbrv}
\bibliography{parab}

\end{document}